%% file: DW.tex
\documentclass{amsart}

\def\titlepaper{Semisimple Field Theories Detect Stable Diffeomorphism}
\def\authorpaper{David Reutter \& Christopher Schommer-Pries}

\include{DWStable-preamble}

\newcommand\ignore[1]{}

\newcommand{\Sp}{\cS}

\title{\titlepaper}

\author{David Reutter}
\address{Universit\"at Hamburg}
\email{david.reutter@uni-hamburg.de}
\urladdr{https://www.davidreutter.com}

\author{Christopher Schommer-Pries}
\address{University of Notre Dame}
\email{cschomme@nd.edu}
\urladdr{https://sites.nd.edu/chris-schommer-pries/}

\begin{document}

\begin{abstract}Extending the work of the first author, we introduce a notion of semisimple topological field theory in arbitrary even dimension and show that such field theories necessarily lead to stable diffeomorphism invariants. The main result of this paper is a proof that this 'upper bound' is optimal: To this end, we introduce and study a class of `finite path integral' topological field theories which are semisimple and which generalize well known theories constructed by Dijkgraaf-Witten, Freed and Quinn. We show that manifolds satisfying a certain finiteness condition --- including $4$-manifolds with finite fundamental group --- are indistinguishable to these field theories if and only if they are stably diffeomorphic. Subject to these finiteness conditions, such finite path integral theories therefore provide the strongest semisimple TFT invariants possible.  These results hold for a large class of ambient tangential structures. 

We discuss a number of applications, including the constructions of unoriented $4$-dimensional semisimple field theories which can distinguish unoriented  smooth structure and oriented higher-dimensional semisimple field theories which can distinguish certain exotic spheres.

Along the way, we show that dimensional reductions of finite path integral theories are again finite path integral theories, we utilize ambidexterity in the rational setting, and we develop techniques related to the $\infty$-categorical M\"obius inversion principle of G\'alvez-Carrillo--Kock--Tonks. 
\end{abstract}

\maketitle

\section{Introduction}

Topological field theories (TFT) provide invariants of smooth manifolds. However what precisely these invariants measure and which manifolds can be distinguished remains largely unknown. In this paper, we focus on \emph{semisimple field theories} in even dimensions (see Definition~\ref{def:semisimpleTFT}), a class of field theories which contains all currently known functorial topological field theories in more than two dimensions (see Example~\ref{exm:TFTs}).

 Extending a result of the first author~\cite{Reutter:2020aa} to arbitrary even dimension, we show that such field theories only depend on the stable diffeomorphism class of a manifold\footnote{Two $2q$-dimensional manifolds $M$ and $N$ are stably diffeomorphic if there exists a natural number $n\geq 0$ and a diffeomorphism $M\#^n(S^q \times S^q) \cong N\#^n(S^q \times S^q)$.}, providing an `upper bound' on the sensitivity of the induced manifold invariant. 
 
\begin{maintheorem} Stably diffeomorphic even-dimensional manifolds are indistinguishable by semisimple topological field theories. \end{maintheorem}
The precise version of this theorem, Theorem~\ref{thm:semisimpletostablediffoinv}, allows for more general tangential structures; in this introduction we focus on the oriented and unoriented case. The main contribution of this paper is a proof that this `upper bound' is optimal for manifolds fulfilling a certain finiteness condition.

\begin{maintheorem} Connected closed $2q$-manifolds with $\pi$-finite tangential $(q-1)$-type\footnote{A $2q$-manifold has $\pi$-finite tangential $(q-1)$-type if $\pi_iF$ is finite for all $i \leq q-1$ and all basepoints, where $F$ is the homotopy fiber of the classifying map $M \to \rB O(2q)$ of the tangent bundle.} are stably diffeomorphic if and only if they are indistinguishable by semisimple field theories.
\end{maintheorem}
This theorem appears as Theorem~\ref{thm:MainThm}. We will discuss a number of applications below, such as the existence of semisimple field theories which can distinguish smooth structure.

To prove Theorem B, we introduce and study a class of \emph{super topological field theories} generalizing well known constructions of Dijkgraaf-Witten, Freed and Quinn~\cite{MR1048699,MR1240583,MR1256993,MR1338394}. Associated to each such \emph{finite path integral theory} is a certain numerical dimension, called its \emph{type} and defined below at the end of Section~\ref{sec:introDW}. We show that finite path integral theories are semisimple (Theorem~\ref{thm:DWareSemisimple}) and that type-$(q-1)$ finite path integral theories distinguish \emph{non} stably diffeomorphic $2q$-manifolds, provided they have $\pi$-finite $(q-1)$-types (Theorem~\ref{thm:DWareSemisimple}), proving Theorem B.

\subsection{Manifold results} The four-dimensional oriented version of Theorem A appears as \cite[Theorem A]{Reutter:2020aa}, and generalizes earlier results on the sensitivity of positive TFTs \cite{MR2209373}. In this dimension, Theorem B reduces to the following converse of Theorem A: 

\begin{theorem}\label{thm:FirstIntroThm}
	Let $M$ and $N$ be connected, closed 4-manifolds with finite fundamental groups. Then $M$ and $N$ are stably diffeomorphic if and only if they are indistinguishable by semisimple topological field theories if and only if they are indistinguishable by type-1 finite path integral theories. 
\end{theorem}
As explained in~\cite{Reutter:2020aa}, Theorem A implies that \emph{oriented} semisimple 4-dimensional TFTs cannot detect smooth structure: It follows from a result of Gompf~\cite{MR769285} that smooth oriented $4$-manifolds which are orientation preserving homeomorphic are stably diffeomorphic. However, Theorem~\ref{thm:FirstIntroThm} shows that semisimple oriented 4-dimensional TFTs can see more than the homotopy type of a manifold --- in~\cite[Ex.~5.2.4]{TeichnerPHD}, Teichner constructed families of homotopy equivalent oriented $4$-manifolds with finite fundamental group which are not stably diffeomorphic and hence can be distinguished by oriented semisimple field theories, answering Question 1.1 of~\cite{Reutter:2020aa}.

In contrast, Theorem~\ref{thm:FirstIntroThm} also shows that semisimple $4$-dimensional TFTs can sometimes distinguish  \emph{unoriented} smooth structure: There are examples of unoriented smooth $4$-manifolds which are homeomorphic but not stably diffeomorphic, such as $\RR\PP^4$ and Cappell-Shaneson's fake $\RR\PP^4$. Provided they have finite fundamental groups, such manifolds can therefore be distinguished by semisimple $4-$dimensional field theories, recovering and extending an example of A. Debray~\cite{362517} (a version of which also appears in~\cite{2104.14567}).

In higher dimensions, the shortcomings of semisimple field theories remain significant: For example, as observed by Kreck and Schafer \cite{MR743942}, arbitrarily large families of $(2k-1)$-connected $4k$-manifolds ($k\geq 2$) which are pairwise stably diffeomorphic, but pairwise not homotopy equivalent have been implicit in the literature since Wall's classification of these manifolds up to the action of homotopy spheres \cite{MR145540}. Many more examples have appeared more recently \cite{CCPS21-1, CCPS21-2}, including infinite families of pairwise stably diffeomorphic manifolds which are pairwise not homotopy equivalent. By Theorem~A, all such manifolds are indistinguishable by semisimple TFT. 

On the other hand, we may use Theorem~B to construct semisimple field theories which can distinguish certain exotic smooth spheres (see Corollary~\ref{cor:hitchinspheredetection}).

\begin{proposition}
	In every dimension $d = 8k +1$ or $8k+2$ there exists pairs of distinct exotic spheres which are distinguished by certain oriented type-1 finite path integral theories.
\end{proposition}

In fact, in dimensions $> 4$, (exotic) spheres are diffeomorphic if and only if they are stably diffeomorphic \cite{385015}. Unfortunately, in the relevant dimensions $2q > 6$ homotopy spheres do not have $\pi$-finite tangential $(q-1)$-type, so that Theorem B does not apply. This raises the following question:
\begin{question} Can the finiteness assumption in Theorem B be weakened? Do finite path integral theories, or more generally semisimple TFTs, distinguish all exotic spheres in dimensions $>4$? 
\end{question}
The finiteness assumption in Theorem~B cannot be completely omitted: In Proposition~\ref{prop:Thompson} we provide an example of a $4$-manifold with infinite fundamental group, but which nevertheless cannot be distinguished from $S^4$ by type-1 finite path integral theories.

In practice, we will use the following alternative characterization of $2q$-manifolds having $\pi$-finite $(q-1)$-type: For $4k \leq q$ let $p_{k,M}^\QQ: \pi_{4k}M \to \QQ$ denote the composite 
\begin{equation*}
	\pi_{4k}M \stackrel{\pi_{4k}(T_M)}{\to} \pi_{4k} \rB O(2q) \to H_{4k}(BO(2q); \QQ) \stackrel{p_{4k}^\QQ}{\to} \QQ
\end{equation*}
induced by the classifying map $T_M:M \to \rB O(2q)$ of the tangent bundle, the rational Hurewicz homomorphism, and the rational $k^\text{th}$ Pontryagin class. Then, a connected $2q$-manifold $M$ has $\pi$-finite tangential $(q-1)$-type if and only if for each $i \leq q-1$ with $i \neq 4k$ the group $\pi_{i}M$ is finite, for each $4k \leq q$ the homomorphism $p_{k,M}^\QQ$ is non-zero, and for each $4k \leq q-1$ the kernel of $p_{k,M}^\QQ$ is finite. The $6$ and $8$-dimensional analogues of Theorem~\ref{thm:FirstIntroThm} therefore become:
\begin{corollary} Connected closed $6$-manifolds with finite $\pi_1$ and $\pi_2$ are stably diffeomorphic if and only if they are indistinguishable by semisimple topological field theories, if and only if they are indistinguishable by type-$2$ finite path integral theories. 

Connected closed $8$-manifolds with finite $\pi_1$, $\pi_2$, and $\pi_3$, and non-zero signature are
	stably diffeomorphic if and only if they are indistinguishable by semisimple topological field theories if and only if they are indistinguishable by type-$3$ finite path integral theories.
\end{corollary}
In dimension $2q$ with $q$ even, stable diffeomorphism can fail to distinguish homotopically distinct manifolds. In contrast, when $q$ is odd, simply connected manifolds which are stably diffeomorphic are themselves diffeomorphic \cite[Thm~D]{MR1709301}. This yields the following corollaries, the first of which partially answers a question raised in \cite{MR2417448}.
\begin{corollary}
	 Simply connected $6$-manifolds with finite $\pi_2$ are distinguished up to diffeomorphism by 6-dimensional semisimple TFTs; specifically by type-2 finite path integral theories.
\end{corollary}

\begin{corollary}
	Simply connected 10-manifolds with finite $\pi_2$ and $\pi_3$, $\pi_4 \otimes \QQ$ rank one, and non-zero rational 1st Pontryagin class  are distinguished up to diffeomorphism by 10-dimensional semisimple TFTs; specifically by type-4 finite path integral theories.
\end{corollary}

Further examples and applications are discussed in Section~\ref{sec:exmappl}.

\subsection{Finite path integral theories}\label{sec:introDW}
In~\cite{MR1048699}, R. Dijkgraaf and E. Witten construct a gauge theory for a finite gauge group $G$, a construction which was later improved, streamlined and expressed as an oriented functorial TFT  by D. Freed and F. Quinn in~\cite{MR1256993, MR1240583}.  Such Dijkgraaf-Witten TFT were thereafter considered by many authors~\cite{MR2648901,Trova:2016aa,MR3825014,MR3862670,MR4021927,  MR4133704}.

Based on ideas of Kontsevich~\cite{Kontsevich}, F. Quinn defined in~\cite{MR1338394} oriented \emph{twisted finite homotopy TQFTs}, a generalization of Dijkgraaf-Witten theories in which the classifying space $\rB G$ of the group $G$ is replaced by any $\pi$-finite\footnote{A space $Y$ is $\pi$-finite if $\pi_0 Y$ is finite and if for each base point $\pi_n Y$ is finite for all $n$ and non-zero for only finitely many $n$.} space $Y$ and a cohomology class $\omega \in H^d(Y; \CC^\times)$. In physical terms, Quinn's version of Dijkgraaf-Witten theory may be understood as a `finite $\sigma$-model'. The `space of fields' on a closed $d$-manifold $M$ is the mapping space $\Map(M, Y)$, and the partition function on $M$ is computed via a `finite path integral' involving $\omega\in H^d(Y; \CC^\times)$ as a Lagrangian of the theory:
\begin{equation}\label{eq:QuinnDW}
	Z_{Y, \omega}(M) = \sum_{[f] \in \pi_0 \Map(M,Y)} \#(\Map(M,Y), f) \cdot \langle [M], f^*\omega \rangle .
\end{equation}
Here, $\langle [M], -\rangle: H^d(M, \mathbb{C}^{\times}) \to \mathbb{C}^{\times}$ denotes the evaluation of a top-dimensional cohomology class against the fundamental class of $M$ and $\#(Z,z)$ denotes the \emph{homotopy cardinality of $Z$ at $z$} (cf.~\cite[Lecture 4]{MR1338394}) defined for $\pi$-finite spaces as
\begin{equation}\#(Z,z):= \prod_{i \geq 1} | \pi_i(Z, z)|^{(-1)^i} = |\pi_1(Z,z)|^{-1} |\pi_2(Z,z)| |\pi_3(Z,z)|^{-1} \cdots.
\end{equation}
If $Y = \rB G$ is the classifying space of a finite group $G$, the space $\Map(M, Y)$ is the space of principal $G$-bundles on $M$ and~\eqref{eq:QuinnDW} recovers the classical Dijkgraaf-Witten partition function. 

Rather than summing over spaces of maps $\Map(M, Y)$, we generalize this theory to allow for summations over more general tangential structures on $M$ --- for example, starting from a spin theory one may sum over spin structures to produce an oriented field theory.
Such \emph{finite path integral theories} can be used to produce topological field theories of arbitrary dimension $d$ and for arbitrary tangential structure, however for expository purposes we will focus on the oriented case in this introduction.
In this case, one starts with a space $X$ with a map $\xi: X \to \rB SO(d)$ whose homotopy fibers are $\pi$-finite. The space of fields will be the space of $(X, \xi)$-structures on an oriented manifold $M$, i.e. the space $\Map_{BSO(d)}(M, X)$ of maps to $X$ over $BSO(d)$ (in the appropriate homotopical sense), where $M$ is viewed as a space over $BSO(d)$ via its tangent classifying map $T_M: M \to BSO(d)$. For example, if $X= Y \times \rB SO(d)$ the space of fields is $\Map(M, Y)$ as in Quinn's field theory, whereas for $X= \rB Spin(d)$ the space of fields is the space of spin structures on $M$ which are compatible with the given orientation. Since $X$ is $\pi$-finite, the space $\Map_{BSO(d)}(M, X)$ will itself be $\pi$-finite.

The second input to our construction is a $d$-dimensional $(X, \xi)$-structured field theory $\cW$. In Section~\ref{sec:DWsTFTs}, we produce from this data $(X,\xi, \cW)$ an oriented $d$-dimensional field theory with  partition function  on a closed oriented $d$-manifold $M$ given by 
\begin{equation*}
	FP_{\xi, \cW}(M) =  \sum_{[f] \in \pi_0 \Map_{\rB SO(d)}(M,X)} \#(\Map_{\rB SO(d)}(M,X), f) \cdot \cW(M, f).
\end{equation*}
This field theory may be thought of as arising from \emph{gauging} (or `integrating') the $(X,\xi)$-structured field theory $\cW$ along $X\to \rB SO(d)$ to an oriented field theory.

For our purposes, we will mainly restrict attention to \emph{invertible} super vector space valued $(X,\xi)$-theories $\cW$. Such theories are classified by their partition function: For each tangential structure $(X, \xi)$, there is an associated finitely generated abelian group $\Omega_d^{T\xi}$ which may be thought of as an `unstable tangential' version of more familiar bordism groups. In special cases, this group  $\Omega_d^{T\xi}$ has been called the SKK-group \cite{MR0362360} or the Reinhart vector field cobordism group \cite{MR153021}. Closed $d$-manifolds with $(X, \xi)$-structures give elements in $\Omega_d^{T\xi}$, and
 invertible super field theories are uniquely determined by the induced group homomorphism (see Theorem~\ref{thm:invertarechars})
\begin{equation*}
	\omega: \Omega_d^{T\xi} \to \CC^\times. 
\end{equation*}
The finite path integral theory constructed from $(X, \xi)$ and an invertible theory determined by $\omega: \Omega_d^{T\xi} \to \CC^\times$ is an oriented super topological field theory $\FP_{\xi, \omega}$ with partition function 
\begin{equation}\label{eq:generalizedDW}
	\FP_{\xi, \omega}(M) =  \sum_{[f] \in \pi_0 \Map_{\rB SO(d)}(M,X)} \#(\Map_{\rB SO(d)}(M,X), f) \cdot \omega(M, f).
\end{equation}
We say that a finite path integral theory $\FP_{\xi, \omega}$ is of \emph{type-n} if the homotopy fiber $F$ of $\xi: X \to \rB SO(d)$ is an $n$-type, i.e. if $\pi_{\geq n+1} F =0$. 

A cohomology class $\omega\in H^d(Y; \CC^\times)$ gives rise to an invertible theory for $X= Y \times \rB SO(d)$ and the resulting oriented TFT recovers Quinn's $Z_{Y, \omega}$ (see Example~\ref{ex:DWQuinnasGenDW}).

\subsection{What is visible to finite path integral theories}

The manifold invariants~\eqref{eq:generalizedDW} arise from averaging $(X,\xi)$-structured bordism invariants over all $(X,\xi)$-structures on a given manifold $M$, weighted by homotopy cardinality. To study the `lossiness' of this averaging procedure more systematically, we consider the following categorical generalization.

Let $\cC$ be an $\infty$-category enriched in $\pi$-finite spaces and let $\Omega:\cC \to \mathrm{Ab}$ be a functor into the category of abelian groups. From this data, we obtain a linear pairing

\begin{equation} \langle -,- \rangle_{\cC, \Omega}: \bigoplus_{[c] \in \pi_0 \cC} \mathbb{C}[\widehat{\Omega(c)}/{\pi_0 \Aut_{\cC}(c)}] \otimes\bigoplus_{[d] \in \pi_0 \cC} \mathbb{C}[\Omega(d)/\pi_0 \Aut_{\cC}(d)] \to \mathbb{C}
\end{equation}
where the direct sum is taken over the isomorphism classes of objects of $\cC$, and where for an abelian group $A$ with an action by a group $G$, $\widehat{A}:= \Hom(A, \mathbb{C}^{\times})$ denotes the group of characters and $A/G$ denotes the set of orbits of the $G$-action.
The pairing is characterized by the following property: For each $c,d \in \cC$, the induced pairing
\[ \widehat{\Omega(c)} \times \Omega(d) \to \mathbb{C}
\]
 is given by 
\[(\phi, v) \mapsto \sum_{f\in \pi_0 \cC(d,c)} \#(\cC(d,c), f)~\phi(\Omega(f)(v)).
\]
Note that this formula indeed only depends on the orbits of $\phi$ and $v$ under the $\Aut_{\cC}(c)$ and $\Aut_{\cC}(d)$ action, respectively, and the isomorphism classes of $c$ and $d$.

\begin{example} Let $\cC$ be an ordinary $1$-category with finite hom sets and let $\Omega$ be the constant functor with value the trivial group. Then, the pairing becomes a matrix indexed by the isomorphism classes of objects of $\cC$ with $(x,y)$-coefficient given by the cardinality $|\cC(y,x)|$.\end{example}

\begin{example}\label{exm:introbordismpairing}Let $n,d\geq 0$ and consider the $\infty$-category $\cC$ of spaces $X$ equipped with maps $\xi: X\to \rB SO(d)$ whose homotopy fibers $F$ are $\pi$-finite $n$-types (i.e. $F$ is $\pi$-finite and $\pi_{\geq n+1} F =0$). Let $\Omega$ be the functor $\Omega(X,\xi) = \Omega_d^{T\xi}$. If $M$ is a closed $d$-manifold with $\pi$-finite tangential $n$-type, then the (oriented) finite path integral invariant associated to $(X,\xi)$ and $\omega: \Omega_d^{X,\xi} \to \mathbb{C}^\times$ can be expressed in terms of the pairing $\langle-,-\rangle_{\cC, \Omega}$:
\[\FP\FP_{\xi, \omega} (M) = \langle ((X,\xi), \omega), (\tau_{\leq n} M, [M]) \rangle_{\cC, \Omega}.
\] 
Here, $\tau_{\leq n} M \to \rB SO(d) $ is the tangential $n$-type obtained from the oriented manifold $M$ by factoring the classifying map of the tangent bundle $$ M \to \tau_{\leq n}M \to\rB SO(d)$$ 
into an $n$-connected map followed by an $n$-truncated map\footnote{\label{ftn:intro}A map $f:X \to Y$ of spaces is \emph{$n$-connected} if for all points $x\in X$, $\pi_{i}(f;x)$ is an isomorphism for $i < n+1$ and surjective for $i=n+1$, or equivalently if all homotopy fibers $F$ of $f$ are $n$-connected, i.e. $\pi_{\leq n}F =0$. A map $f:X\to Y$ is \emph{$n$-truncated} if $\pi_i (f;x)$ is injective for $i=n+1$ and an isomorphism for $i >n+1$, or equivalently if all homotopy fibers $F$ are $n$-types, i.e. $\pi_{\geq (n+1)}F = 0$.} and \[[M] \in \Omega_d^{\tau_{\leq n}(M)}/\pi_0\Aut\left(\vphantom{\frac{a}{b}}\tau_{\leq n} M \to\rB SO(d)\right)\] is the induced orbit in the bordism group.
\end{example}
Our proof of Theorem B relies on a general criterion for when pairings $\langle-, -\rangle_{\cC, \Omega}$ are  non-degenerate, applied to Example~\ref{exm:introbordismpairing}. This question of non-degeneracy of $\langle-, -\rangle_{\cC, \Omega}$ is closely related to the  M\"obius inversion principle\footnote{Roughly speaking, the M\"obius inversion principle states that the zeta function of any incidence algebra of a M\"obius $\infty$-category is invertible for the convolution product. Our sufficient condition in Theorem~\ref{thm:introtechnical} does not require $\cC$ to be a M\"obius category, and hence does neither imply nor is implied by the results in~\cite{MR3818099}.} for decomposition spaces (also known as $2$-Segal spaces~\cite{MR3970975}) developed in~\cite{MR3804694,MR3818099,MR3828744}. 

In the case where $\cC$ is a $1$-category and $\Omega$ is the trivial functor, non-degeneracy of $\langle-,-\rangle_{\cC, \Omega}$ is equivalent to the question of whether each object $x$ is determined up to isomorphism by the sizes of the hom sets $\cC(x,y)$ for all $y$, and vice versa. This question was studied by Lovasz \cite{MR214529, MR284391} and answered positively for many categories of finite algebraic structure including various categories of finite graphs and the category of finite groups. An important consequence is for example that for any three finite groups (or finite graphs) $G$, $H$, and $K$ if $G \times K \cong H \times K$, then $G \cong H$.

Lovasz's technique is, roughly, to use a factorization system on $\cC$ to write the pairing in simpler terms. For example, if $\cC$ is the category of finite sets,  the set of isomorphism classes of objects of $\cC$ can be identified with $\mathbb{N}$ and the pairing of finite sets  $n$ and $m$ of cardinality $n$ and $m$, respectively,  is given by  $\langle n, m \rangle = |\Hom(m, n)| =  n^m$. Factoring a function $m \to n$ as a surjection followed by an injection, it follows that \[|\Hom(m, n)| = \sum_{a\in \mathbb{N}} \frac{|\mathrm{Surj}(m, a)| ~|\mathrm{Inj}(a, n)|}{|\Aut(a)|}.\] As both $|\mathrm{Surj}(-,-)|$ and $|\mathrm{Inj}(-,-)|$ are given by (infinite) triangular matrices with non-zero diagonal entries, and hence lead to non-degenerate pairings, it follows that the pairing $|\Hom(-,-)|$ itself is non-degenerate.

We generalize this method to categories $\cC$ enriched in $\pi$-finite spaces which are equipped with a nested sequence of factorizations systems, and prove the following theorem (see Theorem~\ref{thm:linearpairing} for details).

\begin{theorem}\label{thm:introtechnical}
	Let $\cC$ be a $\infty$-category enriched in $\pi$-finite spaces. For $k=0, \dots, n$ let $(\cL^{(k)}, \cR^{(k)})$ be a collection of orthogonal factorization systems on $\cC$, which are nested in the sense that $\cR^{(k-1)} \subseteq \cR^{(k)}$. Suppose that each of the subcategories 
	\begin{equation*}
		\cR^{(0)}, \quad \cR^{(1)} \cap \cL^{(0)}, \quad \cR^{(2)} \cap \cL^{(1)}, \quad\dots \quad\cR^{(n)} \cap \cL^{(n-1)}, \quad\cL^{(n)}
	\end{equation*}
satisfies the condition that all endomorphisms are equivalences. Then for every functor $\Omega: \cC \to \mathrm{Ab}$ the corresponding 
pairing $\langle-,-\rangle_{\cC, \Omega}$ is non-degenerate.\end{theorem}

The Postnikov factorization system into ($k$-connected/$k$-truncated) maps for $-1\leq k \leq n$ (see footnote~\ref{ftn:intro}) endows the $\infty$-category $\cC$ from Example~\ref{exm:introbordismpairing} with such a nested sequence of factorization systems (see Definition~\ref{def:nestedfactsystem}).
Hence, applied to Example~\ref{exm:introbordismpairing}, Theorem~\ref{thm:introtechnical} leads to the following precise characterization of the manifold topology visible to finite path integral theories.

\begin{theorem}\label{thm:Intro_DWdistinguish}
	Two $d$-manifolds $M$ and $N$ with $\pi$-finite tangential $n$-types $\tau_{\leq n}M$ and $\tau_{\leq n}N$ are indistinguishable by type-$n$ oriented finite path integral theories if and only if they have equivalent tangential $n$-types $\tau_{\leq n}M \simeq (X, \xi) \simeq \tau_{\leq n}N$, and the bordism classes $[M]$ and $[N]$ lie in the same orbit in $\Omega^{T\xi}_d$ under the action of $\pi_0 \Aut(X, \xi)$.
\end{theorem} 
A precise version of this theorem for general ambient tangential structure appears as Theorem~\ref{thm:MainThm}.
Theorem~\ref{thm:Intro_DWdistinguish} should be compared to Kreck's theorem~\cite{MR1709301} which implies (see Section~\ref{sec:stablediffeoandKrecksthm}) that the stable diffeomorphism class of a $2q$-manifold $M$ is precisely determined by the data appearing in Theorem~\ref{thm:Intro_DWdistinguish}: namely, its tangential $(q-1)$-type $(X, \xi)$ and the orbit of its bordism class $[M]$ in $\Omega_d^{T\xi}$ under the action of $\pi_0 \Aut(X,\xi)$. Hence, combined with Kreck's result, Theorem~\ref{thm:Intro_DWdistinguish} implies Theorem~B. In other words, finite path integral theories are in a sense `Pontryagin dual' to stable diffeomorphism classes.

\begin{remark}
	An interesting consequence of Theorems~\ref{thm:Intro_DWdistinguish}, A, and B is that if two manifolds with $\pi$-finite tangential $n$-type and even dimension $d = 2q$ are distinguished by some semisimple field theories, then in fact they can be distinguished by finite path integral theories of type $(q-1)$, one less than the middle dimension. For example, using finite path integral theories of higher type does not result in stronger invariants. 
\end{remark}

On the other hand, non-degeneracy of our pairing also implies that in high enough dimensions $d$, finite path integral theories are distinguished by their partition functions.

\begin{theorem}
	Suppose that $n< \lfloor \frac{d}{2} \rfloor$. If $((Y_1, \xi_1),\omega_1)$ and $((Y_2, \xi_2), \omega_2)$ give rise to type-$n$ finite path integral theories whose partition functions are identical on closed $d$-manifolds with $\pi$-finite tangential $n$-type, then $\xi_1 \simeq \xi_2 = \xi$, and $\omega_1$ and $\omega_2$ are in the same orbit of $\Hom(\Omega^{T\xi}_d, \mathbb{C}^{\times})$ under the action of $\pi_0 \Aut(\xi)$. The resulting finite path integral theories are consequently isomorphic. 
\end{theorem}
This theorem appears for more general tangential structures as Theorem~\ref{thm:partitionfunctiondetected}.

\subsection{Organization of the paper}

Throughout this paper we freely use the language of $\infty$-categories; in particular, all colimits and limits are to be considered homotopical. 

The structure of this paper can be roughly summarized as the reverse order of the introduction. 

In Section~\ref{sec:SpansAndLinearization}, we start by reviewing homotopy cardinality and finite path integrals. This transitions to a discussion of spans of topological spaces equipped with local systems. Spans of $\pi$-finite spaces can be linearized/integrated to linear maps, as discussed in Section~\ref{sec:linearizingspans}. Any $\infty$-category naturally gives rise to a certain span, and in Section~\ref{sec:linearizingfinitecats} we give criteria for when its linearization exists and is invertible. For $\infty$-categories enriched in $\pi$-finite spaces but with infinitely many isomorphism classes of objects, this linearization does not exist, but we may still construct a certain pairing. This is discussed in Sections~\ref{sec:linearizelocpifinite} and~\ref{sec:pontryaginpairing} along with conditions ensuring that this pairing is non-degenerate.

In Section~\ref{sec:bordandinvertTFT}, we discuss a variety of tangential structures, including structures on the stable normal bundle, structures on the stable tangent bundle, and on the unstable tangent bundle. Bordism groups for these are defined, as well as the bordism category. The classification of invertible topological field theories is discussed in Section~\ref{sec:invertTFT}. These are determined by characters on the homotopy groups of certain Madsen-Tillmann spectra. Section~\ref{sec:compntypes} contains some computations that relate these bordism groups to more familiar classical bordism groups.  

In Section~\ref{sec:DWtheories}, we construct finite path integral topological field theories. This utilizes the linearization of spans developed in Section~\ref{sec:SpansAndLinearization}. We show that the resulting manifold invariants can be expressed in terms of the Pontryagin pairings constructed in Section~\ref{sec:SpansAndLinearization}. Our main Theorem~\ref{thm:Intro_DWdistinguish}  is proven in Section~\ref{sec:mainthm} which classifies the information that finite path integral theories can detect about manifolds.
In Section~\ref{sec:dimred},   we also consider the process of dimensional reduction in the context of finite path integral theories. We show that the dimensional reduction of a finite path integral theory is again a finite path integral theory.

In Section~\ref{sec:semisimple},  we first generalize the results of \cite{Reutter:2020aa} to higher dimensions, super field theories, and more general tangential structures.  We define semisimple topological field theories and establish that they induce stable diffeomorphism invariants. Finally, we make use of the dimensional reduction results of the previous section to show that finite path integral theories are themselves semisimple topological field theories. 

In Section~\ref{sec:stablediffeomandmain}, we recall Kreck's classification of manifolds up to stable diffeomorphism~\cite{MR1709301} and relate it to our main Theorem~\ref{thm:Intro_DWdistinguish} concerning finite path integral theories (c.f. Section~\ref{sec:DWtheories}). From this we are able to deduce Theorem B and end with a discussion of various examples and applications.  

Appendix~\ref{sec:nfindom} includes a technical discussion of finitely $n$-dominated spaces which is used to give conditions under which certain mapping spaces will be $\pi$-finite.

\subsection{Acknowledgements}

We are grateful to Matthias Kreck, Theo Johnson-Freyd, Noah Snyder, Stephan Stolz and  Peter Teichner for their continued interest in this work and many helpful discussions. We would like to thank Arun Debray for sparking our interest in the question of what finite path integral theories can detect. We also thank Dan Freed and Lukas M\"uller for helpful comments on an early draft of this manuscript. This material is based upon work supported by the National Science Foundation under Grant No. DMS-1440140, while the authors were in residence at the Mathematical Sciences Research Institute in Berkeley, California, during the Spring 2020 semester. D.R. is funded by the Deutsche Forschungsgemeinschaft (DFG, German Research Foundation) – 493608176 and is grateful for the hospitality and financial support of the Max-Planck Institute for Mathematics in Bonn. C.SP. is supported by the National Science Foundation under Grant No. DMS-2204297.

\tableofcontents

\section{Spans and linearization}\label{sec:SpansAndLinearization}

\subsection{Homotopy cardinality and finite path integrals}

Recall that a space $X$ is \emph{$\pi$-finite} if it has finitely many components, and for each choice of basepoint $x \in X$, $\pi_i(X,x)$ is finite and non-trivial for only finitely many $i$. 
For a $\pi$-finite space
 and a base point $x\in X$, define the \emph{homotopy cardinality of $X$ at $x$} to be (cf.~\cite[Lecture 4]{MR1338394})
\begin{equation}\#(X,x):= \prod_{i \geq 1} | \pi_i(X, x)|^{(-1)^i} = |\pi_1(X,x)|^{-1} |\pi_2(X,x)| |\pi_3(X,x)|^{-1} \cdots \in \mathbb{Q} .
\end{equation}Evidently, this rational number only depends on the component $[x] \in \pi_0 X$ of $x$. 
Define the \emph{total homotopy cardinality} of $X$ to be $\#^\textrm{tot}X:= \sum_{x \in \pi_0 X} \#(X,x)$. A key property of total homotopy cardinality is its compatibility with fiber sequences: If $X\to B$ is a map of spaces with connected $\pi$-finite base $B$ and $\pi$-finite fiber $F$, it follows from the associated long exact sequence of homotopy groups that $X$ is also $\pi$-finite and that 
\begin{equation}\label{eq:totalhomotopyproduct} \#^\textrm{tot}X =( \#^{\mathrm{tot}} B) (\#^\textrm{tot}F).
\end{equation}

Homotopy cardinality can be used to develop a theory of `finite path integrals'. 
\begin{definition}\label{def:finitepathintegral}
If $V$ is a rational vector space, $X$ is a $\pi$-finite space and $\alpha: \pi_0 X\to V$ is a function, the \emph{integral} of $\alpha$ over $X$ is the vector
\begin{equation}\label{eq:integralnotation}\int_X \alpha := \sum_{x \in \pi_0 X} \#(X,x ) ~\alpha(x) \in V.
\end{equation}
 We will also utilize the notation $\int_{x \in X} \alpha(x)$, expressing $\alpha$ as a formula in $\pi_0X$. 
\end{definition}

Most importantly, and most evidently, this integral is linear. Given a linear function $\phi:V \to W$ and a function $\alpha: \pi_0 X \to V$ where $X$ is $\pi$-finite, the following elements of $W$ agree:
\begin{equation}\label{eq:integrallinear}
 \phi(\int_X \alpha) = \int_X \phi \circ \alpha 
 \end{equation}

Equation~\eqref{eq:totalhomotopyproduct} is an instance of the following `generalized Fubini theorem'.
\begin{prop}[Fubini theorem]\label{prop:Fubini}
Let $s:X\to A $ be a map of $\pi$-finite spaces and let $\alpha: \pi_0 X \to V$ be a function to a rational vector space $V$. Then, the fiber $\iota_a:X_a \to X$ of $X\to A$ at a point $a\in A$ is $\pi$-finite, and the following elements of $V$ agree:
\[\int_{x\in X} \alpha(x) = \int_{a \in A} \int_{f \in X_a} \alpha(\iota_a f)
\]
\end{prop}
\begin{proof}
It suffices to show that for every $[x] \in \pi_0 X$, the following equation holds in $\QQ$: 
\begin{equation}\label{eq:proofFubini}
\#(X,x) = \sum_{[a] \in \pi_0 A} \#(A,a) \sum_{[f] \in  (\pi_0\iota_a)^{-1}([x]) \subseteq \pi_0(X_a)} \#(X_a, f)
\end{equation}
Proposition~\ref{prop:Fubini} follows from summing the product of this rational number~\eqref{eq:proofFubini} and the vector $\alpha([x]) \in V$ over $[x] \in \pi_0 X$. 
To show~\eqref{eq:proofFubini}, note that for any $[a] \neq \pi_0 s ([x])$ the preimage $(\pi_0 \iota_a)^{-1}([x])$ is empty, and hence the summand corresponding to such an $[a] \in \pi_0 A$ is zero. Therefore, equation~\eqref{eq:proofFubini} is equivalent to \begin{equation}\label{eq:proofFubini2}
\#(X, x) = 
 \#(A, s([x])) \sum_{[f] \in (\pi_0 \iota_{s([x])})^{-1}([x])} \#(X_{s([x])}, f).
\end{equation}
Let $X|_{[x]}$ denote the connected component of $[x]\in \pi_0X$ and note that $s|_{[x]}: X|_{[x]} \to A|_{s([x])}$ has  fiber $\sqcup_{[f] \in (\pi_0\iota_{s([x])})^{-1}([x]) \subseteq \pi_0(X_{s([x])})} X_{s([x])}|_{[f]}$. Equation~\eqref{eq:proofFubini2} follows from applying~\eqref{eq:totalhomotopyproduct} to $s|_{[x]}$ and noting that the homotopy cardinality of a disjoint union of spaces is the sum of their respective homotopy cardinalities.
\end{proof}
This generalized Fubini theorem indeed implies the more familiar statement that for a function $\alpha: \pi_0 (X\times Y) \to V$ out of a product of $\pi$-finite spaces
\[\int_{(a,b) \in A \times B} \alpha(a,b) = \int_{a\in A} \int_{b \in B} \alpha(a,b).
\]
This follows from applying Proposition~\ref{prop:Fubini} to the case where $s:X\to A$ is the projection $A \times B \to A$.

Another important property of this integral is its compatibility with group actions.
\begin{lemma}\label{lem:equivariance} Let $X$ be a $\pi$-finite space and let $G$ be a (discrete) group acting on $\pi_0 X$ so that for all $[x] \in \pi_0 X$ and $g\in G$ the connected components $[x]$ and $g[x]$ of $X$ are homotopy equivalent spaces. Then, for a rational vector space $V$ with a $G$ action and a $G$-equivariant function $\alpha: \pi_0 X \to V$, the integral $\int_{X} \alpha \in V$ is a $G$-fixed point in $V$. \end{lemma}
\begin{proof} By assumption, $\#(X, [x])= \#(X, g[x])$ for all components $[x] \in \pi_0 X$ and $g\in G$. The lemma then follows from reindexing the sum in Definition~\ref{def:finitepathintegral}. 
\end{proof}
Of course, the main application of Lemma~\ref{lem:equivariance} is when there is a topological (or homotopy coherent) group $\GG$ with $\pi_0 \GG = G$ and when the action of $G$ on $\pi_0 X$ is induced by an action of  $\GG$ on the space $X$. In this case, the assumption that all components in a given $G$-orbit are homotopy equivalent is automatically satisfied.

In the next section, we unify these and further properties into Corollary~\ref{cor:spanfunctor}.

\subsection{Linearizing spans of spaces}\label{sec:linearizingspans}
For later applications to super field theories, we now generalize finite path integrals from rational vector spaces to objects in more general categories (such as the category of super vector spaces). For the remainder of this section, we therefore let $\cV$ be a $1$-category which is enriched in rational vector spaces and which has small colimits.  
Most of the ideas and observations in this Section~\ref{sec:linearizingspans} have appeared before~\cite{MR4133704} and are only included here for the sake of completeness.

A ($\cV$-valued) \emph{local system} on a space $A$ is a functor $\cL:A \to \cV$ and a map of local systems $\alpha: \cL \to \cP$ is a natural transformation, where $A$ is viewed as an $\infty$-groupoid. As $\cV$ is an ordinary $1$-category, a local system is equivalent to a functor $\cL:\pi_{\leq 1}A \to \cV$ out of the fundamental $1$-groupoid $\pi_{\leq 1}A$ of $A$. The \emph{$\cV$-object of sections} of $\cL$ is, if it exists, the limit of this functor $\cL:\pi_{\leq 1}A \to \cV$, the \emph{$\cV$-object of co-sections} is the colimit.
Choosing a basepoint $a\in A$ in every component $[a] \in \pi_0 A$, a local system therefore explicitly amounts to an object $\cL(a) \in \cV$ for every  $[a] \in \pi_0 A$ equipped with an action of $\pi_1(A, a)$. A map of local systems $\cL \Rightarrow \cP$ amounts to a choice of $\pi_1(A, a)$ intertwiner $\alpha_a: \cL(a) \to \cP(a)$ for every $[a] \in \pi_0 A$. These choices of basepoints identify the $\cV$-objects of sections of $\cL$ with the invariants $ \lim \cL \cong \prod_{[a] \in \pi_0 A} \cL_A(a)^{\pi_1(A,a)}$ and the $\cV$-object of co-sections with the coinvariants $ \colim \cL \cong \coprod_{[a] \in \pi_0 A} \cL_A(a)_{\pi_1(A,a)}$.

\begin{definition}\label{def:decoratedspan} Let $(A, \cL_A)$ and $(B, \cL_B)$ be spaces equipped with local systems. A \emph{decorated span} $(X, \alpha)$ is a span of spaces  $B \stackrel{t}{\leftarrow} X \stackrel{s}{\rightarrow} A$ equipped with a map of local systems $\alpha :s^*\cL_A \to t^* \cL_B$. 

Two decorated spans $(X, \alpha)$ and $(Y, \beta)$  between spaces equipped with local systems $(A, \cL_A)$ and $(B, \cL_B)$ are \emph{isomorphic} if there is a homotopy equivalence $f:X \to Y$ and homotopies $s_Y\circ f \cong s_X$ and $t_Y \circ f \cong t_X$ which intertwine the natural transformations $\alpha$ and $\beta$.

We say that a decorated span $(X, \alpha)$ is \emph{source finite} if $s:X \to A$ has $\pi$-finite (homotopy) fibers, \emph{target finite} if $t:X \to B$ has $\pi$-finite fibers, and \emph{$\pi$-finite} if all spaces $A, B$ and $X$ are $\pi$-finite. 
\end{definition}

\begin{prop}\label{prop:linearization}For a source finite decorated span $(B, \cL_B) \ot (X, \alpha) \to (A, \cL_A) $ and a point $a\in A$ consider the following integral in the vector space $\Hom(\cL_A(a), \colim \cL_B)$: 
\begin{equation}\label{eq:deflinearization}  \left( \Phi_{X, \alpha}\right)_a:= \int_{(x, \gamma) \in X_a} \left[ \cL_A(a) \to[\cL_A(\gamma)] \cL_A(s(x)) \to[\alpha_x] \cL_B(t(x)) \to [\iota^B_{t(x)}] \colim \cL_B \right]  \end{equation}
Here, $X_a$ denotes the homotopy fiber of $s:X\to A$ at $a$, whose points are represented by pairs  $(x, \gamma)$ of a point $x\in X$ and a path $\gamma: a \rightsquigarrow s(x)$ in $A$, and  $\iota^B_b: \cL_B(b) \to \colim \cL_B$ (for $b \in B$) denote the universal morphisms into the colimit.

These maps $\left( \Phi_{X, \alpha}\right)_a: \cL_A(a) \to \colim \cL_B$ assemble into a map \[ \Phi_{X, \alpha} : \colim \cL_A \to \colim \cL_B\]
which only depends on the isomorphism class of the decorated span.
\end{prop}
\begin{proof}
The morphism $\left(\Phi_{X, \alpha}\right)_a \in \Hom(\cL_A(a), \colim \cL_B)$ is defined as the integral of the function $\pi_0 X_a \to \Hom(\cL_A(a), \colim \cL_B)$ given by $(x, \gamma) \mapsto \iota^B_{t(x)} \circ \alpha_{x} \circ \cL_A(\gamma)$. This function intertwines the canonical $\pi_1(A, a)$ action on $\pi_0 X_a$ with the action on $\Hom(\cL_A(a), \colim \cL_B)$ induced by the local systems. Hence, it follows from Lemma~\ref{lem:equivariance} that this integral is a $\pi_1(A, a)$ fixed point in $\Hom(\cL_A(a), \colim \cL_B)$, or equivalently that it defines an element in $\Hom(\cL_A(a)_{\pi_1(A, a)}, \colim \cL_B)$ and hence that the assignments~\eqref{eq:deflinearization} assemble into a map $\Phi_{X, \alpha}: \colim \cL_A \to \colim \cL_B$. \end{proof}

Henceforth, we refer to the map $\Phi_{X, \alpha}: \colim \cL_A \to \colim \cL_B$ associated to a decorated span as the \emph{linearization} of this span.

\begin{example} Let $X$ be a $\pi$-finite space and consider the span $* \leftarrow X \to *$ decorated with the trivial local system in $\cV= \Vec_\QQ$. Then, $\Phi_{X}: \QQ \to \QQ$ is given by the total homotopy cardinality $\#^{\mathrm{tot}}(X)$. 
\end{example}

\begin{remark} \label{rem:limitversion}
If $\cV$ has limits and  $(B, \cL_B) \leftarrow (X, \alpha) \to (A, \cL_A)$ is a target finite span (i.e. $t:X \to B$ has $\pi$-finite fibers), then a completely analogous construction leads to a morphism \[\Phi^{X, \alpha}:  \lim \cL_A \to \lim \cL_B\] with component at $b \in B$ given by the following map $\lim \cL_A \to \cL_B(b)$:
\[ \int_{(x, \gamma:t(x) \rightsquigarrow b) \in X_b} \left[\lim \cL_A \to \cL_A(s(x)) \to[\alpha_x] \cL_B(t(x))  \to[ \cL_B(\gamma)] \cL_B(b) \right]\]
Here, $X_b$ denotes the fiber of $t:X \to B$ at $b \in B$.
In fact, for any $\pi$-finite space $X$ and rational vector space $V$, the integration function $\int_X: \mathrm{Func}(\pi_0 X, V) \to V$ itself arises as the linearization $\Phi^X$ of the span $ (*, \mathrm{const}_V) \leftarrow (X, \id) \to (X, \mathrm{const}_V) $ where $\mathrm{const}_V$ denotes the constant local system at $V$ with limit $\lim(\mathrm{const}_V: X \to \Vec_{\QQ}) = \mathrm{Func}(\pi_0 X, V)$. 
\end{remark}

\begin{remark}\label{rem:finiteversion}
If $\cV$ has limits and colimits and if $A$ is a $\pi$-finite space,  the \emph{norm map} $\mathrm{Nm}_A:\colim \cL_A \to  \lim \cL_A$ is defined as the map with coefficient $\cL_A (a) \to \cL_A(a')$ given by
\[ \int_{\gamma \in \mathrm{Path}_A(a,a')} \left[\cL_A(a) \to[\cL_A(\gamma)]\cL_A(a')\right].
\]
In the category $\cV=\Vec_k$ of vector spaces over a field $k$ of characteristic zero, and for connected $A$ with a choice of basepoint $a\in A$, the norm map therefore explicitly unpacks to a map $\cL_A(a)_{\pi_1(A,a)} \to \cL_A(a)^{\pi_1(A,a)}$ which sends a coinvariant $[v]$ to the invariant
\begin{equation}\label{eq:normmap}
\#(A,a)^{-1}~\frac{1}{|\pi_1(A,a)|}\sum_{\gamma \in \pi_1(A,a)} ~\cL(\gamma) v.\end{equation}
Here, we used that for any loop $\gamma \in \Omega_aA$,  
$$\#(\Omega_a A, \gamma) = \#(\Omega_a A, \id) =|\pi_2(A,a)|^{-1} |\pi_3(A,a)| \cdots = \#(A, a)^{-1} |\pi_1(A,a)|^{-1}.$$

Since $\cV$ is $\Vec_{\QQ}$-enriched, the norm map $\mathrm{Nm}_A: \colim \cL_A \to \lim \cL_A$ is invertible. If  $A$, $B$ and $X$ are $\pi$-finite, the linearizations $\Phi_{X, \alpha}: \colim \cL_A \to \colim \cL_B$ and $\Phi^{X, \alpha}: \lim \cL_A \to \lim \cL_B$ are identified with one another under these norm isomorphisms $\mathrm{Nm}_A$ and $\mathrm{Nm}_B$. 

In fact, for this last $\pi$-finite case, it is sufficient to assume that the $\Vec_{\QQ}$-enriched category $\cV$ is additive and idempotent complete, as all necessary limits and colimits can be computed in these terms.\end{remark}

 Given a pair of decorated spans
\begin{equation}\label{eq:composablespans}
\begin{tz}[scale=0.92]
	\node (S1) at (0,0) {$ (C, \cL_{C}) $};
	\node (S2) at (3,0) {$ (B, \cL_{B}) $};
	\node (S3) at (6,0) {$ (A, \cL_{A}) $};
		
	\node (X1) at (1.5,1) {$ (Y, \beta) $};
	\node (X2) at (4.5,1) {$ (X, \alpha) $};
		
	\draw [->] (X1) to node[above left]{$t_Y$} (S1);
 	\draw [->] (X1) to node[above right]{$s_Y$} (S2);
	\draw [->] (X2) to  node[above left]{$t_X$}(S2);
 	\draw [->] (X2) to node[above right ]{$s_X$} (S3);
	\end{tz}
\end{equation}
their \emph{pullback composite} is the span $C \ot[t_Y \circ p_Y] Y \times_B X \to[s_X \circ p_X] A$, where $p_X: Y \times_B X \to X$ and $p_Y: Y \times_B X \to Y$ denote the projections, decorated with the local systems $\cL_C$ on $C$, $\cL_A$ on $A$ and transformation $\beta \circ \alpha: p_X^* s_X^* \cL_A \To p_Y^* t_Y^* \cL_C$ with component  $(\beta\circ \alpha)_{(x,y, \gamma)}$ at a point $(y \in Y, x\in X,  \gamma: t_X(x)\rightsquigarrow s_Y(y)) \in Y \times_B X$ given by the composite
\[ \cL_A(s_X(x)) \to[\alpha_x] \cL_B(t_X(x)) \to [\cL_B(\gamma)] \cL_B(s_Y(y)) \to [\beta_y] \cL_C(t_Y(y)).
\]

\begin{theorem}\label{thm:composingspans}
Given a pair of decorated spans as in~\eqref{eq:composablespans} which are source finite (i.e. $s_Y:Y \to B$ and $s_X:X\to A$ have $\pi$-finite fibers). Then, their pullback composite $(C, \cL_C) \leftarrow (Y\times_B X, \beta \circ \alpha) \to (A, \cL_A)$ is again source finite and their linearizations compose:
\begin{equation}\label{eq:compositephi}
\Phi_{Y \times_B X, \beta \circ \alpha} = \Phi_{Y, \beta} \circ \Phi_{X, \alpha}  : \colim \cL_A \to \colim \cL_C
\end{equation}
\end{theorem}
\begin{proof}
The fiber of $Y\times_B X \to X$ at a point $x\in X$ is identified with the fiber of $Y \to B$ at the point $t_X(x) \in B$.  Since $Y\to B$ has $\pi$-finite fibers, it therefore follows that $Y\times_B X \to X$ has $\pi$-finite fibers. The source map of the composite span is therefore a composite $Y \times_B X \to X \to A$ of maps with $\pi$-finite fibers, and hence has itself $\pi$-finite fibers.

A point in $Y\times_B X$ is represented by a triple $(y, x, \mu)$ consisting of a point $y \in Y$, a point $x\in X$ and a path $\mu: t_X(x) \rightsquigarrow s_Y(y))$ in $B$. Hence, a point in the fiber $(Y\times_B X)_a$ of  $Y\times_B X \to X \to A$ at $a\in A$ is a quadruple $(y, x, \mu, \gamma)$ of a point $(y, x, \mu) \in Y \times_B X$ as above and a path $\gamma:a \to s_X(x)$. Using this notation, the map  $\left(\Phi_{Y \times_B X, \beta\circ \alpha}\right)_a \in \Hom(\cL_A(a), \colim \cL_C)$ is defined as the following integral (see~\eqref{eq:deflinearization}):
\begin{align}\nonumber
&\phantom{{}={}}\int_{\left(\vphantom{\frac{a}{b}}y, x, \mu: t_X (x) \rightsquigarrow s_Y(y),  \gamma:a \rightsquigarrow s_X(x)  \right) \in (Y \times_B X)_a} \iota^C_{t_Y(y)} \circ \beta_y\circ  \cL_B(\mu)\circ  \alpha_x \circ \cL_A(\gamma)
\\\label{eq:proofcomp1}
&=
\int_{\left(\vphantom{\frac{a}{b}}x, \gamma: a \rightsquigarrow s_X(x)\right) \in X_a} \left(\int_{\left(\vphantom{\frac{a}{b}}y, \mu: t_X(x) \rightsquigarrow s_Y(y)\right) \in Y_{t_X(x)}} \iota^C_{t_Y(y)} \circ \beta_y\circ  \cL_B(\mu)\right)\circ  \alpha_x \circ \cL_A(\gamma)
\\\label{eq:proofcomp2}
 &=
 \int_{\left(\vphantom{\frac{a}{b}}x, \gamma: a \rightsquigarrow s_X(x)\right) \in X_a} \left(\Phi_{Y, \beta}\right)_{t_X(x)} \circ \alpha_x \circ \cL_A(\gamma)
\\\label{eq:proofcomp3}
&=\Phi_{Y, \beta} \circ \left(\int_{\left(\vphantom{\frac{a}{b}}x, \gamma: a \rightsquigarrow s_X(x)\right) \in X_a}  \iota^B_{t_X(x)}\circ  \alpha_x \circ \cL_A(\gamma)\right) = \Phi_{Y, \beta} \circ \left(\Phi_{X,\alpha} \right)_a
\end{align}
In equation~\eqref{eq:proofcomp1}, we identified the fiber of $(Y\times_B X)_a\to X_a$ at $(x, \gamma)$ with $Y_{t_X(x)}$ and used Proposition~\ref{prop:Fubini} to split the integral. Bilinearity of composition allows to interchange the integral and composition (see~\eqref{eq:integrallinear}).
Equation~\eqref{eq:proofcomp2} uses the definition~\eqref{eq:deflinearization} of $(\Phi_{Y, \beta})_{t_X(x)}$, and equation~\eqref{eq:proofcomp3} uses the defining property $\Phi_{Y, \beta} \circ \iota^B_b = \left( \Phi_{Y, \beta} \right)_b$ for $b \in B$ and again bilinearity of composition. \end{proof}

\begin{example} Theorem~\ref{thm:composingspans} may be understood as a direct generalization of the compatibility of homotopy cardinality with fiber sequences. Indeed, let $X\to B$ be a map of spaces with $\pi$-finite fiber and connected pointed base $B$. Applying Theorem~\ref{thm:composingspans} to the sequence of spans $* \leftarrow * \to B  \leftarrow X \to *$ with composite span $* \leftarrow F \to *$ (all decorated with the trivial local system in $\Vec_\QQ$) recovers~\eqref{eq:totalhomotopyproduct}. \end{example}

Theorem~\ref{thm:composingspans} suggests to think of linearization $\Phi$ as a functor out of a category of decorated source finite spans. 

\begin{definition}\label{def:decoratedspancat}
	We denote the $1$-category of spaces and isomorphism classes of spans by $\mathrm{Span}(\Sp)$. Similarly, 
	we denote the $1$-category of spaces equipped with $\cV$-valued local systems and isomorphism classes of decorated spans by   $\mathrm{Span}(\Sp, \cV)$. Any symmetric monoidal structure on $\cV$ induces a symmetric monoidal structure on $\mathrm{Span}(\Sp, \cV)$ given on objects by the cartesian product $A\times B$ equipped with the local system $\cL_A \otimes \cL_B: A\times B \to \cV \times \cV \to \cV$. 

We also denote the (symmetric monoidal) subcategory of source finite and target finite decorated spans (see Definition~\ref{def:decoratedspan}) by $\mathrm{Span}^{s.f}(\Sp, \cV)$ and $\mathrm{Span}^{t.f}(\Sp, \cV)$, respectively, and denote the further subcategory on $\pi$-finite spaces and $\pi$-finite decorated spans by $\mathrm{Span}(\Sp^{\pi}, \cV)$. 
\end{definition}

 \begin{remark}\label{rem:inftyspan}
2-categories of spans in general categories were considered as early as \cite{MR0220789}. The importance of spans of spaces with local systems to topological field theories appears in \cite{MR2555928,MR2648901}. Higher $\infty$-categories of spans were constructed by \cite{Barwick:2013th}, and generalized to $(\infty,n)$-categories of iterated spans with local systems in \cite{MR3830256}. Spans of groupoids with local systems were also considered in \cite{MR3348259}.  
 \end{remark}

 \begin{cor}\label{cor:spanfunctor} If $\cV$ admits small colimits and is enriched in rational vector spaces, the linearizations $\Phi_{X, \cV}: \colim \cL_A \to \colim \cL_B$ from Proposition~\ref{prop:linearization} assemble into a functor 
 \[ \Phi:\mathrm{Span}^{s.f.}(\Sp, \cV) \to \cV.
 \]
 If $\cV$ has a symmetric monoidal structure which distributes over colimits, then this functor is symmetric monoidal. 
 \end{cor}
 \begin{proof}
 Functoriality follows from Theorem~\ref{thm:composingspans} and the fact that the identity span on $(A, \cL_A)$ linearizes to the identity map on $\colim \cL_A$. Symmetric monoidality may ultimately be reduced to the following property of $\int$: If $X$ and $Y$ are $\pi$-finite spaces, and $\alpha: \pi_0 X \to V$ and $\beta: \pi_0 Y \to W$ are functions into rational vector spaces, then the following vectors in $V\otimes W$ agree:
 $$\int_{X\times Y} \alpha \otimes \beta = \left(\int_X \alpha \right) \otimes\left( \int_Y \beta\right)  $$
  \end{proof}

\begin{remark} The importance of the linearization functor to topological field theories was made clear in \cite{MR2648901}. A construction of this linearization for spans of 1-groupoids with cocycles appears in \cite{MR3348259}, and for spans of 1-groupoids with more general local systems in~\cite{Trova:2016aa,MR3862670}. Subsequently, the theory of \emph{ambidexterity} has provided a far reaching generalization, see \cite{HopkinsLurie13,MR3221291,MR4419631} and especially \cite{MR4133704}.
\end{remark}

\begin{remark}\label{rem:limitlinearization}As in Remark~\ref{rem:limitversion}, there is a limit/target-finite variant of Corollary~\ref{cor:spanfunctor}. If $\cV$ has limits (and a symmetric monoidal structure which distributes over limits), then the linearization $\Phi^{X, \alpha}: \lim\cL_A \to \lim \cL_B$ induces a (symmetric monoidal) functor $\mathrm{Span}^{t.f.}(\Sp, \cV) \to \cV$. If one further restricts to $\pi$-finite spaces as in Remark~\ref{rem:finiteversion}, 
	then the norm map isomorphisms induce a monoidal natural equivalence between the colimit and limit variant of the functor $\mathrm{Span}(\Sp^{\pi}, \cV) \to \cV$.
\end{remark}

\subsection{Linearizing $\pi$-finite categories}\label{sec:linearizingfinitecats}
For the remainder of this section, we restrict attention to the case $\cV = \Vec_k$ where $k$ is a field of characteristic zero. 

A rich source of decorated spans comes from $\pi$-finite categories. Let $\cC$ be an $\infty$-category (which for our purposes can be taken to be a category enriched in spaces). We let $\cC_0 = \Map([0], \cC)$ and $\cC_1 = \Map([1], \cC)$ be the \emph{moduli space of objects of $\cC$}, respectively the \emph{moduli space of arrows of $\cC$}. Here $[0] = pt$ is the terminal category and $[1]$ is the `free-walking-arrow'.

The source and target map lead to a span of spaces associated to the $\infty$-category $\cC$:
\begin{equation*}
	\cC_0 \ot[t] \cC_1 \to[s] \cC_0.
\end{equation*} 
Any functor $F: \cC \to \Vect_{k}$ restricts to a ($\Vec_{k}$-valued) local system $\cL_F$ on $\cC_0$, and a map of local systems $\alpha_F:s^* \cL_F \to t^* \cL_F$ on $\cC_1$ (with component at $(g:a\to b) \in \cC_1$ given by $(\alpha_F)_g: s^* \cL_F(g) = \cL_F(a) = F(a) \stackrel{F(g)}{\to} F(b) = \cL_F(b) = t^*\cL_F(g)$) .  Thus, any such functor induces a decorated span 
\begin{equation}\label{eq:decorationfunctor}
(\cC_0, \cL_F) \ot[t] (\cC_1, \alpha_F) \to[s] (\cC_0, \cL_F).
\end{equation}
We will also denote that span by $(\cC, F)$ or just $\cC$ when there is no confusion. 

We will say that an $\infty$-category $\cC$ is \emph{$\pi$-finite} if it has finitely many isomorphism classes of objects and if for all objects $a, b \in \cC$ the morphism spaces $\cC(a, b)$ are $\pi$-finite. Equivalently, it is $\pi$-finite if both $\cC_0$ and $\cC_1$ are $\pi$-finite spaces.

\begin{definition}\label{def:linearizationfunctor} Let $\cC$ be a $\pi$-finite $\infty$-category and let $F: \cC \to \Vec_k$ be a functor into the category of vector spaces over a field  $k$ of characteristic zero. The \emph{linearization} of $(\cC, F)$ is the linear map
\[ \Phi_{\cC, F}: \colim \cL_F \to \colim \cL_F 
\]
associated to the decorated span $(\cC, F)$. 
\end{definition}

In the following, we say that an $\infty$-category has the property that \emph{endomorphisms are invertible} if any morphism $f:a\to b$ with isomorphic source and target is itself invertible. 
\begin{theorem}[M\"obius inversion]\label{thm:Moebius} 
Let $\cC$ be a $\pi$-finite $\infty$-category in which endomorphisms are invertible. Then, for any functor $F:\cC \to \Vec_k$, the linearization $\Phi_{\cC, F}$ is invertible. \end{theorem}
\begin{proof} 
Since $\cC$ is an $\infty$-category (and in particular $\cC_0, \cC_1$ are the spaces of $0$- and $1$-simplices in a complete Segal space~\cite{MR1804411}), the unit map $c_0:\cC_0 \to \cC_1$ is an inclusion of components (i.e. an injection on $\pi_0$ and a bijection on higher homotopy groups)\cite[Sections~5 and~6]{MR1804411}. Let $\cC_1^{<}$ denote the complement of $c_0:\cC_0 \hookrightarrow \cC_1$, i.e. the disjoint union of all components of $\cC_1$ not in the image of $c_0$, so that $\cC_1 \cong\cC_0 \sqcup \cC_1^{<}$. Intuitively, $\cC_1^{<}$ is the full subspace of $\cC_1$ of those morphisms $f:a\to b$ which are not invertible. More generally, define
\begin{equation}\label{eq:MoebiusdefCn}
\cC_n^{<} := \cC_1^{<} \times_{\cC_0} \cC_1^{<} \times_{\cC_0} \cdots \times_{\cC_0} \cC_1^{<}.
\end{equation} We may intuitively think of $\cC_n^{<}$ as the space of $n$-chains $a_0 \to[f_{0,1}] a_1 \to... \to[f_{n-1, n}] a_n$ for which no $f_{i, i+1}$ is invertible. 
There are maps $t,s: \cC^<_n \to \cC_0$ which send a chain $a_0 \to a_1 \to \cdots \to a_n$ to $a_n$ (target) and $a_0$ (source), yielding spans $\cC_0 \ot[t] \cC^<_n \to[s] \cC_0$. Abusing notation, we will henceforth simply denote these spans by $\cC_n^{<}$.

As in~\eqref{eq:decorationfunctor}, let $(\cC_1, \alpha_F)$ denote the decoration of the span $\cC_0 \ot[t] \cC_1 \to[s] \cC_0$ induced by the functor $F: \cC \to \Vec_{k}$. Let $\alpha_F^{<}$ denote the restriction of $\alpha_F$ to the subspan $\cC_1^{<} \hookrightarrow \cC_1$. Unitality of $F$ implies that the decomposition of spans $\cC_1\cong \cC_0 \sqcup \cC_1^{<}$ is compatible with decorations:  $$(\cC_1, \alpha_F) \cong (\cC_0, \id) \sqcup (\cC_1^{<}, \alpha_F^{<})$$
Composition of morphisms defines a map of spans $c_n: \cC_n^{<} \to \cC_1$ (generalizing the inclusions $c_0: \cC_0 \hookrightarrow \cC_1$ and $c_1: \cC_1^{<} \hookrightarrow \cC_1$). Functoriality of $F$ implies that the pulled-back decoration $c_n^* \alpha_F$ on $\cC_n^{<}$ is compatible with the defining decomposition~\eqref{eq:MoebiusdefCn} of $\cC_n^{<}$: \begin{equation}\label{eq:moebiusproof}
(\cC_n^{<}, c_n^*\alpha_F)  \cong (\cC_1^{<}, \alpha_F^{<}) \times_{(\cC_0, \cL_F)} \cdots \times_{(\cC_0, \cL_F)} (\cC_1^{<}, \alpha_F^{<})\end{equation}

We can amalgamate the decorated spans $(\cC_n^{<}, c_n^* \alpha_F)$ into a single decorated span $(\cC_0, \cL_F) \ot[t] (M, \mu) \to[s] (\cC_0, \cL_F)$ with
\begin{equation*}
	(M, \mu) := \coprod_{n \geq 0} (\cC_n^<, (-1)^n c_n^* \alpha_F) \cong (\cC_0, \id) \sqcup (\cC_1^<, -\alpha_F) \sqcup (\cC_2^<, \alpha_F \circ \alpha_F) \sqcup  \cdots 
\end{equation*}
Since endomorphisms in $\cC$ are invertible, for $[x], [y] \in \pi_0 \cC_0$ the relation $[x] \leq [y]$ if there exists a morphism $x\to y$ in $\cC$ defines a poset structure on the set $\pi_0 \cC_0$ of isomorphism classes of objects of $\cC$. In particular, it follows that if a composite $g_1\circ g_2 \circ \cdots \circ g_n$ of morphisms is invertible, then every single morphism $g_i$ has isomorphic source and target and hence is also invertible. In particular, for any point in $\cC_n^{<}$ represented by a chain $a_0 \to[f_{0,1}] a_1 \to... \to[f_{n-1, n}] a_n$ of non-invertible morphisms, it follows that also the composite of any subchain is non-invertible. Hence, if $n > |\pi_0 \cC_0|$, the space $\cC_n^{<}$ is empty. In particular, the space $M$ is $\pi$-finite and the span $(M, \mu)$ can be linearized to a map $\Phi_{(M, \mu)}: \colim \cL_F \to \colim \cL_F$. 

Using the decomposition of spans $(\cC_1, \alpha_F) \cong (\cC_0, \id) \sqcup (\cC_1^<, \alpha_F^<)$, it follows from~\eqref{eq:moebiusproof} that
\[(\cC_1, \alpha_F) \times_{(\cC_0, \cL_F)} (\cC_n^{<}, c_n^* \alpha_F) \cong (\cC_n^{<}, c_n^* \alpha_F) \sqcup (\cC_{n+1}^{<}, c_{n+1}^* \alpha_F).
\]
Since $\Phi_{(M,\mu)} = \sum_{n \geq 0} (-1)^n \Phi_{(\cC_n^<, c_n^* \alpha_F)}$ (with only finitely many non-zero terms in this sum), the linearization $\Phi_{(M, \mu)}$ is inverse to the linearization $\Phi_{(\cC_1, \alpha_F)}$.
\end{proof}

\begin{remark}Theorem~\ref{thm:Moebius} is closely related to the theory of decomposition spaces and M\"obius inversion therein~\cite{MR3804694,MR3818099, MR3828744}. Indeed, similar to~\cite{MR3818099}, the assumptions of Theorem~\ref{thm:Moebius} can be weakened: Instead of assuming that endomorphisms are invertible, it would have been sufficient to assume $\pi$-finiteness of $\cC$ and that every morphism $f:a\to b$ has \emph{finite length}, in the sense that there is a finite upper bound on the length of a chain of non-invertible morphisms into which it can be decomposed. In the proof of Theorem~\ref{thm:Moebius}, $\pi$-finiteness of $\cC$ together with this finite length condition is enough to ensure the eventual vanishing of $\cC_n^{<}$.
In fact, similar to~\cite{MR3818099}, we expect that it is not even necessary to assume that $\cC$ is an $\infty$-category --- a version of Theorem~\ref{thm:Moebius} should still apply to mere `decomposition spaces'~\cite[Definition~3.1]{MR3804694} (also known as $2$-Segal space~\cite{MR3970975}). However, the $\pi$-finite categories for which we show invertibility of $\Phi_{\cC, F}$ in Corollary~\ref{cor:factinv} do not in general fulfill this finite length condition and hence go beyond the categories considered in~\cite{MR3818099}. See Remark~\ref{rmk:finitesetcardN}.

\end{remark}

\begin{definition}\label{def:orthongonalfactsystem}[{cf.~\cite[Prop.~5.2.8.17]{MR2522659}}]
	Let $\cC$ be an $\infty$-category, and let $(\cL, \cR)$ be a pair of subcategories, each of which contains all the isomorphisms of $\cC$. The composition of morphism induces a natural map from the composite of the spans $\cC_0 \ot[t] \cR_1 \to[s] \cC_0$ and $\cC_0 \ot[t] \cL_1 \to[s] \cC_0$ to the span $\cC_0 \ot[t] \cC_1 \to[s] \cC_0$. The pair $(\cL, \cR)$ is an \emph{orthogonal factorization system} if this map is an equivalence of spans. We will often abuse notation and let $\cL$ and $\cR$ denote the corresponding spans, as well as the subcategories.
\end{definition}
\begin{remark}
	On homotopy fibers this condition becomes 
	\begin{equation*}
		\cC(a, c) \simeq \coprod_{[b] \in \pi_0\cC_0} \cR(b, c) \times_{\Aut_\cC(b)} \cL(a, b).
	\end{equation*}
	In particular every morphism can be canonically factored into the composite of a morphism in $\cL$ followed by a morphism in $\cR$. \end{remark}

\begin{definition}\label{def:nestedfactsystem}
A \emph{nested factorization system} is a collection $(\cL^{(k)}, \cR^{(k)})$ of orthogonal factorization systems for $1 \leq k \leq n$ such that  $\cR^{(k-1)} \subseteq \cR^{(k)}$ for all $k$ (and equivalently, $\cL^{(k+1)} \subseteq \cL^{(k)}$). Given a nested factorization system, we let $\cT^{(\ell)}$ denote the following subcategories for $0 \leq \ell \leq n$:
	\begin{equation*}
		\cT^{(\ell)} = \begin{cases}
			\cR^{(1)} & \textrm{if } \ell=0 \\
			\cR^{(\ell+1)} \cap \cL^{(\ell)} &\textrm{if } 1 \leq \ell \leq n-1  \\
			\cL^{(n)} & \textrm{if } \ell=n
		\end{cases}
	\end{equation*}
	\end{definition}

\begin{remark}
	Each subcategory $\cT^{(\ell)}$ gives a span $\cC_0 \leftarrow \cT^{(\ell)} \to \cC_0$. For each $k = 1, \dots, n$ we have the following compositions of spans
	\begin{align*}
		\cR^{(k)} &= \cT^{(0)} \times_{\cC_0} \cT^{(1)}  \times_{\cC_0} \cdots \times_{\cC_0}  \cT^{(k-1)} \\
		\cL^{(k)} &= \cT^{(k)} \times_{\cC_0} \cT^{(k+1)}  \times_{\cC_0} \cdots \times_{\cC_0}\cT^{(n)} \\
		\cC_1 &= \cT^{(0)} \times_{\cC_0} \cT^{(1)}  \times_{\cC_0} \cdots \times_{\cC_0}  \cT^{(n)}.
	\end{align*}
The nested sequence of orthogonal factorization systems can be recovered from the subcategories $\cT^{(\ell)}$, and so they provide an equivalent presentations.   
\end{remark}

\begin{example}\label{exm:truncatedconnected} For $n \geq -2$, a map $f: X \to Y$ between topological spaces is called \emph{$n$-connected} if for all points $x\in X$, $\pi_{i}(f;x)$ is an isomorphism for $i < n+1$ and surjective for $i=n+1$. A map $f:X\to Y$ is \emph{$n$-truncated} if $\pi_i (f;x)$ is injective for $i=n+1$ and an isomorphism for $i >n+1$. (Recall that the set $\pi_{-1}$ of a space $X$ is either $\emptyset$ if $X$ is, or is $\{*\}$ if $X$ is non-empty.) Equivalently, $f$ is $n$-truncated if the homotopy fiber at every basepoint in $Y$ is $n$-truncated (i.e. is an $n$-type; all homotopy groups $\pi_{i}$ for $i > n$ vanish) and it is $n$-connected if the homotopy fibers are $n$-connected (i.e. all homotopy groups $\pi_{i}$ for $i \leq n$ vanish).

Any map $f:X \to Y$ of spaces may be factored, in an essentially unique way, into a $n$-connected map followed by a $n$-truncated map. Formally, this amounts to the assertion that the $\infty$-category of spaces admits an orthogonal factorization system $(\cL^{(n)}, \cR^{(n)})$ in which the left class $\cL^{(n)}$ consists of the $(n-1)$-connected maps and the right class $\cR^{(n)}$ consists of the $(n-1)$-truncated 
maps, see  \cite[Example~5.2.8.16]{MR2522659}. This system is nested in the sense that $\cR^{(n-1)} \subseteq \cR^{(n)}$ and $\cL^{(n+1)} \subseteq \cL^{(n)}$. Given a map $f:X \to Y$, for each integer $n \geq -2$ we may factor $f$ as $X \to Z_n \to Y$ into an $n$-connected map followed by an $n$-truncated map. The fact that these orthogonal factorization systems are nested implies the the spaces $\{Z_n\}$ assemble into a tower:
	
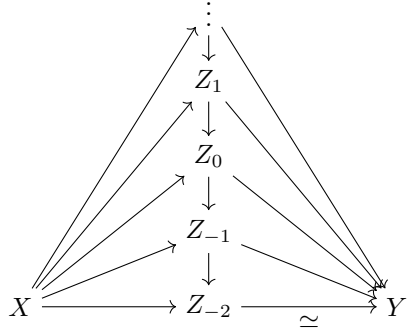
\begin{center}
	\begin{tikzpicture}
		\node (X) at (0,0) {$ X $};
		\node (Zb) at (2.5, 0) {$ Z_{-2}  $};
		\node (Za) at (2.5, 1) {$ Z_{-1}  $};
		\node (Z0) at (2.5, 2) {$ Z_{0}  $};
		\node (Z1) at (2.5, 3) {$ Z_{1}  $};		
		\node (Z2) at (2.5, 4) {$ \vdots  $};
		
		\node (Y) at (5, 0) {$ Y $};
		
		\draw [->] (X) to (Zb); 
		\draw [->] (Zb) to node [below] {$\simeq$} (Y); 
		\draw [->] (X) to (Za); 
		\draw [->] (Za) to  (Y);
		\draw [->] (Za) to  (Zb);
		\draw [->] (X) to (Z0); 
		\draw [->] (Z0) to  (Y);
		\draw [->] (Z0) to  (Za);
		\draw [->] (X) to (Z1); 
		\draw [->] (Z1) to  (Y);
		\draw [->] (Z1) to  (Z0);
		\draw [->] (X) to (Z2); 
		\draw [->] (Z2) to  (Y);
		\draw [->] (Z2) to  (Z1);
	\end{tikzpicture}
	\end{center}
	This is the \emph{Moore-Postnikov tower} of the map $f: X \to Y$. 
\end{example}

\begin{corollary}\label{cor:factinv} Let $\cC$ be a $\pi$-finite $\infty$-category which admits a nested factorization system $(\cL^{(k)}, \cR^{(k)})_{1 \leq k\leq n}$ such that for every $0 \leq \ell \leq n$,  endomorphism in the category $\cT^{(\ell)}$ are invertible. Then, for any functor $F: \cC \to \Vec_k$, the linearization $\Phi_{\cC, F}$ is invertible. 
\end{corollary}

\begin{proof} By definition, the decorated span $\cC_0 \leftarrow \cC_1 \rightarrow \cC_0$ is the composite of decorated spans $\cT^{(0)} \times_{\cC_0} \ldots \times_{\cC_0} \cT^{(n)}$ (with maps of local systems given by restricting the functor $F$ to the subcategories $\cT^{(\ell)}$). Therefore, the corollary follows from Theorems~\ref{thm:Moebius} and~\ref{thm:composingspans}.
\end{proof}

\begin{remark}\label{rmk:finitesetcardN}
	Let $\cC = \Fin\Set_{\leq N}$ be the category of finite sets with cardinality less than or equal to $N$. By limiting the cardinality, $\cC$ is a $\pi$-finite category. This category contains non-trivial idempotents, and thus does not satisfy the \emph{M\"obius condition} of \cite{MR3818099}. Nevertheless, 
epimorphisms and monomorphisms of finite sets which are endomorphisms are automatically bijections. Thus the single factorization system $(\Surj, \Inj)$ satisfies the conditions of Corollary~\ref{cor:factinv}. Consequently for any functor $F: \Fin\Set_{\leq N} \to \Vec_k$, the linearization $\Phi_{\Fin\Set_{\leq N}, F}$ is invertible. 
\end{remark}

\subsection{Linearizing locally $\pi$-finite categories}\label{sec:linearizelocpifinite}
 In this section, we extend the results of Section~\ref{sec:linearizingfinitecats} to $\infty$-categories with possibly infinitely many isomorphism classes of objects. 
An $\infty$-category $\cC$ is \emph{locally $\pi$-finite} if all its morphism spaces $\cC(a, b)$ are $\pi$-finite for any $a, b \in \cC$. For such categories, neither map in the span $\cC_0 \leftarrow \cC_1 \rightarrow \cC_0$ necessarily has $\pi$-finite fibers. Hence, the linearization $\Phi_{\cC, F}$ of Definition~\ref{def:linearizationfunctor} is not  defined. However, the span $* \leftarrow \cC_1 \to[(t,s)] \cC_0 \times \cC_0$ is still source finite (Definition~\ref{def:decoratedspan}) since the fiber of $(t,s): \cC_1 \to \cC_0 \times \cC_0$ at a point $(b,a) \in \cC_0 \times \cC_0$ is precisely the morphism space $\cC(a, b)$.

Analogous to~\eqref{eq:decorationfunctor}, any functor $F:\cC \to \Vec_{k}$ induces a decoration of this span
\begin{equation}\label{eq:decoratedspanpairing}
  (*, k)\leftarrow (\cC_1, p_F) \to[(t,s)] (\cC_0, \cL_F^\vee) \times (\cC_0, \cL_F).
\end{equation}
As before,  $\cL_F: \cC_0 \to \Vec_{k}$ denotes the restriction of $F$ to $\cC_0$ and $\cL_F^{\vee}: \cC_0 \to \Vec_{k}$ denotes the local system with value at $d \in \cC_0$ given by the dual vector space $F(d)^{\vee}:= \Vec_{k}(F(d), k).$ The transformation  $p_F: t^*\cL_F^\vee  \otimes s^*\cL_F \To \mathrm{const}_{k}$ of local systems on $\cC_1$  is given at an $f:a\to b$ in $\cC_1$ by 
\[F(b)^\vee \otimes F(a) \to[ \id_{F(b)^{\vee}}\otimes F(f)] F(b)^\vee \otimes F(b)   \to k.\]

\begin{definition}\label{def:linearpairing} 
Let $\cC$ be a locally $\pi$-finite $\infty$-category and let $F: \cC \to \Vec_k$ be a functor into the category of vector spaces over a field $k$ of characteristic zero. The \emph{linear pairing} associated to $(\cC, F)$ is the linear map 
\[\langle -, - \rangle_{\cC, F}: \colim \cL_F^\vee \otimes \colim \cL_F \to k\]
obtained from linearizing the above span $(*, k)\leftarrow (\cC_1, p_F) \rightarrow (\cC_0, \cL_F^\vee) \times (\cC_0, \cL_F).$
\end{definition}

\begin{remark}\label{rem:explicitformulapairing} Expressing the colimit as a sum over coinvariants, the linear pairing associated to $(\cC, F)$ is a linear map 
\[\langle -, - \rangle_{\cC, F}: \left( \bigoplus_{ [d]\in \pi_0 \cC_0} \left(F(d)^{\vee}\right)_{\pi_0 \Aut_{\cC}(d)} \right) \otimes \left( \bigoplus_{ [c]\in \pi_0 \cC_0} F(c)_{\pi_0 \Aut_{\cC}(c)} \right)  \to k.\]
Explicitly unpacking~\eqref{eq:deflinearization}, given $(d \in \cC_0, \phi \in F(d)^\vee)$ and $(c\in \cC_0, v\in F(c))$, their pairing is given by 
\[\langle (d, \phi), (c, v)\rangle_{\cC, F} = \sum_{[f]\in \pi_0 \cC(c, d)} ~\#(\cC(c,d), f) ~ \phi(F(f)(v)).
\]
Note that this formula indeed only depends on the orbits of $\phi$ and $v$ under the $\Aut_{\cC}(d)$ and $\Aut_{\cC}(c)$ action, respectively, and the isomorphism classes of $c$ and $d$. 
\end{remark}

\begin{example}\label{exm:pairing1cat}
As an important special case of Definition~\ref{def:linearpairing}, suppose that $\cC$ is locally $\pi$-finite and that $F: \cC \to \Vec_{k}$ is the constant functor at the one-dimensional vector space $k$. In this case, the pairing $\langle -, - \rangle_{\cC, \mathrm{const}_{k}}$ (henceforth simply denoted by $\langle -, - \rangle_{\cC}$) is a linear map 
\[ k[\pi_0 \cC_0] \otimes k[\pi_0 \cC_0] \to k
\]
on the free $k$-vector spaces $k[\pi_0 \cC_0]$ on the set $\pi_0 \cC_0$
given at a $[c] \in \cC_0$ and $[d] \in \cC_0$ by 
\[ \langle [d], [c] \rangle_{\cC} = \#^{\mathrm{tot}} \cC(c,d).
\]
In particular, if $\cC$ is a $1$-category, then $\langle [d], [c] \rangle_{\cC} = |\cC(c,d)|$ records the cardinalities of the hom sets. 
For example, for $\cC =\mathrm{FinSet}$ the category of finite sets and after identifying $\pi_0 \cC \cong \mathbb{N}_{\geq 0}$, the pairing $\langle -, -\rangle_{\mathrm{FinSet}}$ is represented by the $\mathbb{N}_{\geq 0} \times \mathbb{N}_{\geq 0}$ matrix with $(n,m)$-coefficient given by $n^m$.\end{example}

\begin{example}\label{exm:diagonalgroupoid}
If $\cC$ is a locally $\pi$-finite $\infty$-groupoid, or equivalently a locally $\pi$-finite space $X$ (i.e. a possibly infinite disjoint union of $\pi$-finite spaces), equipped with a local system $\cL:X \to \mathrm{Vec}_{k}$, then it follows from the explicit formula in Remark~\ref{rem:explicitformulapairing} that the pairing
\[\langle -, - \rangle_{X, \cL}: \left( \bigoplus_{ [x]\in \pi_0 X} \left(\cL(x)^{\vee}\right)_{\pi_1(X, x)} \right) \otimes \left( \bigoplus_{ [y]\in \pi_0 X} \cL(y)_{\pi_1(X,y)} \right)  \to k\]
is \emph{diagonal}, i.e. $$\left\langle\vphantom{\frac{a}{b}} (x\in X, \phi \in \cL(x)^\vee), (y\in X, v \in \cL(y)) \right\rangle_{X, \cL} =0 \hspace{1cm}\text{ if }[x] \neq [y] \in \pi_0 X,$$ and has diagonal entries $$\left\langle\vphantom{\frac{a}{b}} (x, \phi\in \cL(x)^\vee), (x, v\in \cL(x)) \right\rangle_{X, \cL} =  \#(X, x)^{-1}~\frac{1}{|\pi_1(X,x)|}\sum_{\gamma \in \pi_1(X, x)} \phi(\cL(\gamma)(v)).$$
As in~\eqref{eq:normmap}, we used that for any loop $\gamma \in \Omega_x X$, $\#(\Omega_x X, \gamma) = \#(X, x)^{-1} |\pi_1(X,x)|^{-1}$.
\end{example}

\begin{remark} \label{rem:pairingsubcats}If the $\infty$-category $\cC$ is not locally $\pi$-finite, the decorated span~\eqref{eq:decoratedspanpairing} is not source finite and hence cannot be linearized. However, its restriction to a pairing between certain full subcategories of $\cC$ can still be defined. Namely, if $\cA, \cB \subseteq \cC$ are full subcategories of $\cC$ such that for every $a\in \cA$ and $b \in \cB$, the mapping space $\cC(a, b)$ is $\pi$-finite, then the decorated span $(*, k) \ot (\cC_1|_{\cA, \cB}, p_F|_{\cA, \cB}) \to (\cB_0,\cL_F^\vee|_{\cB}) \times (\cA_0, \cL_F|_{\cA})$ is source finite and can hence be linearized. Here, $\cC_1|_{\cA, \cB}$ is the pullback \[\begin{tikzcd}
	{\cC_1|_{\cA, \cB}} & {\cC_1} \\
	{\cB_0 \times \cA_0} & {\cC_0 \times \cC_0}
	\arrow[from=1-1, to=1-2]
	\arrow[from=1-1, to=2-1]
	\arrow[from=2-1, to=2-2]
	\arrow[from=1-2, to=2-2, "{(t, s)}"]
	\arrow["\lrcorner"{anchor=center, pos=0.125}, draw=none, from=1-1, to=2-2]
\end{tikzcd}\] (i.e. the space of morphisms of $\cC$ whose source is in $\cA$ and whose target is in $\cB$) and $\cL_F|_{\cA}$, $\cL_F^\vee|_{\cB}$,  and $p_F|_{\cA, \cB}$  are the obvious restrictions of the local systems and map of local systems. As in Remark~\ref{rem:explicitformulapairing}, this restricted pairing unpacks to a linear map \[\langle -, - \rangle^{\cA, \cB}_{\cC;  F}: \left( \bigoplus_{ [b]\in \pi_0 \cB_0} \left(F(b)^{\vee}\right)_{\pi_0 \Aut_{\cC}(b)} \right) \otimes \left( \bigoplus_{ [a]\in \pi_0 \cA_0} F(a)_{\pi_0 \Aut_{\cC}(a)} \right)  \to k.\]
explicitly given by
\begin{equation}\label{eq:subcatpairingunpacked}
\left\langle\vphantom{\frac{a}{b}} (b\in \cB_0, \phi \in F(b)^\vee), (a\in \cA_0, v \in F(a))\right\rangle^{\cA, \cB}_{\cC, F} = \sum_{[f]\in \pi_0 \cC(a, b)} ~\#(\cC(a,b), f) ~ \phi(F(f)(v)).
\end{equation}
\end{remark}

We say that a pairing $\langle -, - \rangle : V\otimes W \to k$ between $k$-vector spaces is \emph{left non-degenerate} if the induced map $V\to \Hom(W, k)$ is injective. We say it is \emph{right non-degenerate} if the induced map $W\to \Hom(V, k)$ is injective. We say the pairing is \emph{non-degenerate} if it is both left and right non-degenerate.  Note that a pairing between finite-dimensional vector spaces is non-degenerate if and only if it is \emph{perfect}, i.e. if the induced map $V\to \Hom(W, k)$ is an isomorphism. This is not true for infinite-dimensional vector spaces.

\begin{prop}\label{prop:nondeggroupoid} 
Let $X$ be a locally $\pi$-finite space and let $\cL:X \to \Vec_{k}$ be a local system. Then, $\langle -, -\rangle_{X, \cL}$ is non-degenerate. 
\end{prop}
\begin{proof} Following Example~\ref{exm:diagonalgroupoid}, $\langle-, -\rangle_{X,\cL}$ is diagonal. Hence, it suffices to show that for every $x\in X$ the pairing $\langle-,-\rangle_x: \cL(x)^\vee_{\pi_1(X,x)} \otimes \cL(x)_{\pi_1(X,x)} \to k$ given by 
\begin{equation}\label{eq:pairinggroupoidformula}
\langle f\in \cL(x)^\vee, v\in \cL(x)\rangle_x:=  \#(X, x)^{-1} ~\frac{1}{|\pi_1(X,x)|} \sum_{\gamma \in \pi_1(X,x)} f(\cL(\gamma)v)
\end{equation}
is non-degenerate. Consider the pairings 
$$( -,-)_x: (\cL(x)^\vee)_{\pi_1(X,x)} \otimes \cL(x)^{\pi_1(X,x)} \to k\hspace{1cm} f, v \mapsto f(v)$$
 $$(-,-)^x: (\cL(x)^\vee)^{\pi_1(X,x)} \otimes \cL(x)_{\pi_1(X,x)} \to k\hspace{1cm} f,v \mapsto f(v).$$
The maps $v\mapsto (-,v)_x$ and $f\mapsto (f,-)^x$ are injective. Let $\mathrm{Nm}_{\cL, X|_x}: \cL(x)_{\pi_1(X,x)} \to \cL(x)^{\pi_1(X,x)}$ denote the norm map~\eqref{eq:normmap} of the local system $\cL$ restricted to the connected component $X|_x$ of $x\in X$. Comparing the definition of the norm map with~\eqref{eq:pairinggroupoidformula}, it follows that $\langle -, -\rangle_x = (-, \mathrm{Nm}_{\cL, X|_x} - )_x$. Since $\mathrm{Nm}_{\cL, X|_x}$ is invertible, the map $v\mapsto \langle -, v \rangle_x$ is injective. 
On the other hand, let $\mathrm{Nm}_{\cL^\vee, X|_x}: (\cL(x)^\vee)_{\pi_1(X,x)} \to (\cL(x)^\vee)^{\pi_1(X,x)}$ denote the invertible norm map of the local system $\cL^\vee$. It follows that $\langle - , - \rangle_x = (\mathrm{Nm}_{\cL^\vee, X|_x}-, - )^x$ and hence that also $f\mapsto \langle f, -\rangle_x$ is injective.\end{proof}

The linear pairing $\langle-,-\rangle_{\cC, F}$ is closely related to the linearization $\Phi_{\cC, F}$ of Definition~\ref{def:linearizationfunctor}.

\begin{lemma}\label{lem:pifinitefactpairingformula}
	Let $\cC$ be $\pi$-finite with a factorization system $(\cL, \cR)$. Then, regarding $\cL$ and $\cR$ as categories, we have that for any $F:\cC \to \Vec_k$ \begin{equation*}
		\langle -, - \rangle_{\cC, F} = \langle -, \Phi_{\cL, F}(-)\rangle_{\cR, F} = \langle \Phi_{\cR^\op, F^\vee}(-), -\rangle_{\cL, F}.
	\end{equation*} 
\end{lemma}
	\begin{proof} This follows from functoriality and monoidality of the linearization functor (Corollary~\ref{cor:spanfunctor}) and the corresponding equation of spans in the category $\mathrm{Span}(\cS^{\pi }, \Vec_k)$.
	\end{proof}

\begin{corollary}\label{cor:pairingcatvsgroupoid}
	If $\cC$ is $\pi$-finite, then $\langle -, - \rangle_{\cC, F} = \langle -, \Phi_{\cC, F}-\rangle_{\cC_0, \cL_F}$. 
	\qed
\end{corollary}
	
\begin{proof}
	Apply Lemma~\ref{lem:pifinitefactpairingformula} to the trivial factorization systems $(\cC, \cC_0)$ in which the left class are all maps in $\cC$ and the right class are just the invertible maps. 
\end{proof}

Combining Corollary~\ref{cor:factinv} with Proposition~\ref{prop:nondeggroupoid} and Corollary~\ref{cor:pairingcatvsgroupoid} immediately leads to the following corollary. 
\begin{corollary}\label{cor:nondegfinite} Let $\cC$ be a $\pi$-finite $\infty$-category equipped with a nested factorization system  such that for every $0 \leq \ell \leq n$,  endomorphism in the category $\cT^{(\ell)}$ are invertible. Then, for any functor $F: \cC \to \Vec_k$, the linear pairing $\langle -, -\rangle_{\cC, F}$ is non-degenerate. \qed
 \end{corollary}

In the remainder of this subsection, we generalize Corollary~\ref{cor:nondegfinite} to $\infty$-categories $\cC$ which are merely locally $\pi$-finite with possibly infinitely many isomorphism classes of objects. The main ingredient of this generalization will be the ability to restrict factorization systems to $\pi$-finite full subcategories.

\begin{definition} Let $\cC$ be a locally $\pi$-finite $\infty$-category with a  factorization system $(\cL, \cR)$. We call a full subcategory $\cB$ of $\cC$ \emph{factorizable} if the  factorization system restricts to $\cB$. Equivalently $\cB$ is factorizable if for every morphism $f:a \to a'$ in $\cB$, the intermediate object $c$ in the $(\cL, \cR)$-factorization $a\to c \to a'$ of $f$ remains in $\cB$.
\end{definition}	
	
\begin{lemma}\label{lem:factorizablesubcategory}
	Suppose that $\cC$ is locally $\pi$-finite with a factorization system $(\cL, \cR)$. Then any full  $\pi$-finite subcategory $\cA \subseteq \cC$ is contained in a factorizable full $\pi$-finite subcategory $\cA \subseteq \cB \subseteq \cC$. 
\end{lemma}	

\begin{proof}
	We define $\cB$ to be the full subcategory of $\cC$ on those objects $c\in \cC$ for which there exists objects $a, a' \in \cA$ and morphisms $l:a \to c$ and $r: c \to a'$ with $l \in \cL$ and $r\in \cR$. Equivalently, $\cB$ is the full subcategory on all objects which occur as the middle object of an $(\cL, \cR)$-factorization of a map between objects of $\cA$. From this second description we see that 
factorization 
		with respect to $(\cL, \cR)$ provides a map $\mathrm{interm}: (\cA)_1 \to (\cB)_0$ sending a morphism to the intermediate object in its factorization. This is surjective on $\pi_0$, showing that $\cB$ has finitely many isomorphism classes of objects.
	
In fact the category $\cB$ is factorizable, as we will now see. Let $b, b' \in \cB$ be objects, and $f:b \to b'$ a map. Let 	
\begin{equation*}
	b\stackrel{l_1}{\to} c \stackrel{r_1}{\to} b'
\end{equation*}
be the $(\cL, \cR)$-factorization of $f$ in $\cC$. By assumption there exist objects $a, a'\in\cA$, a map $l: a \to b$ in $\cL$, and a map $r:b' \to a'$ in $\cR$. Then $l_1\circ l: a \to c$ and $r\circ r_1: c \to a'$ are the $(\cL, \cR)$ factorization of the map $r \circ f \circ l: a \to a'$. Thus the object $c$ is contained in $\cB$, showing that $\cB$ is factorizable.
\end{proof}
	
\begin{lemma}\label{lem:nondegfacttonondegenwhole}
	Let $\cC$ be locally $\pi$-finite with a factorization system $(\cL, \cR)$. Suppose that
	\begin{enumerate}
		\item the pairing $\langle -, -\rangle_{\cR, F}$ is left non-degenerate;
		\item in the left class $\cL$, endomorphisms are invertible;
	\end{enumerate}
	then the pairing $\langle -, -\rangle_{\cC, F}$ is left non-degenerate.
	
	Suppose instead that
	\begin{enumerate}
		\item the pairing $\langle -, -\rangle_{\cL, F}$ is right non-degenerate;
		\item in the right class $\cR$, endomorphisms are invertible;
	\end{enumerate}
	then the pairing $\langle -, -\rangle_{\cC, F}$ is right non-degenerate.
\end{lemma}

\begin{proof}
	We will prove the first half of the lemma (regarding left non-degeneracy). The proof of the second half (right non-degeneracy) is completely analogous. Consider an arbitrary non-zero element
	\begin{equation*}
		u \in \bigoplus_{ [d]\in \pi_0 \cC_0} \left(F(d)^{\vee}\right)_{\pi_0 \Aut_{\cC}(d)}.
	\end{equation*}
	We wish to show that there is a $v \in \bigoplus_{ [c]\in \pi_0 \cC_0} \left(F(c)\right)_{\pi_0 \Aut_{\cC}(c)}$ such that $\langle u, v\rangle_{\cC, F} \neq 0$. 
	
	By assumption the pairing $\langle -, -\rangle_{\cR, F}$ is left non-degenerate, and so there exists a $w \in \bigoplus_{ [c]\in \pi_0 \cC_0} \left(F(c)\right)_{\pi_0 \Aut_{\cC}(c)}$ such that  
	\begin{equation*}
		\langle u,w \rangle_{\cR,F} \neq 0.
	\end{equation*}
	The vectors $u$ and $w$ are supported at a finite number of elements of $\pi_0\cC_0$. Since $\cC$ is locally $\pi$-finite, the full subcategory of $\cC$ on these objects is $\pi$-finite. By Lemma~\ref{lem:factorizablesubcategory} this may be enlarged to a full $\pi$-finite factorizable subcategory $\cB \subseteq \cC$, containing both the supports of $u$ and $w$. Theorem~\ref{thm:Moebius} applies to the $\pi$-finite category $\cB \cap \cL$ in which endomorphisms are invertible. Thus $\Phi_{\cB \cap \cL, F}$ is invertible. Set $v = \Phi^{-1}_{\cB \cap \cL, F}(w)$. Then, using Lemma~\ref{lem:pifinitefactpairingformula}, we have
		\begin{align*}
			\langle u,v \rangle_{\cC,F} &= \langle u,v \rangle_{\cB,F} \\ 
			&= \langle u, \Phi^{-1}_{\cB \cap \cL, F}(w)\rangle_{\cB,F} \\
			&= \langle u, \Phi_{\cB \cap \cL, F}\Phi^{-1}_{\cB \cap \cL, F}(w) \rangle_{\cR \cap B,F} \\
			& = \langle u, w \rangle_{\cR \cap B,F} \\
			&= \langle u, w \rangle_{\cR,F} \neq 0.
		\end{align*}
This establishes that $\langle -, -\rangle_{\cC, F}$ is left non-degenerate, as desired.  
\end{proof}

\begin{theorem}\label{thm:linearpairing} Let $\cC$ be a locally $\pi$-finite $\infty$-category equipped with a nested factorization system  such that for every $0 \leq \ell \leq n$,  endomorphisms in the category $\cT^{(\ell)}$ are invertible. Then, for any functor $F: \cC \to \Vec_k$, the linear pairing $\langle -, -\rangle_{\cC, F}$ is non-degenerate.
\end{theorem}

\begin{proof}
	We will prove that $\langle -, -\rangle_{\cC, F}$ is both left and right non-degenerate, separately. The proofs are similar, but dual to each other. 
	
	First, we will establish left non-degeneracy. We prove by induction that the pairing is left non-degenerate on the following sequence of subcategories:
	\begin{equation*}
		\cC_0, \cR^{(1)}, \cR^{(2)}, \dots, \cR^{(n)}, \cC.
	\end{equation*}
	For the base case we note that the restriction of $F$ to $\cC_0$ forms a local system on the locally $\pi$-finite space $\cC_0$. The non-degeneracy of $\langle - , - \rangle_{\cC_0, F}$ was established in  Prop.~\ref{prop:nondeggroupoid}, and so in particular this pairing is left non-degenerate.  
	
	On the category $\cR^{(1)}$ we have the factorization system $(\cR^{(1)}, \cC_0)$, that is the trivial factorization system in which the left class consists of all maps in $\cR^{(1)}$ and the right class consists of just the isomorphisms. We have already shown the the pairing induced by the right class is left non-degenerate, and by assumption the left class $\cR^{(1)} = \cT^{(0)}$ satisfies that endomorphisms are invertible. Thus the conditions of Lemma~\ref{lem:nondegfacttonondegenwhole} are satisfied, showing that $\langle - , - \rangle_{\cR^{(1)}, F}$ is left non-degenerate. 
	
	Similarly, the category $\cR^{(\ell)}$ admits a factorization system $(\cT^{(\ell-1)}, \cR^{(\ell-1)})$. The left class satisfies that endomorphisms are invertible by assumption, and by induction the pairing $\langle - , - \rangle_{\cR^{(\ell-1)}, F}$ is left non-degenerate. Thus Lemma~\ref{lem:nondegfacttonondegenwhole} shows that the pairing $\langle - , - \rangle_{\cR^{(\ell)}, F}$ is left non-degenerate.
	
	Finally, $\cC$ admits the factorization system $(\cL^{(n)}, \cR^{(n)})$. As we have seen, the pairing $\langle - , - \rangle_{\cR^{(n)}, F}$ is left non-degenerate, and again by assumption the left class $\cL^{(n)} = \cT^{(n)}$  satisfies that endomorphisms are invertible. Thus once again Lemma~\ref{lem:nondegfacttonondegenwhole} shows that the pairing $\langle - , - \rangle_{\cC, F}$ is left non-degenerate, as desired. 

To show that $\langle - , - \rangle_{\cC, F}$ is right non-degenerate, we may use an analogous induction on the categories 
\begin{equation*}
	\cC_0, \cL^{(n)}, \cL^{(n-1)}, \dots, \cL^{(1)}, \cC.
\end{equation*}
The pairing on $\cC_0$ is non-degenerate by Prop.~\ref{prop:nondeggroupoid}, establishing the base case. For the induction step we consider the factorization systems corresponding to each category:
\begin{align*}
	\cL^{(n)} &: (\cC_0, \cL^{(n)}) \\
	\cL^{(\ell)} &: (\cL^{(\ell + 1)}, \cT^{(\ell)}) , \quad \ell < n \\
	\cC &: (\cL^{(1)}, \cR^{(1)})
\end{align*}
In each case the right class satisfies that endomorphisms are invertible. By induction the pairing for the left class is right non-degenerate, and so Lemma~\ref{lem:nondegfacttonondegenwhole} applies. This establishes the right non-degeneracy of the pairing for the subsequent left class. The final case establishes that the pairing $\langle - , - \rangle_{\cC, F}$ is right non-degenerate, as claimed. 
\end{proof}

\begin{example} \label{exm:FinSetnondeg}The category $\mathrm{FinSet}$ of finite sets (or analogously, of finite graphs or finite groups etc.) is locally $\pi$-finite --- its hom sets are finite --- and admits a factorization system $\cL =\cT^{(1)} = \{\mathrm{surjections}\}$ and $\cR = \cT^{(0)}= \{\mathrm{injections}\}$. Since endomorphism injections and endomorphism surjections of finite sets are bijections, it follows that $\mathrm{FinSet}$ fulfills the conditions of Theorem~\ref{thm:linearpairing}. Hence, for any functor $F: \mathrm{FinSet} \to \Vec_k$, the associated pairing $\langle-, -\rangle_{\mathrm{FinSet}, F}$ is non-degenerate. In particular, for $F= \mathrm{const}_k$ the constant functor at $k$ the associated pairing $\langle n, m \rangle_{\mathrm{FinSet}} = n^m$ is indeed non-degenerate (c.f. Remark~\ref{rmk:finitesetcardN}). Categories like $\Fin\Set$ are more general than the M\"obius categories considered by \cite{MR3818099}. 
\end{example}	

\subsection{The Pontryagin pairing}\label{sec:pontryaginpairing}
 For our main application, we do not consider functors $F: \cC \to \Vec_{k}$ but rather functors $\Omega: \cC \to \mathrm{Ab}$ into the category of abelian groups. For such functors, we introduce a refinement of the pairing of Definition~\ref{def:linearpairing}. Let $k$ be an algebraically closed field of characteristic zero. For an abelian group $A$, let $\widehat{A}$ denote the `dual group' $\Hom(A, k^\times)$ of group homomorphisms, and let $k[A]$ denote the $k$-vector space obtained from linearizing the underlying set of $A$ (equivalently, the vector space of finitely supported $k$-valued functions on the underlying set of $A$). 
For a set $A$ with action by a group $G$, we write $A/G$ for its set of orbits.

\begin{definition}\label{def:pontryaginpairing}Let $\cC$ be a locally $\pi$-finite $\infty$-category, let $\Omega: \cC \to \mathrm{Ab}$ be a functor and let $k$ be an algebraically closed field of characteristic zero.
The \emph{Pontryagin pairing} is the map of $k$-vector spaces
\begin{equation}\label{eq:pontryaginpairing}
 \langle -,- \rangle_{\cC, \Omega}: \bigoplus_{[c] \in \pi_0 \cC} k[\widehat{\Omega(c)}/{\pi_0 \Aut_{\cC}(c)}] \otimes\bigoplus_{[d] \in \pi_0 \cC} k[\Omega(d)/\pi_0 \Aut_{\cC}(d)] \to k
\end{equation} with coefficient at  $\left(c\in \cC, \phi \in \widehat{\Omega(c)} = \Hom(\Omega(c), k^\times) \right)$ and $(d \in \cC, x \in \Omega(d))$ given by 
\begin{equation}\label{eq:pontryagin}
\sum_{[f] \in \pi_0 \cC(d, c)} \#\left(\cC(d,c), f\right)~\phi(\Omega(f)(x)).
\end{equation}
\end{definition}
Since the sum is over $\pi_0 \cC(c,d)$,  the expression~\eqref{eq:pontryagin} only depends on the orbits of $\phi$ and $x$ under the $\Aut_{\cC}(c)$ and $\Aut_{\cC}(d)$ action, respectively. Moreover, note that while $\widehat{\Omega(c)}$ and $\Omega(c)$ depend on a choice of basepoint $c$ in $[c]$, the set of orbits $\widehat{\Omega(c)}/\pi_0 \Aut_{\cC}(c)$ is independent of that choice up to unique isomorphism, and hence the pairing~\eqref{eq:pontryagin} only depends on the isomorphism classes of the objects $c$ and $d$. 
\begin{remark}\label{rem:subcatPontryaginpairing} As in Remark~\ref{rem:pairingsubcats} and analogously to~\eqref{eq:subcatpairingunpacked}, if $\cC$ is not locally $\pi$-finite we may still define a Pontryagin pairing $\langle -, - \rangle^{\cA, \cB}_{\cC, \Omega}$ between full subcategories $\cA, \cB \subseteq \cC$ for which for every $a\in \cA$ and $b \in \cB$ the mapping space $\cC(a, b)$ is $\pi$-finite. 
\end{remark}

\begin{corollary} \label{cor:pontryaginnondeg}
Let $\cC$ be a locally $\pi$-finite $\infty$-category with a nested factorization systems  such that for every $0 \leq \ell \leq n$,  endomorphisms in the category $\cT^{(\ell)}$ are invertible. Then, for any functor $\Omega: \cC \to \mathrm{Ab}$ and any algebraically closed field $k$ of characteristic zero, the Pontryagin pairing $\langle -, - \rangle_{\cC , \Omega}$ is non-degenerate.
\end{corollary}
\begin{proof}
For an abelian group $A$, let $\psi_A:  k[\widehat{A}] \to k[A]^\vee$ denote the evident `linearization' map from the linearization of the group of characters to the vector space dual $k[A]^\vee:= \Hom(k[A], k)$ of $k[A]$. It follows from linear independence of characters (see e.g.~\cite[Section~9.13]{Stacksproject}) that for any abelian group $A$ and any field $k$, this map $\psi_A$ is injective. For any group $G$ acting on the underlying set of $A$, the map $\psi_A$ is $G$-equivariant for the evident induced actions of $G$ on $k[\widehat{A}]$ and $k[A]^\vee$ and hence descends to a map on the coinvariants $(\psi_{A})_{G}: k[\widehat{A}]_G \to (k[A]^\vee)_{G}$. For finite $G$, using the norm map and its inverse, this map $(\psi_A)_G$ may be expressed in terms of the injective map $(\psi_A)^G:k[\widehat{A}]^G \to (k[A]^\vee)^{G}$ on the invariants and hence is again injective. 

 Since for any set $X$ with a group action, the linearization $k[X/G]$ is canonically isomorphic to the space of coinvariants $k[X]_G$, these maps assemble into an injective map $$\Psi:= \bigoplus_{[c] \in \pi_0 \cC_0} (\psi_{\Omega(c)})_{\pi_0 \Aut_{\cC}(c)} : \bigoplus_{[c] \in \pi_0 \cC_0} k[\widehat{\Omega(c)}/(\pi_0\Aut_{\cC}(c))] \to \bigoplus_{[c] \in \pi_0 \cC_0} (k[\Omega(c)]^\vee)_{\pi_0 \Aut_{\cC}(c)}$$
Let $F: \cC \to \Vec_{k}$ be the functor $k[\Omega(-)]$. Comparing~\eqref{eq:pontryagin} with the explicit expression for the linear pairing $\langle -, - \rangle_{\cC, F}$ from Remark~\ref{rem:explicitformulapairing}, it follows that 
\[ \langle x, y \rangle_{\cC, \Omega} = \langle \Psi x , y \rangle_{\cC, F}.
\]
Since $\Psi$ is injective,  it follows from Theorem~\ref{thm:linearpairing} that $x\mapsto \langle x , - \rangle_{\cC, \Omega}$ is injective.

To prove injectivity in the other slot, note that for any abelian group $A$ and any algebraically closed field $k$, the canonical map $ A\to \widehat{\widehat{A\, }}$ is injective.\footnote{We may see this as follows. Let $a\neq 1\in A$. Pick a character $\chi: \langle a \rangle \to k^\times$ on the cyclic group generated by $a$, such that $\chi(a) \neq 1$ (this exists because $k$ is algebraically closed). Then, since $k^\times$ is an injective abelian group, any homomorphism $\langle a \rangle \to k^\times$ may be lifted along the injection $\langle a \rangle \to A$ .}

 Let $\theta_A: k[A] \to k[\widehat{\widehat{A\, }}] \to \Hom(k[\widehat{A}], k)$ denote the composite of this injective map with the injective map $\phi_{\widehat{A}}$. As before, we may assemble these maps into an injective map
\[\Theta:= \bigoplus_{[d] \in \pi_0 \cC_0} (\theta_{\Omega(d)})_{\pi_0 \Aut_{\cC}(d)} :\bigoplus_{[d] \in \pi_0 \cC} k[\Omega(d)/\pi_0 \Aut_{\cC}(d)] \to \bigoplus_{[d] \in \pi_0 \cC} (k[\widehat{\Omega(d)}]^\vee)_{\pi_0 \Aut_{\cC}(d)}
\] Defining $G: \cC^{\op} \to \Vec_{k}$ to be the functor $k[\widehat{\Omega(-)}]$, comparing~\eqref{eq:pontryagin} with Remark~\ref{rem:explicitformulapairing} it follows that
\[\langle x,y  \rangle_{\cC, \Omega} = \langle \Theta y, x \rangle_{\cC^{\op}, G}.
\]
Since $\cC^{\op}, G$ also fulfills the requirements of Theorem~\ref{thm:linearpairing}, it follows from injectivity of $\Theta$ that $y \mapsto \langle -, y \rangle_{\cC, \Omega}$ is also injective. 
\end{proof}

\section{Bordism Groups and Invertible Topological Field Theories}\label{sec:bordandinvertTFT}

\subsection{Stable structures and their normal and tangential bordism groups}\label{sec:stabletangstablenorm}

Classical bordism theory, \emph{\`a la} Ren\'e Thom, considers structures on the \emph{stable normal bundle}. Given a space $B$ with a vector bundle\footnote{More generally $\beta$ only needs to be a \emph{stable} vector bundle on $B$ - a filtration of $B$ with vector bundles on each piece, but where the restriction maps are stable isomorphisms.} $\beta$ a \emph{(stable) normal $(B, \beta)$-structure} on $M$ is a pair $(f, \theta)$ of a map $f: M \to B$ and a \emph{stable isomorphism}
\begin{equation*}
	\theta: \varepsilon^{N} \cong \tau_M \oplus f^*\beta \oplus \varepsilon^{N-d -k}
\end{equation*}
where $N$ is large, $d = \dim M$, $k$ is the rank of $\beta$, and $\varepsilon$ is the trivial bundle of rank one. 

Once one fixes the classifying maps of the stable normal bundle $\nu_M: M \to \rB O$ and $\beta: B \to \rB O$, then the space of $(B,\beta)$-structures on $M$ is equivalent to the space of lifts (in the homotopical sense)
\[
\begin{tz}
	\node (B) at (2.5, 1.5) {$ B$};
	\node (C) at (0, 0) {$ M $};
	\node (D) at (2.5, 0) {$ \rB O $};
	\draw [->, dashed] (C) to node [above left] {$f$}  (B); 
	\draw [->] (B) to node [right] {$\beta$} (D); 
	\draw [->] (C) to node [below] {$\nu_M$} (D); 
	\node at (2, 0.5) {$\simeq \theta$};	
\end{tz}
\]
of $\nu_M$ through the map $\beta$. Notice, however, that this equivalence depends on the choice of classifying map (which is unique up to homotopy). We will henceforth use these two perspectives interchangeably.

If $M_0$ and $M_1$ are closed $n$-manifolds equipped with normal $(B, \beta)$-structures $(f_0, \theta_0)$ and $(f_1, \theta_1)$, then a normal $(B, \beta)$-bordism from $M_0$ to $M_1$ is a $d$-dimensional bordism $W$ from $M_0$ to $M_1$ together with a normal $(B, \beta)$-structure $(f, \theta)$ on $W$ which \emph{extends} the ones on $M_0$ and $M_1$. Normal bordism is an equivalence relation and the equivalence classes form a group which is usually denoted $\Omega_n^B$.

A similar theory can be developed using the \emph{stable tangent bundle} in place of the stable normal bundle. Given a smooth manifold $M$ of dimension $n$ we fix a map $T_M: M \to \rB O(n)$ classifying the tangent bundle of $M$. The stable tangent bundle is the composite map $\tau_M: M \stackrel{T_M}{\to} \rB O(n) \stackrel{i_n}{\to} \rB O$. A \emph{stable tangential $(B, \beta)$-structure} on $M$ is a lift (in the homotopical sense)
\[
\begin{tz}
	\node (B) at (2.5, 1.5) {$ B$};
	\node (C) at (0, 0) {$ M $};
	\node (D) at (2.5, 0) {$ \rB O $};
	\draw [->, dashed] (C) to  (B); 
	\draw [->] (B) to node [right] {$\beta$} (D); 
	\draw [->] (C) to node [above] {$\tau_M$} (D); 	
\end{tz}
\]
of $\tau_M$ through the map $\beta$. Bordisms are defined similarly to the normal case, and so a priori we have two notions of stable bordism groups:

\begin{definition}\label{def:stablebordism}
	Let $\beta: B \to \rB O$ be a stable structure.  The \emph{stable tangential bordism group} $\Omega_d^{\tau \beta}$ is the bordism group of $d$-dimensional manifolds equipped with $(B, \beta)$-structures on their \emph{stable tangent bundle} modulo the relation of cobordism with $(B, \beta)$-structures on their \emph{stable tangent bundle}. The \emph{stable normal bordism group} $\Omega_d^{\nu \beta}$ is the bordism group of $d$-dimensional manifolds equipped with $(B, \beta)$-structures on their \emph{stable normal bundle} modulo the relation of cobordism with $(B, \beta)$-structures on their \emph{stable normal bundle}. 
\end{definition}

However, stable tangential bordism does not actually provide a new notion. Given a stable structure $\beta: B \to \rB O$, we obtain a new stable structure $\overline{\beta}:B \to \rB O$ as the composite $B \to \rB O \to[(-1)] \rB O$, where $(-1): \rB O \simeq \rB O$ corresponds to inversion in the infinite loop space structure on $\rB O$ (which may be thought of as taking the `orthogonal complement' of a stable vector bundle).
 An $(B, \overline{\beta})$-structure on the stable normal bundle of $M$ is the same thing as an $(B,\beta)$-structure on the stable tangent bundle of $M$, and conversely.
\begin{equation*}
	\Omega_d^{\tau \beta} \cong \Omega_d^{\nu \overline{\beta}} \qquad \Omega_d^{\nu \beta} \cong \Omega_d^{\tau \overline{\beta}}.
\end{equation*}

\begin{remark}
	Sometimes it is the case that $\overline{\beta} \simeq \beta$. For example, this holds for orientations and spin structures - an orientation (or spin structure) on a manifold can equivalently be defined on the stable normal bundle or the (stable) tangent bundle. An example where $\overline{\beta} \not\simeq \beta$ is given by $Pin^\pm$-structures. Instead, the operation $(B, \beta) \mapsto (B, \overline{\beta})$ exchanges stable $Pin^+$ and $Pin^-$ structures~\cite[Lem.~1.3]{MR1171915}. 
\end{remark}

\begin{example}\label{ex:normalntype}
	Let $M$ be a smooth manifold, with stable normal map $\nu_M: M \to \rB O$, and fix a natural number $n$. The map $\nu_M$ factors uniquely up to homotopy $M \to B_M \to \rB O$ into a $n$-connected map followed by an $n$-truncated map (see Example~\ref{exm:truncatedconnected}). The space $B_M$ is the \emph{(stable) normal $n$-type} of $M$ \cite{MR1709301}. The map from $M$ to $B_M$ realizes $M$ as an element of $\Omega_n^{\nu B_M}$, the normal $B_M$-bordism group which appears in Kreck's modified approach to surgery. See \cite{MR1709301} for details. 
	Similarly, the \emph{stable tangential $n$-type} of $M$ is obtained as the $n$-connected/$n$-truncated factorization of the stable tangent bundle $\tau_M:M \to \rB O$. 
\end{example}

\subsection{Unstable tangential structures, bordism categories and topological field theories} \label{sec:unstabletangstr}

By an \emph{unstable tangential structure} we mean a rank $d$ vector bundle  $\beta$ on a space $B$. Given a smooth manifold $M$ of dimensions $\dim(M) \leq d$, a $(B, \beta)$-structure on $M$ consists of a pair $(f, \theta)$ of a map $f: M \to B$ and a vector bundle isomorphism $\theta: T_M \oplus \varepsilon^{d- \dim(M)} \cong f^* \beta$, where $\varepsilon$ is the trivial one-dimensional vector bundle. For manifolds of dimension strictly less than $d$, there is a notion of $(B,\beta)$-bordism. However, in contrast to the stable case, the relation of $(B,\beta)$-bordism may fail to be reflexive and hence can fail to be an equivalence relation. However bordisms can be composed, and so we will obtain a bordism \emph{category}. This can be thought of as a substitute in the unstable setting. 

 If $M_0$ and $M_1$ are closed $(d-1)$-manifolds equipped with $(B, \beta)$-structures $(f_0, \theta_0)$ and $(f_1, \theta_1)$, then a $(B, \beta)$-bordism from $M_0$ to $M_1$ is a $d$-dimensional compact bordism $W$ from $M_0$ to $M_1$ together with a $(B, \beta)$-structure $(f, \theta)$ on $W$ and a choice of an \emph{inward pointing normal vector} on the incoming boundary $M_0$, and the \emph{outward pointing normal} on the outgoing boundary $M_1$. These latter structures are understood as appropriate isomorphisms $(T_{M_i} \oplus \varepsilon) \cong T_W|_{M_i}$. Furthermore the  $(B, \beta)$-structure $(f, \theta)$ on $W$  \emph{extends} the ones on $M_0$ and $M_1$. The meaning of `extends' in this context is first that $f: W \to B$ restricts to $f_0$ and $f_1$ on $M_0$ and $M_1$, and furthermore for $i=1,2$, the vector bundle isomorphism $\theta_i: T_{M_i} \oplus \epsilon \cong f_i^* \beta$ agrees with 
 \begin{equation*}
	(T_{M_i} \oplus \varepsilon) \cong T_W|_{M_i}  \stackrel{\theta|_{M_i}}{\cong} f_i^* \beta.
\end{equation*}

Bordisms may be composed giving rise to a symmetric monoidal bordism category of manifolds with $(B,\beta)$-structure $\Bord_d^{B, \beta}$. The objects of this category are closed $(d-1)$-manifolds with $(B,\beta)$-structures, and the morphisms are equivalence classes of compact $d$-bordisms with $(B,\beta)$-structures. The equivalence relation on bordisms is generated by isotopy of $(B, \beta)$-structure together with diffeomorphism rel. boundary over $B$.

\begin{definition}
A \emph{$(B, \beta)$-structured topological field theory} valued in a symmetric monoidal category $\cC$ is a symmetric monoidal functor $Z:\Bord_d^{B, \beta} \to \cC$. A $(B, \beta)$-structured \emph{topological super field theory} is a topological field theory valued in the symmetric monoidal category $\sVec$ of super vector spaces.\end{definition}

\begin{remark}\label{rmk:bordasinftyone}
The category $\Bord_d^{B, \beta}$ is the homotopy $1$-category of an $(\infty, 1)$-category of $(B, \beta)$-bordisms (for example constructed as a topological category in \cite{MR2653727, MR2506750}) which also contains information about the moduli spaces of $B$-manifolds and $B$-bordisms. There are also $(\infty,n)$-versions of the bordisms $n$-category considered in \cite{MR2555928,MR3924174,Schommer-Pries:2017aa}.
\end{remark}

\begin{remark}
Fixing a classifying map $\beta: B \to \rB O(d)$ of the vector bundle $(B, \beta)$, given a smooth manifold $M$ of dimension $\dim(M) \leq d$ and fixing a classifying map $T_M: M \to \rB O(\dim(M))$ of its tangent bundle, a $(B, \beta)$-structure may equivalently be given by a lift (in the homotopical sense)
\begin{equation}\label{eq:unstablestructure}
\begin{tz}
	\node (B) at (2.5, 1.5) {$ B $};
	\node (C) at (-2.5, 0) {$ M $};
	\node (E) at (0,0) {$\rB O(\dim(M))$};
	\node (D) at (2.5, 0) {$ \rB O(d) $};
	\draw [->, dashed] (C) to (B); 
	\draw [->] (B) to node [right] {$\beta$} (D); 
	\draw [->] (C) to node [below] {$T_M$} (E); 	
	\draw [->] (E) to  (D); 	
\end{tz}
\end{equation}
 of the composite $M \to \rB O(\dim(M)) \to \rB O(d) $ through $\beta$. In particular, the mapping space $\Map_{\rB O(d)} (M, B)$ in the $(\infty,1)$-category $\cS_{/\rB O(d)}$ of spaces over $\rB O(d)$ is the \emph{moduli space of $(B, \beta)$-structures on $M$}. (If $\beta:B \to \rB O(d)$ is a fibration, this is the usual point-set mapping space.) The correspondence between these two perspectives on $(B, \beta)$-structures depends on a choice of these classifying maps (which are unique up to homotopy).
 
 \end{remark}

\subsection{Invertible topological field theories}\label{sec:invertTFT}
A topological super field theory is \emph{invertible} if it factors through the groupoid $\mathrm{sLine}_k$ whose objects are super lines (either odd or even) and whose morphisms are isomorphisms of super lines. In the following section, we briefly summarize the classification of such invertible super field theories. Variations of these results have for example appeared in \cite{MR4268163, MR4313235}

A symmetric monoidal category in which all morphisms are invertible and in which all objects are $\otimes$-invertible is called a \emph{Picard groupoid}. Every symmetric monoidal category $\cC$ has a maximal Picard subgroupoid $\Pic(\cC)$, and in particular any symmetric monoidal functor from a Picard groupoid into a symmetric monoidal category $\cC$ factors (essentially uniquely) through $\Pic(\cC)$. Conversely, a symmetric monoidal functor from $\cC$ into a Picard groupoid factors through the \emph{Picard completion} $\widehat{\cC}$ of $\cC$, which is obtained by freely inverting the morphisms, and then performing a group-completion on the objects. In other words, forming the Picard subgroupoid and the Picard completion are right and left adjoints to the inclusion of the $2$-category of Picard groupoids into the $2$-category of symmetric monoidal $1$-categories. Equivalently, the Picard completion may be constructed by first observing that the geometric realization $\|\cC\|$ of $\cC$ is an $E_\infty$-space, and then defining $\widehat{\cC}$ to be the fundamental groupoid $\pi_{\leq 1} \Omega \rB \|\cC\|$ of the group-completion of $\|\cC\|$. If the category $\cC$ has duals, then $\|\cC\|$ is already grouplike and $\Omega \rB \|\cC\| \simeq \|\cC\|$. For the bordism category the geometric realization is identified by the celebrated theorem of Galatius-Madsen-Tillmann-Weiss \cite{MR2506750} (see also \cite{MR2653727}).

\begin{theorem}[{\cite{MR2506750,MR2653727}}]
Let $\beta:B \to \rB O(d)$ be an unstable tangential structure. Then, 
	$\|\Bord_d^{B, \beta}\| \simeq \Omega^{\infty-1}MT\beta$, where the \emph{Madsen-Tillmann}  spectrum  $MT \beta$ associated to $\beta:B \to \rB O(d)$  is the Thom spectrum $ B^{- \beta}$ of the virtual vector bundle $- \beta$.  \qed
\end{theorem}

Picard groupoids $\sA$ are classified by the following three \emph{Postnikov invariants} \cite[App.~B]{MR2192936} (see also \cite{MR2981952}):
\begin{itemize}
	\item An abelian group $A_0$, corresponding to the isomorphism classes of objects of $\sA$;
	\item An abelian group $A_1$, corresponding to the automorphisms of the unit object of $\sA$;
	\item A homomorphisms $k_\sA: A_0 \otimes \ZZ/2\ZZ \to A_1$, the \emph{$k$-invariant} of $\sA$.
\end{itemize}
The map $k_\sA$ is defined as follows. For each object $x \in \sA$ select a dual-inverse $\overline{x}$, equipped with unit $\eta:\mathbbm{1} \to x \otimes \overline{x}$ and counit $\varepsilon:\overline{x} \otimes x \to \mathbbm{1}$ isomorphisms, which satisfy the usual zig-zag equations
\begin{equation*}
		id_x = (id_x \otimes \varepsilon) \circ  (\eta \otimes id_x), \quad id_{\overline{x}} =  (  \varepsilon \otimes id_{\overline{x}}) \circ  (  id_{\overline{x}} \otimes\eta).
\end{equation*}
Define $k_\sA(x): \mathbbm{1} \to \mathbbm{1}$ as
\begin{equation*}
k_\sA(x) = \varepsilon \circ \beta_{ x, \overline{x}} \circ \eta
\end{equation*}
where $\beta_{x, \overline{x}}: x \otimes \overline{x} \to \overline{x} \otimes x$ is the braiding isomorphism of $\sA$. The morphism $k_\sA(x): \mathbbm{1} \to \mathbbm{1}$ only depends on the isomorphism class of $x$.  Since $\sA$ is symmetric monoidal, $k_A(x) \in \End_{\sA}(\mathbbm{1}) = A_1$ is an order 2 element and $k_\sA: A_0 \otimes \ZZ/2\ZZ \to A_1$ defines a homomorphism.

\begin{example}\label{ex:slines}
The Picard groupoid $\Pic(\Vec_k)$ of the category of vector spaces is the groupoid $\Line_k$ of one-dimensional vector spaces and isomorphisms between them. Its Postnikov invariants are $L_0 = 0$, $L_1 = k^{\times}$ and trivial $k$-invariant. The Picard groupoid $\Pic(\sVec_k)$ of the category of super vector spaces is the groupoid $ \sLine_k$ of super lines (even and odd). Its Postnikov invariants are $sL_0 = \ZZ/2\ZZ$, $sL_1 = k^{\times}$, and the $k$
-invariant is the homomorphisms $k_{\sLine_k}: \ZZ/2\ZZ \to k^{\times}$ which sends the non-trivial element to $-1 \in k^{\times}$. In particular, if the characteristic of $k$ is not $2$, $k_{\sLine_k}$ is injective. \end{example}

\begin{example}[{\cite[App.~B]{MR2192936}}] \label{ex:picardfromspectra}
	Given a spectrum $E$ and an integer $i$, we obtain a Picard groupoid $\pi_{\leq 1} \Omega^{\infty+i}E$, the fundamental groupoid of the shifted infinite loop space underlying $E$. It has Postnikov invariants:
	\begin{itemize}
		\item the isomorphism classes of objects are $\pi_i E$;
		\item the automorphisms of the unit object are $\pi_{i+1}E$;
		\item the $k$-invariant $k_E = (\eta \cdot -):\pi_i E \to \pi_{i+1}E$ is given by multiplication by $\eta \in \pi_1 \mathbb{S}$ in the sphere spectrum. 
	\end{itemize}
	All Picard groupoids arise in this way. (Indeed, the $2$-groupoid of Picard groupoids is equivalent to the moduli space of spectra $E$ with $\pi_i E =0 $ for $i \neq 0,1$.)
\end{example}

Symmetric monoidal functors between Picard groupoids $\sA$ and $\sB$ can be described in terms of their Postnikov invariants (see~\cite{MR2981952}, also cf.~\cite{stolz-appendix}). Specifically, there is an exact sequence:
\begin{equation*}
	0 \to \Ext(A_0, B_1) \to \pi_0 \Fun(\sA, \sB) \to \Hom(A_0, B_0) \times \Hom(A_1, B_1) \to \Hom(A_0 \otimes \ZZ/2\ZZ, B_1)
\end{equation*}
where the last map sends $(f,g) \in \Hom(A_0, B_0) \times \Hom(A_1, B_1)$ to $k_\sB \circ f - g \circ k_\sA$.

In particular, if $B_1$ is a divisible group and the map $B_0 \to B_0 \otimes \ZZ/2\ZZ \to B_1$ induced by the $k$-invariant of $\sB$ is injective, then $\pi_0 \Fun(\sA, \sB) \cong \Hom(A_1, B_1)$.

\begin{corollary}\label{cor:universalpropertysline}
Let $k$ be an algebraically closed field of characteristic $\neq 2$. Then, for all Picard groupoids $\sA$:
\begin{align*}
	\pi_0 \Fun^{\otimes}(\sA, \sLine_k) &\cong \Hom(A_1, k^{\times}) 
	\qedhere \qed
\end{align*}

\end{corollary}

Thus, symmetric monoidal functors into the Picard groupoid of super lines (over an algebraically closed field, not of characteristic 2) are the same as characters on the automorphisms of the unit object, see \cite[Rmk.~6.91]{MR3969923}.

\begin{remark}\label{rmk:BrownComentz}
	The Picard groupoid $\sLine_k$ also arises from a spectrum as in Example~\ref{ex:picardfromspectra}. Specifically  let  $I_{k^{\times}}$ be the $k^{\times}$-based version of the Brown-Comenetz dual of the sphere spectrum \cite{MR405403}. Since $k^{\times}$ is an injective abelian group, this spectrum has homotopy groups $\pi_i I_{k^{\times}} \cong \hom(\pi_{-i}\mathbb{S}, k^{\times})$. Thus if the characteristic of $k$ is not $2$, $\pi_{-1}I_{k^{\times}} \cong \ZZ/2\ZZ \cong sL_0$ and $\pi_0I_{k^{\times}} \cong k^{\times} \cong sL_1$. The corresponding $k$-invariants coincide, and the universal property of $\sLine_k$ follows from the universal property of $I_{k^{\times}}$. See \cite{MR405403} and \cite[App.~B]{MR2192936}.
\end{remark}

\begin{definition}\label{def:tangentialbordism}
	Given an unstable structure $\beta: B \to \rB O(d)$, we write $\Omega_k^{T\beta}:= \Omega_k^{\tau (i_d \circ \beta)}$  where $i_d: \rB O(d)\to \rB O$ is the stabilization map. We refer to this as the \emph{semistable\footnote{We use the term `semistable' since the elements of this bordism group are $d$-manifolds equipped with a stable structure obtained by stabilizing the unstable structure $(B, \beta)$.} bordism group}.  
\end{definition}

Explicitly, given an unstable structure $\beta: B \to \rB O(d)$ thought of as a $d$-dimensional vector bundle $\beta$ on $B$, a stable $(B, i_d \circ \beta)$-structure on a manifold $M$ is a map $f:M \to B$ and an isomorphism 
$$TM \oplus \varepsilon^{N} \cong f^* \beta \oplus \varepsilon^{N + \dim(M) - d}$$ for some $N \gg 0$.

\begin{remark}
For $\beta: \rB SO(d) \to \rB O(d)$, the semistable bordism group $\Omega_d^{T\beta}=  \Omega_d^{T\rB SO(d)}$ has been widely studied. It was shown in~\cite{MR3212580} to be isomorphic to Reinhart's vector field cobordism group \cite{MR153021}, which in turn agrees with the SKK groups of oriented manifolds~\cite{MR0362360}. 

In fact, in forthcoming work, Kreck, Stolz and Teichner show that for an arbitrary structure $(B, \beta: B \to \rB O(d))$ and in all dimensions $d \neq 2$, the  semistable bordism group $\Omega_k^{T\beta}$ agrees with an appropriately defined SKK group of closed $d$-manifolds with $(B, \beta)$-structure. \end{remark}

Note that any (unstable) $(B, \beta)$-structure on a closed manifold (with $\dim(M) \leq d$) as in~\eqref{eq:unstablestructure} induces a stable tangential $(B, i_d\circ \beta)$-structure and hence an element $[M] \in \Omega_n^{T\beta}$. Furthermore, it follows from the Pontryagin-Thom construction that the homotopy groups $MT \beta$ can be identified with these semistable bordism groups~\cite{179327}.

\begin{lemma}\label{lem:homotopyisbordism}
Given an unstable structure $\beta: B \to \rB O(d)$, 
	$\pi_{n-d} MT \beta \cong \Omega^{T\beta}_n.$ \qed
\end{lemma}

Combining this computation with the theorem of  Galatius-Madsen-Tillmann-Weiss\\ \cite{MR2506750}, we obtain the following proposition which asserts that invertible super topological field theories are uniquely characterized by their \emph{partition functions}.

\begin{theorem}\label{thm:invertarechars} Let $\beta:B \to \rB O(d)$ be an unstable structure and let $k$ be an algebraically closed field of characteristic $\neq 2$. Then, isomorphism classes of $d$-dimensional invertible $(B, \beta)$-structured super topological field theories $\cZ:\Bord_d^{B, \beta} \to \sVec_k$ are in natural bijection with group homomorphisms $\omega: \Omega^{T\beta}_d \to k^{\times}$.

Explicitly, the value of the TFT $\cZ_{\omega}$ associated to a group homomorphism $\omega$ on a closed $d$-dimensional $(B, \beta)$-manifold $W$ is given by 
\[\cZ_{\omega}(W)= \omega([W])\]
where $[W] \in \Omega^{T\beta}_d$ is the bordism class of $W$ with stable tangential $(B, \beta)$-structure induced from its given unstable structure.
\end{theorem}
\begin{proof}The bordism category $\Bord_d^{(B, \beta)}$ has duals, and thus the Picard completion $\widehat{\Bord_d^{(B, \beta)}}$ is the fundamental groupoid of the classifying space $\pi_{\leq 1} \|\Bord_d^{(B, \beta)}\|$.
In their famous work, Galatius, Madsen, Tillmann, and Weiss consider a category $\Cob^{(B, \beta)}_d$ enriched in topological spaces. They identify the homotopy type of the geometric realization as $||\Cob^{(B, \beta)}_d|| \simeq \Omega^{\infty -1} MT\beta$ \cite{MR2506750}. Our category $\Bord_d^{(B, \beta)}$ is the homotopy category of $\Cob_d^{(B, \beta)}$ and hence the canonical map $\Cob_d^{(B, \beta)} \to \Bord_d^{(B, \beta)}$ induces an equivalence on the fundamental $1$-groupoid of classifying spaces $$\pi_{\leq 1}\|\Cob_d^{(B, \beta)}\| \cong  \pi_{\leq 1}\|\Bord_d^{(B, \beta)}\|.$$ Hence, it follows that $\pi_1(\widehat{\Bord_d^{(B, \beta)}}) \cong \pi_0 MT\beta \cong \Omega_d^{T\beta}$. The proposition then follows from Corollary~\ref{cor:universalpropertysline}.
\end{proof}

\subsection{Computing semistable bordism groups}\label{sec:unstablebordism}
In this section, we explain how to compute semistable bordism groups from Definition~\ref{def:tangentialbordism} in terms of the more familiar stable tangential bordism groups from Definition~\ref{def:stablebordism}.

Let $\beta_{\infty}: B_{\infty} \to \rB O$ be a stable tangential structure and for each $k\geq 0$, consider the unstable tangential structure $\beta_k: B_k \to \rB O(k)$ defined as the pullback:
\begin{center}
\begin{tikzpicture}
	\node (A) at (0,1.5) {$ B_k $};
	\node (B) at (2.5, 1.5) {$ B_{\infty} $};
	\node (C) at (0, 0) {$ \rB O(k) $};
	\node (D) at (2.5, 0) {$ \rB O $};
	\draw [->] (A) to (B); 
	\draw [->] (A) to node [left] {$\beta_k$} (C); 
	\draw [->] (B) to node [right] {$\beta_\infty$}(D); 
	\draw [->] (C) to (D); 
	\node at (0.5, 1) {$\lrcorner$};	
\end{tikzpicture}
\end{center} 

\begin{example} Consider $\beta_{\infty}: \rB SO \to \rB O$. Then, $\beta_k: \rB SO(k) \to \rB O(k)$, and similarly for $\rB Spin$ and all other structures in the Whitehead tower of $\rB O$.  \end{example}

Recall from Lemma~\ref{lem:homotopyisbordism} that for each $n\geq 0$ the semistable bordism groups $\Omega_n^{T\beta_k} = \Omega_n^{\tau i\circ \beta_k}$ can be identified with homotopy groups of the Madsen-Tillmann spectrum $MT\beta_k$. 
We can compare these various bordism groups $\Omega_n^{T\beta_k}$ for varying $n$ and $k$ using the \emph{Genauer fiber sequence}:
	\begin{equation*}
		MT\beta_{k+1} \to \Sigma^\infty_+ B_{k+1} \to MT\beta_{k}.
	\end{equation*}
The existence of this fiber sequence of spectra is well-known to experts. Genauer constructed a $(-1)$-connective cover of this sequence using a bordism category of manifolds with boundary to model 
$\Sigma^\infty_+ B_{k+1} $	\cite{MR2833590}. It is treated in detail in the unoriented case in \cite[Prop.~3.1]{MR2506750} and the case of general tangential structures is discussed in \cite[Section~5]{MR2506750}. It is also mentioned in \cite[Section~6.1]{MR3158761}.  

\begin{lemma}\label{lem:TangetBordismIsstablebordismplusEuler}
	Let $\beta_{\infty}: B_{\infty} \to \rB O$ be a stable tangential structure and for each $k\geq 0$, let $\beta_k: B_k \to \rB O(k)$ denote the unstable tangential structure obtained by pulling back $\beta_{\infty}$ to $\rB O(k)$. 
	\begin{enumerate}
		\item Let $d < k$. Then $\Omega_d^{T\beta_k} \cong \Omega_d^{T\beta_{k+1}} \cong \Omega_d^{\tau \beta_\infty} \cong \Omega_d^{\nu \overline{\beta}_\infty}$.
		\item There is an exact sequence
		\begin{equation*}
			\Omega_{k+1}^{T\beta_{k+1}} \to \pi_0\Sigma^\infty_+ B_{k+1} \to \Omega_k^{T\beta_k} \to \Omega_k^{T\beta_{k+1}} \to 0.
		\end{equation*}
	\end{enumerate}
Suppose that $B_k$ is connected.
\begin{enumerate}
	\item [3.]  If $k$ is odd, then this last exact sequence reduces to a short exact sequence
\begin{equation*}
	0 \to \ZZ/\ell \to \Omega_k^{T\beta_k} \to \Omega_k^{\tau \beta_\infty} \to 0 
\end{equation*}
for some $\ell \in \{ 0, 1, 2, \dots \}$.
\item [4.] If $k$ is even, there is a map of short exact sequences
\begin{center}
\begin{tikzpicture}
	\node (A1) at (0,1.5) {$ 0 $};
	\node (A2) at (2.5, 1.5) {$ \ZZ  $};
	\node (A3) at (5, 1.5) {$ \Omega_k^{T\beta_k} $};
	\node (A4) at (7.5, 1.5) {$ \Omega_k^{\tau \beta_\infty} $};
	\node (A5) at (10, 1.5) {$ 0 $};

	\node (B1) at (0,0) {$ 0 $};
	\node (B2) at (2.5, 0) {$ \ZZ  $};
	\node (B3) at (5, 0) {$ \ZZ $};
	\node (B4) at (7.5, 0) {$ \ZZ/2\ZZ $};
	\node (B5) at (10, 0) {$ 0 $};

	\draw [->] (A1) to (A2);
	\draw [->] (A2) to (A3);
	\draw [->] (A3) to (A4);
	\draw [->] (A4) to (A5);
 	\draw [->] (B1) to  (B2);
 	\draw [->] (B2) to node [above] {$2 \cdot$} (B3);
 	\draw [->] (B3) to (B4);
 	\draw [->] (B4) to (B5);
 
	\draw [->] (A2) to node [right] {$\cong$} (B2);
	\draw [->] (A3) to node [right] {$\chi$} (B3);
	\draw [->] (A4) to node [right] {$\chi$ mod 2}(B4);
\end{tikzpicture}
\end{center}
where $\chi$ is the Euler characteristic function. 
\end{enumerate}
\end{lemma}

\begin{proof}
	(1), (2), and (3) are direct consequences of the Genauer fiber sequence. We can also compare the Genauer fiber sequence for $(B_k, \beta_k)$ to the sequence for $(BO(k), id)$. We obtain a map of long exact sequences:
	\begin{center}
	\begin{tikzpicture}
		\node (A1) at (0,1.5) {$ \Omega_{{k+1}}^{T\beta_{k+1}} $};
		\node (A2) at (2.5, 1.5) {$ \ZZ  $};
		\node (A3) at (5, 1.5) {$ \Omega_k^{T\beta_k} $};
		\node (A4) at (7.5, 1.5) {$ \Omega_k^{\tau \beta_\infty} $};
		\node (A5) at (10, 1.5) {$ 0 $};
	
		\node (B1) at (0,0) {$ \Omega_{{k+1}}^{T\rB O(k+1)} $};
		\node (B2) at (2.5, 0) {$ \ZZ  $};
		\node (B3) at (5, 0) {$ \Omega_k^{T\rB O(k)} $};
		\node (B4) at (7.5, 0) {$ \cN_k $};
		\node (B5) at (10, 0) {$ 0 $};
	
		\draw [->] (A1) to (A2);
		\draw [->] (A2) to (A3);
		\draw [->] (A3) to (A4);
		\draw [->] (A4) to (A5);
	 	\draw [->] (B1) to node [above] {$\chi$} (B2);
	 	\draw [->] (B2) to (B3);
	 	\draw [->] (B3) to (B4);
	 	\draw [->] (B4) to (B5);	
	 
		\draw [->] (A1) to (B1); 
		\draw [->] (A2) to node [right] {$\cong$} (B2);
		\draw [->] (A3) to (B3);
		\draw [->] (A4) to (B4);
	\end{tikzpicture}
	\end{center}
	Here $\cN_k$ is the unoriented bordism group, and the first lower horizontal map $\chi$ is identified with the Euler characteristic \cite[Eq.~3.6]{MR2506750} (see also \cite{MR3356279}). Thus the $\ell$ in part (3) is the smallest common divisor of the possible Euler characteristics of $(k+1)$-dimensional manifolds with $(B_{k+1}, \beta_{k+1})$-structures. 
	When $k$ is even (so $k+1$ is odd), these Euler characteristics will be zero, and in that case we can extend to a diagram of short exact sequences as follows.
	\begin{center}
	\begin{tikzpicture}
		\node (A1) at (0,1.5) {$ 0 $};
		\node (A2) at (2.5, 1.5) {$ \ZZ  $};
		\node (A3) at (5, 1.5) {$ \Omega_k^{T\beta_k} $};
		\node (A4) at (7.5, 1.5) {$ \Omega_k^{\tau \beta_\infty} $};
		\node (A5) at (10, 1.5) {$ 0 $};
	
		\node (B1) at (0,0) {$ 0 $};
		\node (B2) at (2.5, 0) {$ \ZZ  $};
		\node (B3) at (5, 0) {$ \Omega_k^{T\rB O(k)} $};
		\node (B4) at (7.5, 0) {$ \cN_k $};
		\node (B5) at (10, 0) {$ 0 $};
	
		\node (C1) at (0, -1.5) {$ 0 $};
		\node (C2) at (2.5, -1.5) {$ \ZZ $};
	\node (C3) at (5, -1.5) {$ \ZZ $};
	\node (C4) at (7.5, -1.5) {$ \ZZ/2\ZZ $};
	\node (C5) at (10, -1.5) {$ 0 $};

		\draw [->] (A1) to (A2);
		\draw [->] (A2) to (A3);
		\draw [->] (A3) to (A4);
		\draw [->] (A4) to (A5);
	 	\draw [->] (B1) to  (B2);
	 	\draw [->] (B2) to (B3);
	 	\draw [->] (B3) to (B4);
	 	\draw [->] (B4) to (B5);
		\draw [->] (C1) to (C2);
		\draw [->] (C2) to node [above] {$2 \cdot$} (C3);
		\draw [->] (C3) to (C4);
		\draw [->] (C4) to (C5);
	 
		\draw [->] (A2) to node [right] {$\cong$} (B2);
		\draw [->] (A3) to (B3);
		\draw [->] (A4) to (B4);
		\draw [->] (B2) to node [right] {$\cong$}(C2);
		\draw [->] (B3) to node [right] {$\chi$} (C3);
		\draw [->] (B4) to node [right] {$\chi$ mod 2}(C4);
	\end{tikzpicture}
	\end{center}
The bottom half of the diagram is identified in \cite{MR3356279}. Again $\chi$ is the Euler characteristic. 
\end{proof}

\begin{example}  Let $\beta_d: \rB SO(d) \to \rB O(d)$. This case has been widely studied. It was shown in \cite{MR3212580} that the semistable oriented bordism group $\Omega_d^{T\beta} = \Omega_d^{T\rB SO(d)}$ is isomorphic to Reinhart's vector field cobordism group \cite{MR153021}. This, in turn, was previously shown to agree with the SKK groups of manifolds \cite{MR0362360}, which were shown to fit into a short exact sequence
	\begin{equation*}
		0 \to I_d \to SKK_d \to \Omega^{SO}_d \to 0
	\end{equation*}
where $\Omega^{SO}_d$ denotes the classical bordism group of oriented manifolds and $I_n$ is a cyclic summand - either zero, $\ZZ/2\ZZ$, or $\ZZ$ depending on $d$ mod 4. Moreover these extra summands were known to be detected by the Euler characteristic or Kervaire's semi-characteristic (defined below). 	
This was rediscovered in \cite{MR3356279}. We have:
	\begin{equation*}
		\Omega_d^{T\rB SO(d)} \cong \begin{cases}
		\Omega^{SO}_d  & d \equiv 3 \quad (\textrm{mod } 4) \\
		\ZZ \oplus \Omega^{SO}_d & d \equiv 2 \quad(\textrm{mod } 4) \\
			\ZZ/2\ZZ \oplus \Omega^{SO}_d & d \equiv 1 \quad(\textrm{mod } 4) \\
			\ZZ \oplus \Omega^{SO}_d & d \equiv 0 \quad(\textrm{mod } 4)
		\end{cases}
	\end{equation*}
where $\Omega^{SO}_d$ denotes the classical bordism group of oriented manifolds. If $\rho: \Omega_d^{T\rB SO(d)} \to \Omega^{SO}_d$ is the natural projection, then these isomorphisms are given, respectively, by $(\frac{1}{2} \chi, \rho)$ in dimensions $d = 4k+2$, $(k_\RR, \rho)$ in dimensions $d = 4k+1$, and by $(\frac{1}{2} (\chi + \sigma), \rho)$ in dimensions $d = 4k$, where $\chi$ denotes the Euler characteristic, $\sigma$ denotes the signature, and $k_\RR$ is Kervaire's semi-characteristic (defined when $d=4k +1$):
\begin{equation*}
	k_\RR(M) = \sum_{i = 0}^{\frac{d-1}{2}} \dim H^{2i}(M; \RR).
\end{equation*}
Pontryagin numbers and Stiefel-Whitney numbers can be used to give characters on the classical bordism group, and are well-known to completely determine $\Omega^{SO}_d$ \cite{MR120654}. It follows then from these calculations and Theorem~\ref{thm:invertarechars} that invertible oriented topological field theories determine precisely the following invariants:
\begin{itemize}
	\item Pontryagin and Stiefel-Whitney numbers;
	\item The Euler characteristic;
	\item Kervaire's semi-characteristic. 
\end{itemize}
Of these only Karvaire's semi-characteristic requires that the target category actually be \emph{super} vector spaces. 
\end{example}

\begin{remark}\label{rmk:injection}
	Following Lemma~\ref{lem:TangetBordismIsstablebordismplusEuler}, when $d=2q$ is even we have an injection $(\chi, pr):\Omega_{2q}^{T\beta_{2q}} \to \ZZ \oplus \Omega_{2q}^{\tau \beta_{\infty}}$. If $B_{2q}$ is orientable (i.e. it factors as $B_{2q} \to \rB SO(2q) \to \rB O(2q)$), then the previous example allows us to be more precise. Specifically it follows that  $(\frac{1}{2} \chi, pr): \Omega_{2q}^{T\beta_{2q}} \cong \ZZ \oplus \Omega_{2q}^{\tau \beta_{\infty}}$ if $q$ is odd, and $(\frac{1}{2} (\chi + \sigma), pr): \Omega_{2q}^{T\beta_{2q}} \cong \ZZ \oplus \Omega_{2q}^{\tau \beta_{\infty}}$ if $q$ is even, where $\chi$ is the Euler characteristic and $\sigma$ is the signature. 
\end{remark}

\begin{example}\label{ex:Ycohominvert}
	Fix a space $Y$ and let $\beta_d: Y \times \rB SO(d) \to \rB O(d)$ be the composition of projection to the second factor followed by the canonical map. Then a $\beta_d$-structure on a $d$-manifold consists of an orientation together with a map to $Y$. There is a map of spectra
	\begin{equation*}
		MT\beta_d \simeq MTSO(d) \wedge Y_+ \to \Sigma^{-d} H\ZZ \wedge Y_+
	\end{equation*} 
	induced from the bottom Postnikov truncation $MTSO(d) \to \Sigma^{-d} H\ZZ $. Thus there is map
	\begin{equation*}
		\Omega_d^{T\beta_d} \cong \pi_0(MTSO(d) \wedge Y_+) \to \pi_0(\Sigma^{-d} H\ZZ \wedge Y_+) \cong H_d(Y).
	\end{equation*}
	If $k$ is an algebraically closed field (not of characteristic 2), then we get an induced map
	\begin{equation*}
		H^d(Y; k^\times) \cong \Hom(H_d(Y), k^\times) \to \Hom(\Omega_d^{T\beta_d}, k^\times).
	\end{equation*}
	Thus, in light of Theorem~\ref{thm:invertarechars}, $k^\times$-valued cohomology classes on $Y$ give rise to invertible $\beta_d$-structured topological field theories. 
	
	Fix $\alpha \in H^d(Y; k^\times)$ and let $(M,f)$ be a closed $\beta_d$-structured $d$-manifold, i.e.  $M$ is an oriented manifold and $f: M \to Y$ is a map. Then the \emph{partition function} of the corresponding invertible field theory is given by
	\begin{equation*}
		\cZ_{Y, \alpha}(M,f) = \langle [M],  f^*\alpha\rangle
	\end{equation*}
	where $[M]$ is the fundamental class of $M$.
\end{example}

\subsection{Tangential $n$-types}\label{sec:compntypes}
In Section~\ref{sec:unstablebordism}, we showed that if an unstable tangential structure $\beta_k: B_k \to \rB O(k)$ is pulled back from a stable structure $\beta_{\infty}: B_{\infty} \to \rB O$, then the semistable bordism groups for $\beta_k$ can be computed from the stable bordism groups for $\beta_{\infty}$. In this section, we study an important sufficient condition for an unstable structure $\beta_k$ to arise from an $\beta_{\infty}$ in this way.

Indeed, most tangential structures $\beta_d: B_d \to \rB O(d)$ we will encounter in this paper will be $n$-truncated for some $n\geq 0$.  Recall that this means that the homotopy fibers of $\beta_d$ are $n$-types, equivalently for each basepoint the map $\pi_kB_d \to \pi_k\rB O(d)$ is injective for $k=n+1$ and an isomorphism for $k > n+1$ (see Example~\ref{exm:truncatedconnected}). 
\begin{definition}
\label{def:tangentialntype}
 A \emph{tangential $n$-type} is an $n$-truncated tangential structure $\beta_d: B_d \to \rB O(d)$.\end{definition}
For $k\leq d$, let  $\beta_k: B_k \to \rB O(k)$  denote the tangential structure obtained by pulling back $\beta_d$ along $\rB O(k) \to \rB O(d)$. 
By the universal property of pullbacks, a $B_k$-structure on a $k$-manifold $M$ ($k\leq d$) is the same as a $B_d$-structure on $T_M \oplus \varepsilon^{d-k}$, as defined in Section~\ref{sec:unstabletangstr}. If $\beta_d: B_d \to \rB O(d)$ is $n$-truncated, then so is the pulled back map $\beta_k: B_k \to \rB O(k)$.

If $\beta_d: B \to \rB O(d)$ is $n$-truncated for some $n \leq d-2$, it itself arises as a pullback of a $\beta_{d+1}: B \to \rB O(d+1)$ and hence ultimately of a stable tangential structure $\beta_{\infty}: B \to \rB O$. 
\begin{prop}\label{prop:pullbackstructures} 
	Assume that $\beta_d: B_d \to \rB O(d)$ is $n$-truncated for some $n \leq d-2$. Then, for all $ k > d$ (including $k=\infty$) there are $n$-truncated maps $\beta_k: B_k \to \rB O(k)$ and pullback squares:
\begin{center}
\begin{tikzpicture}
	\node (A) at (0,1.5) {$ B_{d} $};
	\node (B) at (2.5, 1.5) {$ B_{d+1} $};
	\node (BB) at (5, 1.5) {$\cdots$};
	\node (BBB) at (7.5, 1.5){$B_{\infty}$};
	\node (C) at (0, 0) {$ \rB O(d) $};
	\node (D) at (2.5, 0) {$ \rB O(d+1) $};
	\node (DD) at (5,0) {$\cdots$};
	\node (DDD) at (7.5,0) {$\rB O$};
	\draw [->] (A) to (B); 
	\draw[->] (B) to (BB);
	\draw[->] (BB) to (BBB);
	\draw[->] (D) to (DD);
	\draw[->] (DD) to (DDD);
	\draw[->] (BBB) to (DDD);
	\draw [->] (A) to (C); 
	\draw [->] (B) to (D); 
	\draw [->] (C) to (D); 
	\node at (0.5, 1) {$\lrcorner$};
	\node at (3, 1) {$\lrcorner$};
	\node at (5.5, 1) {$\lrcorner$};	
\end{tikzpicture}
\end{center}
\end{prop}
\begin{proof}
 Factor the composite of $\beta_d$ with the inclusion $\rB O(d) \to \rB O(d+1)$ into an $n$-connected map followed by an $n$-truncated map:
\begin{equation*}
	B_d \to[n\textrm{-connected}] B_{d+1} \to[n\textrm{-truncated}] \rB O(d+1). 
\end{equation*}
Since the map $\rB O(d) \to \rB O(d+1)$ is $(d-1)$-connected and hence also $(n+1)$-connected (as $n\leq d-2$), it follows that the maps in the following square have the indicated truncatedness/connectivity:
\begin{center}
\begin{tikzpicture}
	\node (A) at  (0,1.5) {$ B_d $};
	\node (B) at (3.5, 1.5) {$ B_{d+1} $};
	\node (C) at (0, 0) {$ \rB O(d) $};
	\node (D) at (3.5, 0) {$ \rB O(d+1) $};
	\draw [->] (A) to node[above,scale=0.75]{$n$-conn.} (B); 
	\draw [->] (A) to node[left, scale=0.75] {$n$-trunc.}  (C); 
	\draw [->] (B) to node[right, scale=0.75] {$n$-trunc.} (D); 
	\draw [->] (C) to node[below, scale=0.75] {$(n+1)$-conn.}(D); 
\end{tikzpicture}
\end{center} 
Such squares are necessarily pullback squares. (In fact, it would have been sufficient to assume the right vertical map to only be $(n+1)$-truncated.)
Iterating, we obtain a sequence of $n$-truncated $\beta_k: B_k \to \rB O(k)$ for all $k$, including $\beta_\infty: B_\infty \to BO$ (the $k= \infty$ case) such that $(B_k, \beta_k)$ is obtained as a pull-back of the subsequent tangential structure. 
\end{proof}

\section{Finite path integral theories and their topological sensitivity} \label{sec:DWtheories}

\subsection{Finite path integral super topological field theories} \label{sec:DWsTFTs}
As in Section~\ref{sec:bordandinvertTFT}, we fix a map $\beta:B \to \rB O(d)$ thought of as an ambient tangential structure.
Roughly speaking, to build finite path integral theories we start with a map $\xi: X \to B$ and a topological field theory $\cZ: \Bord_d^{(X,\beta \circ \xi)} \to \sVec_k$ and `integrate out the fibers of $\xi$' to obtain a field theory $\FP_{X, \xi, \cZ}: \Bord_d^{(B,\beta)} \to \sVec_k$. 

Explicitly, to a closed $(d-1)$-dimensional $(B, \beta)$-manifold $M$, we associate the space $\Map_B(M,X)$ of maps $f:M \to X$ over $B$. We may think of this space as the moduli space of $(X, \beta \circ \xi)$-structures on $M$ which refine the given $(B, \beta)$-structure. Similarly, to every $(B, \beta)$-bordism $W$ between $(B, \beta)$-manifolds $M_0$ and $M_1$ we associate the span of spaces \begin{equation}\label{eq:modulispan}
\Map_B(M_1, X) \ot[t] \Map_B(W,X) \to[s] \Map_B(M_0, X)\end{equation} which records the moduli space of $X$-structures on $W$ together with the restrictions of such structures to the incoming and outgoing boundary $M_0$ and $M_1$. 

Let $\mathrm{Span}(\cS)$ denote the $1$-category whose objects are spaces and whose morphisms are isomorphism classes of spans of spaces equipped with the symmetric monoidal structure given by the cartesian product of spaces (see Definition~\ref{def:decoratedspancat}). Since the moduli space $\Map_B(W' \circ_{M_1} W)$ of $X$-structures on a glued bordism $W' \circ_{M} W $ is homotopy equivalent to the (homotopy) pullback of moduli spaces $\Map_B(W', X) \times_{\Map_B(M, X)} \Map_B(W, X)$, and since the moduli space of $X$-structures on a disjoint union of $(B, \beta)$-manifolds is a product of moduli spaces of the components, we immediately obtain the following lemma: 

\begin{lemma}\label{lem:decoratedspanlift} For every map $\xi: X \to B$, the assignments of the spaces $\Map_B(M, X)$ to closed $(d-1)$-dimensional $(B, \beta)$-manifolds $M$ and spans~\eqref{eq:modulispan} to $d$-dimensional $(B, \beta)$-bordisms $W$ assemble into a symmetric monoidal functor
$\cF_{X, \xi}: \Bord_d^{B, \beta} \to \mathrm{Span}(\cS).
$ \qed
\end{lemma}

For any $(B, \beta)$-manifold $M$, any point $f\in \Map_B(M, X)$ induces a $(X, \beta \circ \xi)$-structure on $M$. Hence, given a topological field theory $\cZ: \Bord_d^{(X, \beta \circ \xi)} \to \sVec_k$ and a point $f\in \Map_B(M, X)$ we obtain a vector space $\cZ(M, f)$. Likewise, a path in $\Map_B(M, X)$ induces an isomorphism of $(X, \beta \circ \xi)$-structures on $M$ and hence the structure of an (invertible) $(X, \beta\circ \xi)$-bordism on $M \times [0,1]$ between its endpoints $(M, f_0)$ and $(M, f_1)$. Applying the TFT $\cZ$ induces an isomorphism of vector spaces $\cZ(M, f_0)  \to  \cZ(M, f_1)$ which only depends on the homotopy class of the path. In this way, the topological field theory $\cZ$ induces an $\sVec_k$-valued local system $\cL_{M}$ on the moduli space $\Map_B(M, X)$. 
Similarly, for a $(B, \beta)$-bordism $W$ and a point $f \in \Map_B(W, X)$ with $s(f) = f_0\in \Map_B(M_0, X)$ and $t(f) = f_1 \in \Map_B(M_1, X)$, the topological field theory induces a morphism in $\sVec_k$:
\[
\alpha_W(f):= \cZ(W, f): \cL_{M_0}(f_0) = \cZ(M_0, f_0) \to \cZ(M_1, f_1) = \cL_{M_1}(f_1).
\]
Functoriality of the TFT $\cZ$ shows that this defines a map of local systems $\alpha_W: s^*\cL_{M_0} \to t^* \cL_{M_1}$. 
 
As in Definition~\ref{def:decoratedspancat}, let $\mathrm{Span}(\cS, \sVec_k)$ denote the symmetric monoidal $1$-category of spans of spaces equipped with $\sVec_k$-valued local systems. 
The above assignments of local systems and maps of local systems are compatible with gluing of bordisms and disjoints unions and hence may be summarized in the following lemma. 
\begin{lemma} Let $\xi: X \to B$ be a map of spaces and let $\cZ: \Bord_d^{(X, \beta \circ \xi)} \to \sVec_k$ be a symmetric monoidal functor. Then, the above choices of local systems and maps of local systems lift the functor $\cF_{X, \xi}: \Bord_d^{(B, \beta)} \to \mathrm{Span}(\cS)$ to a symmetric monoidal functor $\cF_{X, \xi, \cZ}: \Bord_d^{(B, \beta)} \to \mathrm{Span}(\cS, \sVec_k)$. \qed
\end{lemma}
 
 Our finite path integral theory is built by composing this functor $\cF_{X, \xi, \cZ}: \Bord_d^{(B, \beta)} \to \mathrm{Span}(\cS, \sVec_k)$ with the linearization functor $\mathrm{Span}(\cS^\pi, \sVec_k) \to \sVec_k$ from 
 Corollary~\ref{cor:spanfunctor}. 
 To do so, we need to ensure that the moduli spaces $\Map_B(M, X)$ appearing in the definition of $\cF_{X, \xi}$ are $\pi$-finite. 
 
 The following is an easy exercise in obstruction theory. 
 \begin{lemma} Let $\xi: X \to B$ be a map with $\pi$-finite fibers. Then, for any map $Y \to B$ where $Y$ has the homotopy type of a finite CW complex (such as $Y=M$ for a compact $(B, \beta)$-structured manifold $M$), the space $\Map_B(Y, X)$ is $\pi$-finite. 
\end{lemma}		

\begin{proof}
	Let $F$ be the homotopy fiber of $\xi: X \to B$. By obstruction theory, $\pi_0\Map_B(Y, X)$ is non-canonically isomorphic to a subset of the product $\prod_k H^k(Y; \pi_k F)$. Since $Y$ is a finite CW complex and each $\pi_kF$ is assumed to be a finite group, each term of this product is finite. Moreover as $\pi_kF$ vanishes for sufficiently large $k$, there are only finitely many non-trivial terms. Hence $\pi_0\Map_B(Y, X)$ is isomorphic to a subset of a finite set, and is hence finite. 
	
	Likewise, fixing a given map $f: Y \to X$ over $B$, obstruction theory implies that $\pi_m(\Map_B(Y, X), f)$ is non-canonically isomorphic to a subset of $\prod_k H^k(\Sigma^mY; \pi_k F)$. Hence $\pi_m(\Map_B(Y, X), f)$ is also finite following similar reasoning as above. Finally, since $\pi_kF$ vanishes for sufficiently large $k$ and $H^k(\Sigma^mY; \pi_k F)$ vanishes for $k<m$, it follows that $\pi_m(\Map_B(Y, X), f)$ vanishes once $m$ is sufficiently large. 
\end{proof}

\begin{definition}\label{def:GenDWTheory}
Fix a map $\beta: B \to \rB O(d)$ thought of as an ambient tangential structure.
Let $\xi: X \to B$ be a map with $\pi$-finite fibers and let $\cZ: \Bord_d^{(X, \beta \circ \xi)}\to \sVec_k$ be a symmetric monoidal functor. The associated  \emph{finite path integral topological field theory} $\FP_{X,\xi, \cZ}: \Bord_d^{(B, \beta)} \to \sVec_k$ is the composite $$\Bord_d^{(B, \xi)} \to[\cF_{X, \xi, \cZ}] \mathrm{Span}(\cS^\pi, \sVec_k) \to[\Phi] \sVec_k,$$
where $\cF_{X, \xi, \cZ}$ is the functor from Lemma~\ref{lem:decoratedspanlift} and $\Phi$ is the functor from Corollary~\ref{cor:spanfunctor}.
If $\cZ$ is an invertible field theory which by Theorem~\ref{thm:invertarechars} is uniquely characterized by a character $\omega:\Omega_d^{T(\beta \circ \xi)} \to k^{\times}$, we denote the associated finite path integral theory by $\FP_{X, \xi, \omega}$. 	
\end{definition}

\begin{example}\label{rmk:valueofDWonclosed}
	If $W$ is a closed $d$-dimensional $(B, \beta)$-manifold, the value of $\FP_{X, \xi, \cZ}$ on $W$ may be unpacked to the following formula:
		\begin{equation*}
		\FP_{X,\xi, \cZ}(W) = \sum_{[f] \in \pi_0\Map_B(W, X)} \#(\Map_B(W, X), f) ~\cZ([W,f]),
	\end{equation*}
	where $\cZ([W, f])$ denotes the value of $\cZ$ on $W$ with its $(X, \beta \circ \xi)$-structure induced by its $(B, \beta)$-structure and $f\in \Map_B(W, X)$.
	
	In particular, if $\cZ$ is invertible and hence determined by a character $\omega: \Omega_d^{T(\beta \circ \xi)} \to k^{\times}$, this becomes 
	\begin{equation}\label{eq:DWclosedmanifold}
	\FP_{X,\xi, \omega}(W) = \sum_{[f] \in \pi_0\Map_B(W, X)} \#(\Map_B(W, X), f) ~\omega([W,f]),
	\end{equation}
	where $[W, f] \in \Omega^{T(\beta \circ \xi)}_d$ is the stable tangential $(B, \beta \circ \xi)$-structure on $W$ induced by $f$ and the (unstable) $(B, \beta)$-structure on $W$. 
\end{example}

\begin{example} \label{ex:DWQuinnasGenDW} 
	 Let $Y$ be a $\pi$-finite space with a $d$-dimensional $k^\times$-valued cohomology class $\alpha \in \mathrm{H}^d(Y, k^\times)$. As in Example~\ref{ex:Ycohominvert} we obtain an invertible $X$-structured topological field theory where $X = Y \times \rB SO(d)$. An $X$-structured $d$-dimensional manifold consists of an oriented manifold $M$ together with a map $f: M \to Y$.  The partition function of this invertible topological field theory (i.e. the character $\omega: \Omega_d^{Y \times \rB SO} \to k^\times$) is given by pairing with the fundamental class $[M]$ of $M$:
$$M \mapsto \langle [M],  f^*\alpha\rangle$$ .
	 
Set $B = \rB SO(d)$ and let $\xi: X \to B$ be projection. Then Definition~\ref{def:GenDWTheory} produces an oriented topological field theory $\FP_{X, \xi, \omega}: \Bord_d^{\rB SO(d)} \to \sVec$ out of the above invertible topological field theory. This field theory factors through $\Vec$ and recovers Quinn's twisted finite homotopy TQFT~\cite{MR1338394} for $(Y, \alpha)$ as a special case of Definition~\ref{def:GenDWTheory}.

In particular, if $Y= \rB G$ is the classifying space of a finite group $G$, this recovers the traditional twisted Dijkgraaf-Witten theory.
\end{example}

\subsection{The finite path integral invariant via Pontryagin pairings}

Comparing Formula \eqref{eq:DWclosedmanifold} with Definition~\ref{def:pontryaginpairing}, we see that $\FP_{X, \xi, \omega}(W)$ should be understood as the Pontryagin pairing (as in Definition~\ref{def:pontryaginpairing}) of the $\infty$-category $\cS_{/ \cB}$ of spaces over $B$ (whose mapping spaces are precisely $\Map_B(Y, X)$) and the functor $\Omega_d: \cS_{/B} \to \mathrm{Ab}$ which maps a $\xi: X \to B$ to the semistable bordism group $\Omega^{T(\beta \circ \xi)}_d$ (Definition~\ref{def:tangentialbordism}). However, the category $\cS_{/B}$ is not locally $\pi$-finite and its Pontryagin pairing is therefore not defined everywhere.

We therefore follow Remark~\ref{rem:pairingsubcats} and~\ref{rem:subcatPontryaginpairing} and consider the restricted Pontryagin pairing
\[
\langle -, - \rangle^{\cS^f_{/B}, \cS^\pi_{/B}}_{\cS_{/B} , \Omega_d}: \bigoplus_{[\xi] \in \pi_0 \cS^\pi_{/B}} k[\widehat{\Omega_d^{T(\beta \circ \xi)}}/{\pi_0 \Aut_{\cS_{/B}}(\xi)}] \otimes\bigoplus_{[\psi] \in \pi_0 \cS^f_{/B}} k[\Omega_d^{T(\beta \circ \psi)}/\pi_0 \Aut_{\cS_{/B}}(\psi)] \to k
\]
between the full subcategory $\cS^f_{/B}$ of $\cS_{/B}$ on those $\psi:Y \to B$ for which $Y$ is homotopy equivalent to a finite CW complex and the full subcategory $\cS^\pi_{/B}$ of those $\phi:X \to B$ with $\pi$-finite fibers. 

\begin{lemma}\label{lem:DWPont} For any closed $d$-dimensional $(B, \beta)$-manifold $W$, the following holds:
\[\FP_{X, \xi, \omega}(W) = \left\langle \vphantom{\frac{a}{b}}(\xi:X \to B, \omega), (\theta:W \to B, [W])\right\rangle_{\cS_{/B}, \Omega_d}^{\cS^f_{/B}, \cS^\pi_{/B}} 
\]
Here, $\theta:W \to B$ is the map which comes as part of the $(B, \beta)$-structure on $W$, and $[W] \in \Omega_d^{T(\beta \circ \theta)}$ is the stable tangential $(W, \beta \circ \theta)$-structured bordism class induced by $W$ with its unstable $(W, \beta \circ \theta)$-structure (where the lower triangle encodes the given $(B, \beta)$-structure on $W$):\[\begin{tikzcd}
	&& W \\
	&& B \\
	W && {\rB O(d)}
	\arrow["TW"', from=3-1, to=3-3]
	\arrow["\theta", from=3-1, to=2-3]
	\arrow["\beta", from=2-3, to=3-3]
	\arrow["\id", from=3-1, to=1-3]
	\arrow["\theta", from=1-3, to=2-3]
\end{tikzcd}\]
\end{lemma}
\begin{proof}This follows immediately from comparing the expression~\eqref{eq:DWclosedmanifold} with~\eqref{eq:pontryaginpairing} and Remarks~\ref{rem:subcatPontryaginpairing} and~\ref{rem:pairingsubcats}. 
\end{proof}

In particular, $\FP_{X, \xi, \omega}(W) \in k$ only depends on the equivalence class of the map $\theta:W\to B$ in $\cS_{/B}$ and the orbit of $[W] \in \Omega^{T(\beta \circ \theta)}_d$ under the $\Aut_{\cS_{/B}}(\theta)$ action. Likewise, it also only depends on $\xi$ up to equivalence in $\cS_{/B}$ and only on the orbit of $\omega$ under the $\Aut_{\cS_{/B}}(\xi)$ action on the group of characters $\widehat{\Omega^{T(\beta \circ \xi)}}$.

\subsection{Finite path integral theories of type $n$}

For a systematic study of the sensitivity of finite path integral theories, we introduce the following filtration on the collection of such theories.

\begin{definition}\label{def:typeDW} Fix an ambient tangential structure $B \to \rB O(d)$ and let $n \geq 0$ be an integer. We say that a finite path integral theory $\FP_{X, \xi, \omega}$ associated to a map $\xi:X \to B$ and a character $\omega: \Omega^{T(\beta \circ \xi)} \to k^\times$ is \emph{of type $n$}, if $\xi$ is $n$-truncated (see Example~\ref{exm:truncatedconnected}).

We will also say that two closed $B$-structured $d$-manifolds $M$ and $N$ are \emph{indistinguishable by type $n$ finite path integral theories} if for all $n$-truncated $\xi:X \to B$ and all characters $\omega: \Omega^{T(\beta \circ \xi)} \to k^\times$, \[\FP_{X, \xi, \omega}(M) = \FP_{X,\xi, \omega}(N) \in k
.\] \end{definition}
Note that if $\pi_0 B$ is finite, then any map $\xi: X \to B$ with $\pi$-finite fibers is $n$-truncated for some $n \geq 0$, and hence every finite path integral theory is of type $n$ for some $n \geq 0$.

Let $\tau_{\leq n} \cS_{/B}$ denote the full subcategory of $\cS_{/B}$ on those $\xi: X\to B$ which are $n$-truncated.
   
Given a map  $X \to B$, the $n$-connected/$n$-truncated factorization system on the $\infty$-category of spaces $\cS$ induces a factorization
\begin{equation*}
	X \to \tau_{\leq n}^B X \to B
\end{equation*} 
 into an $n$-connected map $X\to \tau_{\leq n}^B X$ followed by an $n$-truncated map $\tau_{\leq n}^B X \to B$. In this way, given an $X \in \cS_{/B}$, we functorially obtain an $\tau_{\leq n}^B X \in \tau_{\leq n} \cS_{/B}$ which we will henceforth refer to  as the \emph{$B$-$n$-type of $X$}. 
 
\begin{example}
	Let $M$ be a smooth $d$-manifold with $T_M:M \to \rB O(d)$ classifying the tangent bundle of $M$. The map $T_M$ has a factorization (unique up to homotopy) $M \to \tau_{\leq n}^{\rB O(d)} M \to \rB O(d)$ into an $n$-connected map followed by an $n$-truncated map (see Example~\ref{exm:truncatedconnected}). $\tau_{\leq n}^{\rB O(d)} M$ is the \emph{tangential $n$-type of $M$} (c.f. the stable tangential and normal $n$-types of $M$ (Example~\ref{ex:normalntype}))
\end{example}

\begin{definition} Given an $X\in \cS_{/B}$, we will say that it has \emph{$\pi$-finite $B$-$n$-type} if its $B$-$n$-type $\tau_{\leq n}^B X \to B$ has $\pi$-finite fibers. Equivalently, $X\in \cS_{/B}$ has $\pi$-finite $B$-$n$-type if the first $n$ homotopy groups $\pi_k, 0 \leq k \leq n$ of all fibers of $X\to B$ are finite. We let $\tau_{\leq n} \cS_{/B}^\pi$ denote the full subcategory of $\cS_{/B}$ on spaces with $\pi$-finite $B$-$n$-type.
\end{definition}

To apply the non-degeneracy results of Section~\ref{sec:pontryaginpairing}, we need to introduce a mild finiteness condition on the space $B$ describing the ambient tangential structure. This finiteness condition is fulfilled by the classifying spaces $\rB O(d)$ and their connective covers, and includes all examples of interest that we are aware of. This condition is discussed more thoroughly in Appendix~\ref{sec:nfindom}. 

\begin{definition}\label{def:nfindominatedmain}
	A space 
	$X$ is \emph{$n$-finitely dominated} if there exists an $n$-dimensional finite CW complex $K$ and an
	 $(n-1)$-connected map $K \to X$. 
\end{definition}

\begin{assumption}\label{ass:nfindom}
	For most of the following, we will assume that the space $B$ is $n$-finitely dominated. This includes the spaces $\rB O$, $\rB O(d)$, and any $k$-connective covers of these spaces such as $\rB SO(d)$, $\rB Spin(d)$, etc. (Example~\ref{ex:BOfindom}).
\end{assumption}

Under Assumption~\ref{ass:nfindom}, the category  $\tau_{\leq n} \cS_{/B}^{\pi}$ is locally $\pi$-finite (Corollary~\ref{cor:pifinislocpifin}). Thus $\tau_{\leq n} \cS_{/B}^{\pi}$ has an associated Pontryagin pairing. The finite path integral invariant can be expressed in terms of this pairing:

\begin{lemma}\label{lem:DWispontryagin} Let $\beta:B\to \rB O(d)$ be an ambient tangential structure satisfying Assumption~\ref{ass:nfindom}, let $\xi: X \to B$ be an $n$-truncated map with $\pi$-finite fibers, and let $\omega: \Omega_d^{T(\beta \circ \xi)}\to k^{\times}$ be a group homomorphism. 

Then, for any closed $d$-dimensional $(B, \beta)$-manifold $W$ with $\pi$-finite $B$-$n$-type, the finite path integral invariant is given by
\[\FP_{X, \xi, \omega}(W) = \left\langle \vphantom{\frac{a}{b}} (\xi, \omega), (\tau_{\leq n}^B W , [W]) \right\rangle_{\tau_{\leq n} \cS^{\pi}_{/B}, \Omega_d}
\]
where $\langle -, -\rangle_{\tau_{\leq n} \cS_{/B}^{\pi}, \Omega_d}$ denotes the Pontryagin pairing associated to the locally $\pi$-finite $\infty$-category $\tau_{\leq n} \cS_{/B}^{\pi}$ and the functor $\Omega_d$ (see Definition~\ref{def:pontryaginpairing}).
Here, $$[W] \in \Omega^{T(\beta\circ \tau_{\leq n}^B W)}_d/ \pi_0\Aut_{\cS_{/B}}(\tau_{\leq n}^BW \to B)$$ is the bordism class of $W$ with its canonical $\tau^B_{\leq n}W$-structure. 
\end{lemma}
\begin{proof} The lemma immediately follows from~\eqref{eq:DWclosedmanifold} and the fact that the truncation functor $\cS_{/B} \to \tau_{\leq n} \cS_{/B}$ is left adjoint to the inclusion $\tau_{\leq n} \cS_{/B} \to \cS_{/B}$: for any $Y \to B$ and $n$-truncated $X\to B$ the canonical map $\Map_B(\tau_{\leq n}^B Y, X) \to \Map_B(Y, X)$ is an equivalence. 
 \end{proof}

 \subsection{The main theorem}\label{sec:mainthm}

The $k$-connected/$k$-truncated factorization system $(\cL^{(k)}, \cR^{(k)})$ on spaces induces one on $\cS_{/B}$, the category of spaces over $B$. Explicitly, a map of spaces $f:X\to Y$ over $B$ is $k$-connected or $k$-truncated if the underlying map of spaces $f:X \to Y$ is $k$-connected or $k$-truncated, respectively. This connected/truncated factorization system restricts to a factorization system on the subcategories $\tau_{\leq n}\cS_{/B}$ of $B$-$n$-types and $\tau_{\leq n} \cS_{/B}^\pi$ of $\pi$-finite $B$-$n$-types.

These define a nested sequence of factorization systems as in Definition~\ref{def:nestedfactsystem}, i.e. $\cR^{(k-1)} \subseteq \cR^{(k)}$ or equivalently $\cL^{(k+1)} \subseteq \cL^{(k)}$. In particular, we have subcategories 
$\cT^{(\ell)} \subseteq \tau_{\leq n}\cS^{\pi}_{/B}$ for $-2 \leq \ell \leq n$:
	\begin{equation*}
		\cT^{(\ell)} = \begin{cases}
			\cR^{(-1)} & \textrm{if } \ell=-2 \\
			\cR^{(\ell+1)} \cap \cL^{(\ell)} &\textrm{if } -1 \leq \ell \leq n-1  \\
			\cL^{(n)} & \textrm{if } \ell=n
		\end{cases}
	\end{equation*}
The morphisms of $\cT^{\ell}$ are those maps $f:X \to Y$ over $B$ for which $\pi_{\ell}(f)$ is surjective, $\pi_{\ell + 1}(f)$ is injective, and $\pi_i(f)$ is an isomorphism for all $i \neq \ell, \ell+1$. 

\begin{remark}
	In Definition~\ref{def:nestedfactsystem} the indexing of the nested system was for $0\leq \ell \leq n$, but here it is more convenient to index on $-2 \leq \ell \leq n$. This makes no substantiative change and we hope it will not lead to confusion.  
\end{remark}

For $\pi$-finite $B$-$n$-types, endomorphisms in $\cT^{(\ell)}$ are automatically weak equivalences since endomorphism of finite groups which are either injections or surjections are automatically isomorphisms.
Combined with Theorem~\ref{thm:linearpairing} and Corollary~\ref{cor:pontryaginnondeg}, we obtain the following theorem.

\begin{theorem}\label{thm:MainThm}
Let $\beta:B  \to \rB O(d)$ be an ambient tangential structure and let $n \geq 0$. Let $M$ and $N$ be $d$-dimensional $(B, \beta)$-manifolds with $\pi$-finite $B$-$n$-types. 

Then $M$ and $N$ are indistinguishable by type-$n$ finite path integral theories if and only if \begin{enumerate}
	\item they have equivalent $B$-$n$-types $\tau_{\leq n}^BM \simeq (Y, \psi) \simeq \tau_{\leq n}^B N$;
	\item  the bordism classes $[M]$ and $[N]$ in $\Omega_d^{T(\beta \circ \psi)}$ lie in the same orbit under the action of $\pi_0 \Aut_{\cS_{/B}}(\psi)$.
	\end{enumerate}
\end{theorem}
\begin{proof} With Assumption~\ref{ass:nfindom}, we may argue as follows. 
	By Lemma~\ref{lem:DWispontryagin}, the value of $\FP_{X, \xi, \omega}(M)$ for any $n$-truncated $X\to B$ and $\omega$ is given by the Pontryagin pairing, and hence only depends on the vector $M$ induces in $\bigoplus_{\psi \in \pi_0 \tau_{\leq n} \cS_{/B}^{\pi}} k[\Omega_d^{T(\beta \circ \psi)}/\pi_0 \Aut_{\cS_{/B}}(\psi)]$, and hence only on data (1) and (2). 

Conversely, since the truncated/connected nested factorization system on $\tau_{\leq n} \cS_{/B}^{\pi}$ has the property that endomorphisms in $\cT^{(\ell)}$ are isomorphisms, it follows from Corollary~\ref{cor:pontryaginnondeg} that the Pontryagin pairing is non-degenerate and hence that if $M$ and $N$ are indistinguishable by finite path integral theories, then they induce the same data (1) and (2).  

However Theorem~\ref{thm:MainThm} holds without Assumption~\ref{ass:nfindom}. In that case in place of $\tau_{\leq n} \cS_{/B}^{\pi}$ we use the category $\tau_{\leq n} \cS_{/B}^{\pi~\textrm{f.d.}}$ described in Appendix~\ref{sec:nfindom}. Corollary~\ref{cor:pifinislocpifin} states that it is locally $\pi$-finite, and thus, arguing as before, it Pontryagin pairing is non-degenerate. The tangential $n$-types of compact manifolds will be $n$-finitely dominated and hence live in $\tau_{\leq n} \cS_{/B}^{\pi~\textrm{f.d.}}$. In this case it is in fact sufficient to use type-$n$ finite path integral theories based on $n$-finitely dominated $(Y, \psi)$. 
\end{proof}

Lemma~\ref{lem:DWispontryagin} shows that the partition function of finite path integral theories can be computed using the Pontryagin pairing for $B$-manifolds which have $\pi$-finite $B$-$n$-type. If every $\pi$-finite $B$-$n$-type $\xi$ and element of $\Omega_d^T\xi$ were represented by such a manifold, then the non-degeneracy of the Pontryagin pairing on $\tau_{\leq n} \cS_{/B}^{\pi}$ would immediately lead to a converse to Theorem~\ref{thm:MainThm} - that they are determined by their partition functions. However it is a subtle and interesting surgery question to determine which elements of $\Omega_d^T\xi$ can be so represented. The following result of Kreck will allow us to obtain a partial converse to Theorem~\ref{thm:MainThm}.

\begin{proposition}[{\cite[Prop.~4]{MR1709301}}]\label{pro:krecksurgery}
	Suppose that $n < \lfloor \frac{d}{2} \rfloor$. Let $\xi:X \to B$ be a $\pi$-finite $n$-$B$-type (or more generally $X$ need only be $n$-finitely dominated). Then every class $[M] \in \Omega^{TX_d}$ may be represented by a $(X, \xi)$-manifold $N$ such that the map $N \to X$ identifies $\tau_{\leq n}^BN \simeq X$. \qed 
\end{proposition}

\begin{theorem} \label{thm:partitionfunctiondetected}
	Suppose that $n< \lfloor \frac{d}{2} \rfloor$ and that Assumption~\ref{ass:nfindom} holds. If $((X_1, \xi_1),\omega_1)$ and $((X_2,\xi_2), \omega_2)$ give rise to type-$n$ finite path integral theories whose partition functions are identical on closed $d$-manifolds with $\pi$-finite tangential $n$-type, then $X_1 \simeq X_2 = X$ (over $B$), and $\omega_1$ and $\omega_2$ have the same restriction to the fixed points of $\Omega^{T\xi}_d$ under the action of $\pi_0 \Aut(\xi)$. The resulting finite path integral theories are consequently isomorphic. 
\end{theorem}

\begin{proof}
	Proposition~\ref{pro:krecksurgery} implies that every element in $\bigoplus_{\psi \in \pi_0 \tau_{\leq n} \cS_{/B}^{\pi}} k[\Omega_d^{T(\beta \circ \psi)}/\pi_0 \Aut_{\cS_{/B}}(\psi)]$ is represented by a (linear combination) of $B$-manifolds with $\pi$-finite $n$-$B$-type, and Lemma \ref{lem:DWispontryagin}  shows that the partition function of finite path integral theories on these manifolds can be computed using the Pontryagin pairing on $\tau_{\leq n} \cS_{/B}^{\pi}$. The non-degeneracy of this pairing implies that $((X_1, \xi_1),\omega_1)$ and $((X_2,\xi_2), \omega_2)$ represent the same element in $\bigoplus_{[\xi] \in \pi_0 \cC} k[\widehat{\Omega^T\xi_d}/{\pi_0 \Aut_{\cC}(\xi)}]$, which immediately implies the claimed result. 
\end{proof}

\subsection{Dimensional reduction of finite path integral theories}\label{sec:dimred}

The \emph{dimensional reduction} of a field theory $\cZ: \Bord_d \to \cC$ along a closed $(d-k)$-manifold $M$ is defined to be the field theory $\cZ(M \times -): \Bord_{k}\to \cC$. 

Dimensional reduction of tangentially structured field theories is considered in~\cite{MR3811769}.
Given tangential structures  $B_{d-k} \to \rB O(d-k)$, $B_k \to \rB O(k)$  and $B \to \rB O(d)$ fitting into a commutative diagram 
\begin{equation}\label{eq:structuresquare}
\begin{tikzpicture}[baseline=0.75cm]
	\node (A) at (0,1.5) {$ B_k \times B_{d-k} $};
	\node (B) at (3.5, 1.5) {$ B $};
	\node (C) at (0, 0) {$ \rB O(k) \times \rB O(d-k) $};
	\node (D) at (3.5, 0) {$ \rB O(d) $};
	\draw [->] (A) to (B); 
	\draw [->] (A) to (C); 
	\draw [->] (B) to (D); 
	\draw [->] (C) to (D); 	
\end{tikzpicture}
\end{equation}
the product of a $B_{d-k}$-structured $(d-k)$-manifold $M$ and a $B_{k}$-structured $k$-manifold $N$ induces a canonical $B$ structure on $M\times N$. 

Therefore, any closed $B_{d-k}$-structured $(d-k)$-manifold $M$ induces a symmetric monoidal functor $\red_M = (M \times -) : \Bord_{k}^{B_{k}} \to \Bord_d^B $, and any $B$-structured $d$-dimensional field theory may be dimensionally reduced along $M$ to a $B_{k}$-structured $k$-dimensional field theory. 

In fact, for any (unstructured) $(d-k)$-manifold $M$ and any structure $B \to \rB O(d)$, there is a universal $k$-dimensional tangential structure $B_M \to \rB O(k)$ which is essentially defined so that the data of a $B_M$-structure on a $k$-manifold $N$ is precisely the data of a $B$-structure on $M\times N$ and which hence gives rise to a \emph{total dimensional reduction} functor $\red_M^{\textrm{tot}} = (M\times -): \Bord_k^{B_M} \to \Bord_d^B$ (see~\cite[Sec 9.2]{MR3811769}). 

We will later need a version of this construction relative to given tangential structures in a commutative square~\eqref{eq:structuresquare}: Suppose $X \to B$ is a space over $B$. Then, given any $B_{d-k}$-structured closed $(d-k)$-manifold $M = (M, \theta)$, we define a map $X_{M} \to B_k$ as the following pullback:
\begin{equation}\label{eq:totalreduction}
	\begin{tikzpicture}[baseline=0.75cm]
		\node (A) at (0,1.5) {$ X_M $};
		\node (B) at (7, 1.5) {$ \Map(M,X)$};
		\node (C) at (0, 0) {$ B_k \times \{\theta\} $};
		\node (C2) at (3.5, 0) {$B_k \times \Map(M, B_{d-k})$} ;
		\node (D) at (7, 0) {$ \Map(M,B) $};
		\node at (0.5, 1) {$\lrcorner$};
		\draw [->] (A) to (B); 
		\draw [->] (A) to (C); 
		\draw [->] (B) to (D);
		 \draw [->] (C) to (C2); 
		\draw [->] (C2) to (D); 	
	\end{tikzpicture}
\end{equation}
As in~\cite[Sec 9.2]{MR3811769}, the structure $X_M \to B_k$ is universal in the sense that for a $B_k$-structured $k$-manifold $N$, a lift of the canonical $B$-structure on $M\times N$ along $X\to B$ to an $X$ structure precisely amounts to a lift of the $B_k$-structure on $N$ along $X_M \to B_k$ to a $X_M$ structure. In particular, $M$ induces a \emph{total dimensional reduction} symmetric monoidal functor $\red_M^\textrm{tot} = (M \times -) : \Bord_{k}^{X_M} \to \Bord_d^X$. 

\begin{remark}
	The dimensional reduction and total dimensional reductions of invertible TFTs are again invertible. 
\end{remark}

\subsubsection{Dimensional reduction of finite path integral theories}

\begin{theorem}\label{thm:DWdimreductisDW} 
	Consider a $B$-structured $d$-dimensional finite path integral theory $\FP_{X, \cW}: \Bord_d^B \to \sVect$ constructed as in Def.~\ref{def:GenDWTheory} from a map $X\to B$ with $\pi$-finite fibers and a field theory $\cW: \Bord_d^X \to \sVect$.
Given tangential structures as in the commutative square~\eqref{eq:structuresquare} and a closed $B_{d-k}$-structured $(d-k)$-manifold $M$, the dimensionally reduced theory $\FP_{(X, \cW)} \circ \red_M: \Bord_k^{B_k} \to \sVect$ is equivalent to the finite path integral theory $\FP_{(X_M, \cW \circ \red_M^\textrm{tot})}$ built from the map $X_M \to B_k$ defined in~\eqref{eq:totalreduction} and the total dimensional reduction of $\cW$ to a field theory $ \cW \circ \red_M^\textrm{tot} = \cW(M \times -): \Bord_k^{X_M} \to \sVect$. 
\end{theorem}

\begin{proof}
First we establish that $X_M \to B_k$ has $\pi$-finite homotopy fibers, so that the dimensional reduction $\FP_{(X_M, \cW \circ \red_M^\textrm{tot})}$ is meaningful. 
Since~\eqref{eq:totalreduction} is a pullback square, the fiber of $X_M\to B_k$ at a point $b_k \in B_k$ agrees with the fiber of $\Map(M, X) \to \Map (M,B)$ at the map $M \to B_{d-k} \to[B_{d-k} \times b_k] B_{d-k} \times B_k \to B$. Since the fibers of $X\to M$ are $\pi$-finite and $M$ is a compact manifold, this fiber is also $\pi$-finite. 

Furthermore, by construction for each $B_k$-manifold $N$ we have a natural equivalence of spaces:
\begin{equation*}
	\Map_{B_k}(N, X_M) \simeq \Map_B(M \times N, X),
\end{equation*}
and a commuting diagram:
\begin{center}
\begin{tikzpicture}
	\node (A) at (0,1.5) {$ \Map_B(M \times N, X) $};
	\node (B) at (4.5, 0.75) {$ \sVect $};
	\node (C) at (0, 0) {$ \Map_{B_k}(N, X_M) $};
	\draw [->] (A.east) to  node [above right] {$\cW$} (B); 
	\draw [<-] (A) to node [right] {$\simeq$} node [left] {$\red_M^\textrm{tot}$} (C); 
	\draw [->] (C.east) to  node [below right] {$\cW \circ \red_M^\textrm{tot}$}(B); 
\end{tikzpicture}.
\end{center}
Thus we have a natural equivalence of (super) vector spaces:
\begin{align*}
	\FP_{(X, \cW)} \circ \red_M (N) &= \bigoplus_{[\psi] \in \pi_0\Map_B(M \times N, X)} (\cW(M \times N, \psi))_{\pi_1(\Map_B(M \times N, X), \psi)} \\
	&\cong \bigoplus_{[\psi] \in \pi_0\Map_{B_k}(N, X_M)} (\cW \circ \red_M^\textrm{tot}(N))_{\pi_1(\Map_{B_k}(N, X_M), \psi)} \\
	& = \FP_{(X_M, \cW \circ \red_M^\textrm{tot})}(N).
\end{align*}
Similarly, for a $k$-dimensional $B_k$-bordism $W$ from $N_0$ to $N_1$, the value of $\FP_{(X, \cW)} \circ \red_M (W)$ is the linearization of the following decorated span
\begin{equation*}
	(\Map_B(M \times N_1, X), \cL_\cW) \ot[t] (\Map(M \times W, X), \alpha^\cW) \to[s] (\Map_B(M \times N_0, X), \cL_\cW).
\end{equation*}
The value of $\FP_{(X_M, \cW \circ \red_M^\textrm{tot})}(W)$ is the linearization of the following equivalent decorated span
\begin{equation*}
	(\Map_{B_k}(N_1, X_M), \cL_{\cW \circ \red_M^\textrm{tot}}) \ot[t] (\Map_{B_k}(W, X_M), \alpha^{\cW \circ \red_M^\textrm{tot}}) \to[s] (\Map_{B_k}(N_0, X_M), \cL_{\cW \circ \red_M^\textrm{tot}})
\end{equation*}
and thus the two TFTs are isomorphic. 
\end{proof}

\begin{example}Consider a Dijkgraaf-Witten theory in the classical sense, where $B =\rB SO(d)$, $X = \rB SO(d) \times \rB G \to \rB SO(d)$ for $G$ a finite group, and $\cW: \Bord_d^{or, G} \to \sVec$ is a $G$-equivariant field theory, which  leads to an oriented field theory $\FP_{\cW, G}: \Bord_k^{or} \to \sVec$. Let $M$ be an oriented $(d-r)$-manifold. Then, Theorem~\ref{thm:DWdimreductisDW} asserts that the dimensionally reduced\footnote{In the notation of Theorem~\ref{thm:DWdimreductisDW}, we set $B_k = \rB SO(k)$ and $B_{d-k} = \rB SO(d-k)$ and consider the evident maps in the square~\eqref{eq:structuresquare}.} theory $\FP_{\cW,G}(M \times -): \Bord_k^{or} \to \sVec$ is the $k$-dimensional oriented Dijkgraaf-Witten theory built from the finite $1$-groupoid \[\Map(M, \rB G) \cong \bigsqcup_{\substack{\text{isomorphism classes }[P]\\\text{of }G\text{-bundles over }M}} \rB \Aut(P)\]
and $\cW(M \times -)$ considered as a $\Map(M, \rB G)$-equivariant field theory. This recovers (the unextended version of) a result of M\"uller and Woike~\cite{MR4239182}. 
\end{example}

\section{Semisimple field theories and stable diffeomorphism} \label{sec:semisimple}

The goal of this section is two-fold. In~\cite{Reutter:2020aa}, the first author showed that \emph{semisimple} oriented four-dimensional field theories valued in the category of vector spaces lead to stable diffeomorphism invariants. Our first aim is to generalize this to higher dimensions, more general tangential structures, and to TFTs valued in super vector spaces. These results give an upper bound on what this class of topological field theories can detect about smooth manifolds. Our second aim is to show that finite path integral theories satisfy this semisimplicity condition. 

As in~\cite{Reutter:2020aa}, an oriented four-dimensional field theory is defined to be semisimple, if the algebras $Z(S^3)$ and $Z(S^2\times S^1)$, with multiplication and unit given as follows, respectively, are semisimple:
\begin{equation}
\label{eq:S3alg}
Z(D^4 \backslash{} (D^4\sqcup D^4) )
\hspace{2.5cm} 
Z(D^4)
\end{equation}
\begin{equation}
\label{eq:S2S1alg}
Z\left((D^3 \backslash{} D^3 \sqcup D^3) \times S^1\right) 
\hspace{1.5cm}
Z\left(D^3 \times S^1\right) 
\end{equation}
A physical interpretation of these algebra objects is given in~\cite[Rem 2.2]{Reutter:2020aa}, similar interpretations apply to the algebra objects in the rest of this section.

In order to generalize these results to $B$-structured field theories, the spheres and bordisms inducing the unital multiplication must all admit $B$-structures. Likewise, in order to define stable diffeomorphism of $B$-manifolds, one needs to have a well-defined notion of connected sum, and for the product of middle dimensional spheres to admit a canonical $B$-structure. These two issues are closely related, and, as we will see shortly, will require us to make mild assumptions on the ambient tangential structure $B$.

\subsection{Connected sum of $B$-manifolds and stable diffeomorphism}\label{sec:connectsumBFrob}

Let $\beta: B \to \rB O(d)$ be a tangential structure. The tangent bundle of the disk $D^d$ is trivial, and thus it is always possible to give the disk a $B$-structure. In fact if we assume that $B$ is connected, up to isotopy $B$-structures on the disk correspond to elements of $\pi_0F$, where $F$ is the homotopy fiber of $\beta$, and this can either be one element or two depending on whether $\pi_1 B \to \pi_1 \rB O(d)$ is surjective or not. Equivalently it depends on whether the pullback of the first Stieffel-Whitney class $\beta^*w_1$ is non-trivial or trivial. In the latter case the two possible isotopy classes of lifts result in diffeomorphic $B$-manifolds which are exchanged by an orientation reversing diffeomorphism.

The $(d-1)$-sphere $S^{d-1}$ is the boundary of $D^d$, but to get an induced $B$-structure on $S^{d-1}$, we must choose either an inward pointing or outward pointing of the normal bundle of $S^{d-1}$ in $D^d$. These correspond to viewing the disk either as a bordism from the empty manifold to the sphere, or as a bordism from the sphere to the empty manifold. There are thus potentially two canonical diffeomorphism classes of $B$-structures on the $(d-1)$-sphere --- the \emph{bounding} and the \emph{cobounding} $B$-structures. Without further restrictions on $B$ it is possible for these to be distinct. 
For example, when the tangential structure $\beta$ is tangential framing $*\to \rB O(d)$, then the bounding and cobounding framings on $S^{d-1}$ only agree for $d=0,1,3,7$, i.e. precisely when $S^{d}$ admits a tangential framing.

The operation of connected sum requires removing disks from the manifolds in question and then replacing the result with a cylinder. In the presence of tangential $B$-structures this requires that the disks, cylinders, and their boundary spheres all have compatible $B$-structures. In particular, the bounding and cobounding $B$-structures on $S^{d-1}$ need to coincide. For this we need to place some restrictions on the ambient tangential structure $B$. 

\begin{definition} A \emph{connected tangential $n$-type} $\beta:B \to \rB O(d)$ will be a tangential $n$-type (see Definition~\ref{def:tangentialntype}) for which $B$ is connected. \end{definition}

\begin{lemma}\label{lem:Bsphere}
	Let $\beta: B \to \rB O(d)$ be a connected tangential $(d-2)$-type. Then, $S^{d-1}$ admits a unique $B$-structure up to diffeomorphism. This $B$-structure is both bounding and cobounding; it extends both over the incoming and outgoing disk $D^d$. The gluing of these $B$-disks defines the unique $B$-structure on $S^d$ up to diffeomorphism. \end{lemma}
	\begin{proof}
	Obstruction theory shows that $S^{d-1}$ admits a $B$-structure and that there can be at most  two (up to isomorphism) depending on whether $\beta^*w_1$ is non-trivial or trivial. We will now show that these $B$-structures agree up to diffeomorphism. The disk $D^d$ admits precisely two $B$-structures up to isomorphism which are exchanged by an orientation reversing diffeomorphism. Hence, by uniqueness, the two $B$-structures on $S^{d-1}$ must bound these $B$-structures on $D^d$ and are exchanged by the restriction of the orientation reversing diffeomorphism of $D^d$ to the boundary. Similarly, both $B$-structures on $S^{d-1}$ are cobounding. By obstruction theory, $S^d$ has a unique $B$-structure (up to diffeomorphism) which hence must be given by the gluing of the unique (up to diffeomorphism) $B$-structures on $D^d$.\end{proof}

In particular, if $\beta:B \to \rB O(d)$ is a connected tangential $(d-2)$-type, the \emph{connected sum} $M\#N$ of two connected $B$-manifolds may be defined by choosing a $B$-structure preserving embedding of a standard $B$-disk $D^d$ into $M$ and into $N$, excising them both, and gluing the resulting boundary spheres by the standard orientation reversing diffeomorphism.  Lemma~\ref{lem:Bsphere} guarantees that this procedure results in a well-defined $B$-manifold, which a priori depends on the choice of the embedded $B$-disks. However, like the usual connected sum, in some cases the operation is independent of these choices:

The first Stiefel-Whitney class $\beta^* w_1$ is trivial on $B$ if and only if $\beta :B \to \rB O(d)$ factors through $\rB SO(d)$, in which case a $(B, \beta)$-structure induces an orientation. In this case, a $B$-structure preserving embedding of a $B$-disk in particular preserves orientation, and we may appeal to the celebrated \emph{disk theorem} of Palais which states that any two orientation preserving embeddings of disks into a connected oriented manifold are ambient isotopic and hence also isotopic as embedded $B$-disks. 
 It follows that the connected sum is independent of the choice of embedded $B$-disk, just as in the usual oriented setting. 

When the first Stiefel-Whitney class $\beta^*w_1$ is non-trivial on $B$, i.e. when $\beta: B \to \rB O(d)$ cannot be factored through $\rB SO(d) \to \rB O(d)$,  then the connected sum operation may depend on the choice of embedded $B$-disk. This case mimics connected sum of possibly non-orientable manifolds. As in that situation, there are certain cases when the connected sum operation remains independent of the choice of embedded $B$-disks. If a connected manifold is \emph{non-orientable} then any two embedded disks are also ambient isotopic. Thus the connected sum $M\#N$ is well-defined if either of $M$ or $N$ is non-orientable. Likewise if either $M$ or $N$ admits an orientation reversing $B$-diffeomorphism, then the connected sum  $M\#N$ is independent of the choice of embedded $B$-disk. (This is the situation relevant for what follows below). However if $M$ and $N$ are orientable and neither admits an orientation reversing $B$-diffeomorphism, then result of the connected sum operation may genuinely depend on the choice of embedded $B$-disks with different choices yielding distinct, non-isomorphic $B$-manifolds.

We can also obtain canonical $B$-structures on products of spheres. As in Section~\ref{sec:compntypes}, we can compose $\beta: B \to \rB O(d)$ with the inclusion $\rB O(d) \to \rB O(d+2)$, and then factor the resulting map into an $n$-connected map followed by an $n$-truncated map:
\begin{equation*}
	B \to[n\textrm{-connected}] B_{d+2} \to[n\textrm{-truncated}] \rB O(d+2). 
	\end{equation*}
By Proposition~\ref{prop:pullbackstructures}, if $n \leq d-2$, the resulting diagram
\begin{center}
\begin{tikzpicture}
	\node (A) at (0,1.5) {$ B $};
	\node (B) at (2.5, 1.5) {$ B_{d+2} $};
	\node (C) at (0, 0) {$ \rB O(d) $};
	\node (D) at (2.5, 0) {$ \rB O(d+2) $};
	\draw [->] (A) to (B); 
	\draw [->] (A) to (C); 
	\draw [->] (B) to (D); 
	\draw [->] (C) to (D); 
	\node at (0.5, 1) {$\lrcorner$};	
\end{tikzpicture}
\end{center} 
is a homotopy pull-back diagram. Let $a + b = d$, and consider the disk $D^{a+1} \times D^{b+1}$. The disk admits a $B_{d+2}$-structure (which is unique up to orientation reversal). This restricts to a $B_{d+2}$-structure on $S^a \times S^b$, which is equivalent to a $B$-structure on $S^a \times S^b$ by virtue of the above square being a pullback. 
\begin{definition}\label{def:canonicalBstructure}
The $B$-structure on $S^a \times S^b$ so obtained will be called the \emph{canonical} $B$-structure and we will denote it by $\overline{\theta}$.\end{definition} Unlike the case of Lemma~\ref{lem:Bsphere}, this $B$-structure is in general not unique --- there can be many other $B$-structures on $S^a \times S^b$. 
We note that in the problematic case when the first Stieffel-Whitney class $\beta^*w_1$ is non-trivial on $B$, then there is a unique $B_{d+d}$-structure on $D^{a+1} \times D^{b+1}$, which is then necessarily preserved by any orientation reversing diffeomorphism. It follows, that in this situation $(S^a \times S^b, \overline{\theta})$ also admits an orientation reversing $B$-diffeomorphism, and so the operation of connected sum with $(S^a \times S^b, \overline{\theta})$ is well-defined.

\begin{definition}\label{def:stablediffeo}
	If $d=2q$, and $\beta: B \to \rB O(d)$ is a connected tangential $(d-2)$-type, then two $B$-structured connected manifolds $(M_1, \psi_1)$ and $(M_2, \psi_2)$ are \emph{$B$-stably diffeomorphic} if there exists an natural number $N$ and a diffeomorphism of $B$-manifolds
	\begin{equation*}
		(M_1, \psi_1) \# \left( \#^N (S^q \times S^q, \overline{\theta}) \right) \cong (M_2, \psi_2) \# \left( \#^N (S^q \times S^q, \overline{\theta}) \right).
	\end{equation*} 
	where $\overline{\theta}$ is the canonical $B$-structure on $S^q \times S^q$.
\end{definition}

Assuming sufficient additional truncatedness of $\beta$, the canonical $B$-structure on $S^q \times S^q$ used in Definition~\ref{def:stablediffeo} is indeed unique and may be omitted from the notation.
\begin{lemma}\label{lem:uniquecanonical}
	If $d=2q$, and $\beta: B \to \rB O(d)$ is a tangential $(q-1)$-type with $B$ connected, then the canonical $B$-structure on $S^q \times S^q$ is unique (up to orientation reversal). \qed
\end{lemma}

\subsection{Local operators control decompositions of topological field theories}\label{sec:localoperatorsdecomp}
The \emph{algebra of local operators} of a $d$-dimensional oriented TFT is the Frobenius algebra $Z(S^{d-1})$ with unit $Z(D^d)$ and multiplication $Z(D^d\backslash{} (D^d \sqcup D^d))$. This algebra may be generalized to other $B$-structures as follows:
\begin{lemma}\label{lem:sphereFrobAlg}
	Let $\beta: B \to \rB O(d)$ be a connected tangential $(d-2)$-type. Then for any topological field theory $Z: \Bord_d^B \to \sVec$, the value of the $(d-1)$-sphere $Z(S^{d-1})$ (with its unique $B$-structure) is a super (commutative for $d>1$) Frobenius algebra with multiplication and unit given as follows:
\begin{equation*}
Z(D^d \backslash{} (D^d\sqcup D^d) )
\hspace{2.5cm} 
Z(D^d).
\end{equation*}
\end{lemma}

\begin{proof}
Assuming that $B$ is connected and $\beta: B \to \rB O(d)$ is a tangential $(d-2)$-type, consider the unique (up to orientation-reversal) $B$-structure on $S^d$.
Removing a number of incoming and outgoing disks results in a bordism from $\sqcup_m S^{d-1}$ to $\sqcup_\ell S^{d-1}$ where each of these $(d-1)$-spheres is given the unique (up to orientation-reversal) $B$-structure which is simultaneously bounding and cobounding. Analogous to the oriented case, these $B$-bordisms assemble into the structure of a commutative Frobenius algebra. \end{proof}

Following a theorem of Sawin~\cite{MR1359651}, we will now show that this algebra $Z(S^{d-1})$ controls direct sum decomposition of the field theory.
Given a finite family of field theories $Z_i: \Bord_d^B \to \sVec$, $i\in I$, we follow~\cite{MR1295467,MR1359651} and define their \emph{direct sum} $\bigoplus_{i \in I} Z_i$ as the following field theory: To a non-empty connected $(n-1)$-manifold $M$, it assigns the super vector space $\bigoplus_{i \in I} Z_i (M)$  and the tensor product of these to non-connected manifolds. Similarly, to a non-empty connected compact bordism $W$, the field theory $\bigoplus_{i \in I} Z_i $ assigns the direct sum $\bigoplus_{i \in I} Z_i (W)$, interpreted appropriately as a map between the appropriate tensor products of direct sums.  

Direct sums decompositions of a field theory are determined by direct sum decompositions of the algebra $Z(S^{d-1})$.

\begin{lemma}\label{lem:decomposablealgtotft}
	Let $\beta: B \to \rB O(d)$ be a connected tangential $(d-2)$-type and let $Z: \Bord_d^B \to \sVec$ be a topological field theory.	Suppose that there is a direct sum decomposition $Z(S^{d-1}) \iso \bigoplus_{i \in I} A_i$ of super-algebras. Then, $Z$ admits a direct sum decomposition $Z\iso \bigoplus_{i \in I} Z_i$ into field theories with algebra isomorphisms $Z_i(S^{n-1}) \iso A_i$. 
\end{lemma}

\begin{proof}
	In the oriented case this is \cite[Thm.~1]{MR1359651}. The proof in the presence of general $B$-structure (satisfying the assumptions of Lemma~\ref{lem:Bsphere}) is identical. 
\end{proof}

\begin{remark} Both the definition of the direct sum of field theories, and Lemma~\ref{lem:decomposablealgtotft} hold more generally for arbitrary semi-additive and idempotent complete target categories replacing the category of super vector spaces.
\end{remark}

\begin{lemma}\label{lem:semisimplesuperalgebra} A finite-dimensional super-commutative semisimple superalgebra is isomorphic to a direct sum of trivial algebras $k$ concentrated in purely even degree. 
\end{lemma}
\begin{proof} By super-commutativity, any odd element is nilpotent. However, by semisimplicity there cannot be any non-zero nilpotent elements. Hence, the algebra is even and hence is of the claimed form by Artin-Wedderburn. 
\end{proof}

\begin{definition} Let $\beta:B \to \rB O(d)$ be a connected tangential $(d-2)$-type. We will say that a field theory is \emph{simple} if the super vector space $Z(S^{d-1})$ is isomorphic to $k^{1|0}$, i.e. is one-dimensional and in even degree. We will say that a field theory is \emph{indecomposable} if it cannot be decomposed into a direct sum of non-zero field theories.\end{definition}

Using the classification of semisimple super-algebras, indecomposability and simplicity of a sufficiently semisimple super field theory are equivalent:
\begin{corollary} \label{cor:decomposesemisimple} 
Let $d\geq 2$,  $\beta: B \to \rB O(d)$ be a connected tangential $(d-2)$-type and let $Z: \Bord_{d}^{B} \to \mathrm{sVec}$ be a $d$-dimensional topological field theory for which $Z(S^{d-1})$ is a semisimple algebra. 

Then, $Z$ is 
 indecomposable if and only if it is simple.  Moreover, any $d$-dimensional topological field theory with semisimple $Z(S^{d-1})$ decomposes into a finite direct sum of simple field theories. 
\end{corollary} 
\begin{proof} Immediate by Lemmas~\ref{lem:semisimplesuperalgebra} and~\ref{lem:decomposablealgtotft} and the observation that a summand of a semisimple algebra is again a semisimple algebra.
\end{proof}

We recall from~\cite[Prop 3.2]{Reutter:2020aa} that simple field theories are multiplicative under connected sums. 
\begin{prop} \label{prop:multconnectsum}
	Let $\beta:B \to \rB O(d) $ be a connected tangential $(d-2)$-type and let $Z: \Bord^B_d \to \cC$ be a simple topological field theory.  Then, $Z(S^d)$ is invertible and $Z$ is multiplicative under connected sums: For a connected closed $B$-structured $d$-manifold $M$ and a connected $d$-dimensional $B$-bordism $W:A \to B$, the following holds, where the connected sum is taken in the interior of $W$:
\[Z(M\# W) = Z(S^n)^{-1} Z(M) Z(W)
\]
\end{prop}
\begin{proof}The proof is the same as \cite[Prop 3.2]{Reutter:2020aa}. 
\end{proof}

\subsection{Semisimple topological field theories}

In Section~\ref{sec:localoperatorsdecomp}, we used the unique $B$-structure on $S^d$ to obtain a commutative Frobenius algebra structure on $S^{d-1}$. In a similar way, for $a+b =d$, we can use the canonical $B$-structure $\overline{\theta}$ on $S^a\times S^b$ from Definition~\ref{def:canonicalBstructure} to obtain a Frobenius algebra structure on the value of $S^{a-1} \times S^b$. Specifically, we may decompose $S^a$ as a union of an incoming disk followed by an outgoing disk, and then take the product with $S^b$. This gives rise to a $B$-structure on $S^{a-1} \times S^b$ which we call the \emph{bounding-cobounding} $B$-structure, which we will also denote $\overline{\theta}$. Removing disks from $S^{a}$, and crossing with $S^b$ give $B$-structured bordisms which are easily seen to satisfy the axioms of a Frobenius algebra object which is commutative for $a \geq 2$. 

\begin{lemma}\label{lem:Frobspheres}
	Let $\beta: B \to \rB O(d)$ be a connected tangential $(d-2)$-type. Then for any topological field theory $Z: \Bord_d^B \to \sVec$, and any $a + b = d$, the value of $Z(S^{a-1}\times S^b, \overline{\theta})$ is a (commutative for  $a\geq 2$) super Frobenius algebra with multiplication and unit given as 
	\begin{equation}
	Z\left((D^a \backslash{} D^a \sqcup D^a) \times S^b, \overline{\theta}\right) 
	\hspace{1.5cm}
	Z\left(D^a \times S^b,\overline{\theta}\right). \qedhere
	\end{equation}
\end{lemma}

We may now define \emph{semisimple} even-dimensional topological field theories. 

\begin{definition}\label{def:semisimpleTFT}
	Let $d=2q$ and let $\beta: B \to \rB O(d)$ be a connected tangential $(q-1)$-type. Then a topological field theory $Z: \Bord_d^B \to \sVec$, is called \emph{semisimple} if the algebras $Z(S^{2q-1}, \overline{\theta})$ and $Z(S^q\times S^{q-1}, \overline{\theta})$, with multiplication and unit given as in Lemmas~\ref{lem:sphereFrobAlg} and~\ref{lem:Frobspheres}, are semisimple. 
\end{definition}

We will show in Section~\ref{sec:semisimpletostablediffeo} that semisimple topological field theories lead to stable diffeomorphism invariants. 

\begin{example}\label{exm:TFTs}
	The examples in \cite{Reutter:2020aa} of semisimple TFTs still apply. Specifically (\cite[Thm.~2.9]{Reutter:2020aa}) any unitary topological field theory is semisimple, and (\cite[Thm.~2.10]{Reutter:2020aa}) any once-extended topological field theory valued in the $2$-category of Cauchy complete $k$-linear categories and $k$-linear functors, or the $2$-category of $k$-algebras and $k$-bimodules is semisimple. 
\end{example}

\begin{remark}As stated, Definition~\ref{def:semisimpleTFT} would make sense for arbitrary tangential structure $\beta:B \to \rB O(d)$. However, in the case where $\beta$ is not $(q-1)$-truncated, we do not expect this to lead to a good notion of semisimplicity. Indeed, one essential reason that Definition~\ref{def:semisimpleTFT} imposes a considerable constraint on our field theory is that the middle-dimensional sphere $S^q$ admits a unique $B$-structure which we access via the canonical $B$-structure $\overline{\theta}$ on $S^q\times S^{q-1}$. For arbitrary tangential structures, we expect a better notion of semisimplicity to impose conditions on all $B$-structured middle-dimensional spheres $S^q$. 

For our current comparison with the notion of stable diffeomorphism, we only need to consider tangential $(q-1)$-types $(B, \beta)$ for which Definition~\ref{def:semisimpleTFT} suffices.
\end{remark}

\subsection{An eigenvalue equation in the bordism category}

In the next section, we will generalize the arguments of~\cite{Reutter:2020aa} and show that the manifold invariant induced by any even-dimensional semisimple field theory only depends on the stable diffeomorphism class of the manifold. 
Following~\cite[Prop. 3.4]{Reutter:2020aa}, the main step of this argument involves a diffeomorphism between certain $B$-bordisms which we now construct.

For a Frobenius algebra the composition of the comultiplication followed by the multiplication is called the window map. For example, under the assumptions of Lemma~\ref{lem:Bsphere}, we have seen in Section~\ref{sec:localoperatorsdecomp} that the sphere $S^{d-1}$ becomes a Frobenius algebra object with comultiplication and multiplication given by $D^d\backslash (D^d \sqcup D^d)$ with appropriate $B$-structure, seen as a bordism from $S^{d-1} \sqcup S^{d-1} \to S^{d-1}$, or $S^{d-1} \to S^{d-1} \sqcup S^{d-1}$, respectively. Hence, its associated \emph{window bordism} $W_{S^{d-1}}: S^{d-1} \to S^{d-1}$ is diffeomorphic to 
\begin{equation*}
	W_{S^{d-1}} := (D^d \backslash D^d \sqcup D^d) \cup_{S^{d-1} \sqcup S^{d-1}} (D^d\backslash (D^d \sqcup D^d) \cong (S^{d-1} \times S^1) \backslash{} (D^d \sqcup D^d). 
\end{equation*}
More generally, under the same assumptions and for  every $a + b = d$,  we have seen in Lemma~\ref{lem:Frobspheres} that $S^{a-1} \times S^b$ becomes a Frobenius algebra object with \emph{window bordism} $S^{a-1} \times S^b \to S^{a-1} \times S^b$: 
\begin{equation*}
	W_{S^{a-1} \times S^b} := W_{S^{a-1}} \times S^b \cong ((S^{a-1} \times S^1) \backslash{} (D^a \sqcup D^a)) \times S^b
\end{equation*}

The first goal of this section is to prove the following proposition:
\begin{proposition}\label{pro:importantdiffeo}
	Let $\beta:B \to \rB O(d)$ be a connected tangential $(d-2)$-type. Then, we have a diffeomorphism of $B$-structured bordisms $\emptyset \To (S^{a-1} \times S^b, \overline{\theta})$:
	\begin{equation*}
		\left( W_{S^{a-1}} \times S^b, \overline{\theta} \right) \circ \left( S^{a-1} \times D^{b+1}, \overline{\theta} \right) \cong \left( S^{a-1} \times D^{b+1}, \overline{\theta} \right) \#  \left( S^{a-1} \times S^{b+1}, \theta_2\right) \# \left(S^1 \times S^{d-1}, \theta_1 \right) 
	\end{equation*}
	where $\#$ denotes connected sum on the interior, each $\overline{\theta}$ denotes a canonical $B$-structure from Section~\ref{sec:connectsumBFrob}, and $\theta_1$ and $\theta_2$ are some $B$-structures on  $S^{a-1} \times S^{b+1}$ and $S^1 \times S^{d-1}$, respectively. 
\end{proposition}
\begin{remark} 
Connected sum is only unambiguously defined for connected manifolds, and hence the right hand side of Proposition~\ref{pro:importantdiffeo} is only unambiguous for $a>1$. However, although we are not going to use this case in the rest of the paper, there is still a diffeomorphism (constructed in the proof of Proposition~\ref{pro:importantdiffeo}) if the connected sum is taken in the correct components. 
\end{remark}

As in~\cite{Reutter:2020aa}, interpreting the connected sum as a `scalar multiplication' operation, the diffeomorphism of Proposition~\ref{pro:importantdiffeo} may be understood as an `eigenvalue equation' in the bordism category; asserting that the endomorphism $W_{S^{a-1} \times S^b}$ has `eigenvector' $S^{a-1} \times D^{b+1}$ with eigenvalue $\left( S^{a-1} \times S^{b+1}\right)\#\left(S^1 \times S^{d-1} \right) $.

\begin{lemma}\label{lem:movetoconnectifnull}
	Let $M$ be a smooth $d$-dimensional manifold, $a+b = d$, and $i:D^a \times S^b \hookrightarrow M$ an embedding. Suppose that $i$ is isotopic to an embedding so that the image lies in an embedded $d$-disk and the corresponding $S^b$ is unknotted inside this disk. Then there is a diffeomorphism:
	\begin{equation*}
		\left(M \backslash{} D^a \times S^b \right) \cong \left(S^{a-1} \times D^{b+1}\right) \# M.
	\end{equation*} 
	Here $\#$ denotes connected sum on the interior. 
\end{lemma}

\begin{proof}
	By assumption we may change $i$ by an isotopy so that its image lies inside a disk. Thus
	\begin{equation*}
		\left(M \backslash{} D^a \times S^b \right)  \cong \left(S^d \backslash{} i(D^a \times S^b) \right) \# M .
	\end{equation*}
	Since the image of the core $S^b$ is assumed to be unknotted, we have $\left(S^d \backslash{} i(D^a \times S^b) \right) \cong \left(S^{a-1} \times D^{b+1}\right)$, which proves the result. 
\end{proof}

\begin{corollary}\label{cor:closingmanifolds}
		There exists a diffeomorphism of smooth manifolds:
		\begin{multline*}
			\left( W_{S^{a-1}} \times S^b\right) \circ \left( S^{a-1} \times D^{b+1}\right) \\ \cong \left(S^{a-1} \times D^{b+1}\right) \# \left[ \left(S^{a-1} \times S^1 \times S^b \backslash{} D^{a} \times S^b \right) \cup_{S^{a-1} \times S^b} S^{a-1} \times D^{b+1} \right].
		\end{multline*}
		
\end{corollary}
\begin{proof}
	The manifold $W_{S^{a-1}} \times S^b$ is $S^{a-1} \times S^1 \times S^b$ with two parallel copies of $D^a \times S^b$ removed. In $\left( W_{S^{a-1}} \times S^b\right) \circ \left( S^{a-1} \times D^{b+1}\right)$, a copy of $S^{a-1} \times D^{b+1}$ has been attached to the boundary of one of these $D^a \times S^b$ (i.e. we preformed surgery on one of these $D^a \times S^b$). In the resulting manifold, the core $\{0\} \times S^b$ of the remaining  $D^a \times S^b$ now bounds an embedded $(b+1)$-disk. Thus it can be isotoped into a $d$-disk in which it is unknotted. Now Lemma~\ref{lem:movetoconnectifnull} applies and gives the desired result. 
\end{proof}

Note that the second summand on the right hand side of Corollary~\ref{cor:closingmanifolds} is the manifold obtained from $S^{a-1} \times S^1 \times S^b$ by surgery along the embedded $\{0\} \times \{0\} \times S^b$.
\begin{lemma}\label{lem:keydiffeoclosed}
The manifold $\left(S^{a-1} \times S^1 \times S^b \backslash{} D^{a} \times S^b \right) \cup_{S^{a-1} \times S^b} S^{a-1} \times D^{b+1}$ obtained from surgery along $\{0\} \times \{0\} \times S^b \hookrightarrow S^{a-1}\times S^1 \times S^b$ is diffeomorphic to 
	\begin{equation*}
		 S^{a-1} \times S^{b+1}\# S^1 \times S^{d-1}  .
	\end{equation*}
\end{lemma}

\begin{proof}
We will begin by writing $S^{a-1} = D^{a-1}_+ \cup_{S^{a-2}} D^{a-1}_-$ and $S^1 = D^1_+ \cup_{S^0} D^1_-$ as unions of upper and lower hemispheres. This gives the decomposition depicted in Figure~\ref{fig:Torus} of $S^{a-1} \times S^1 \times S^b$ as a gluing of manifolds with corners.
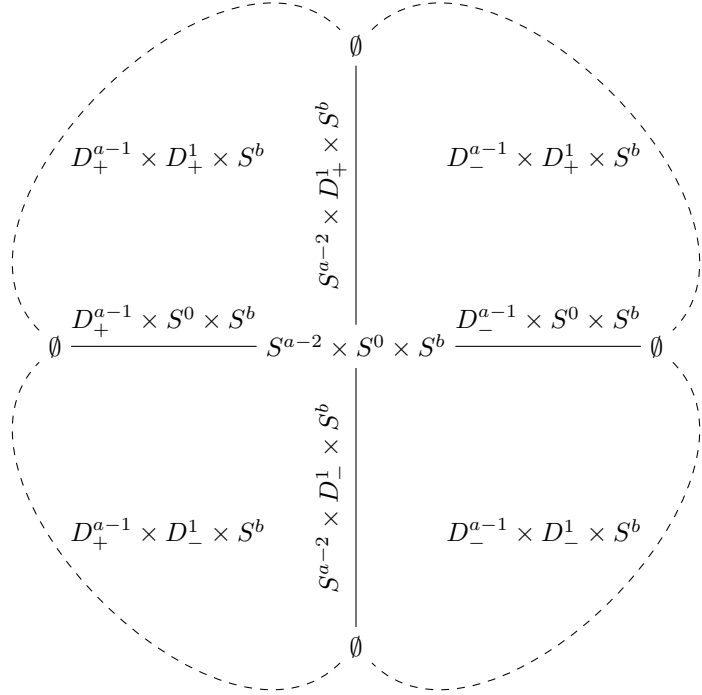
\begin{figure}[ht]
	\begin{center}
	\begin{tikzpicture}
		 \clip (-4.75,-4.75) rectangle (4.75, 4.75);
			\node (C) at (0, 0) {$ S^{a-2} \times S^0  \times S^b$};
			\node (W) at (-4, 0) {$ \emptyset $};
			\node (N) at (0, 4) {$ \emptyset $};
			\node (E) at (4, 0) {$\emptyset $};
			\node (S) at (0, -4) {$\emptyset $};
			\draw [-] (W) -- node [above] {$ D^{a-1}_+ \times S^0 \times S^b $} (C);
			\draw [-] (N) -- node [above, rotate=90] {$ S^{a-2} \times D^1_+ \times S^b$} (C);
			\draw [-] (E) -- node [above] {$ D^{a-1}_- \times S^0 \times S^b $} (C);
			\draw [-] (S) -- node [above, rotate=90] {$ S^{a-2} \times D^1_- \times S^b $} (C);
		
			\draw [dashed] (W) to [out=135, in = 135] (N);
			\node at (-2.5,2.5) {$D^{a-1}_+ \times D^1_+ \times S^b$};
		
			\draw [dashed] (N) to [out=45, in = 45] (E);
			\node at (2.5,2.5) {$D^{a-1}_- \times D^1_+ \times S^b$};
		
			\draw [dashed] (E) to [out=315, in = 315] (S);
			\node at (2.5,-2.5) {$D^{a-1}_- \times D^1_- \times S^b$};
		
			\draw [dashed] (S) to [out=225, in = 225] (W);
			\node at (-2.5,-2.5) {$D^{a-1}_+ \times D^1_- \times S^b$};
		
	\end{tikzpicture}
	\end{center}
	\caption{A decomposition of $S^{a-1} \times S^1 \times S^b$.}
	\label{fig:Torus}
\end{figure}
The surgered manifold is obtained from $S^{a-1} \times S^1 \times S^b$ by removing a copy of $D^a \times S^b$ and replacing it with $S^{a-1} \times D^{b+1}$ (i.e. doing surgery). The copy of $D^a \times S^b$ that we shall use is $D^{a-1}_- \times D^1_- \times S^b$, denoted in the lower right-hand corner of the above decomposition. 
We also have 
\begin{align*}
	S^{a-1} \times D^{b+1} &\cong \left( S^{a-2} \times D^1_- \cup_{S^{a-2} \times S^0} D^{a-1}_- \times S^0 \right) \times D^{b+1} \\
	& \cong \left( S^{a-2} \times D^1_- \times D^{b+1} \right) \cup_{S^{a-2} \times S^0 \times D^{b+1}} \left(D^{a-1}_- \times S^0 \times D^{b+1} \right)
\end{align*}	
and so our surgered manifold is given as the decomposition depicted in Figure~\ref{fig:Torusfirstsurgery} as a gluing of manifolds with corners.

\begin{figure}[ht]
	\begin{center}
	\begin{tikzpicture}
		 \clip (-4.75,-5.5) rectangle (5.5 ,4.75);
			\node (SE) at (4,-4) {$\emptyset$};
			\fill [black!5] (-4,0) to [out=135, in = 135] (0,4) to [out=45, in = 45] (4,0) to [out = 0, in = 0] (4,-4) -- (0,0) -- (-4,0);
		
		
			\fill [black!5] (W) to [out=135, in = 135] (N) -- (W);
			\fill [black!5] (N) to [out=45, in = 45] (E) -- (N);
			\fill [black!5] (E) to [out = 0, in = 0] (SE) -- (E);
		
			\node (C) at (0, 0) {$ S^{a-2} \times S^0  \times S^b$};
			\node (W) at (-4, 0) {$ \emptyset $};
			\node (N) at (0, 4) {$ \emptyset $};
			\node (E) at (4, 0) {$\emptyset $};
			\node (S) at (0, -4) {$\emptyset $};
		
			\node (SE) at (4,-4) {$\emptyset$};

			\draw [-] (W) -- node [above] {$ D^{a-1}_+ \times S^0 \times S^b $} (C);
			\draw [-] (N) -- node [above, rotate=90] {$ S^{a-2} \times D^1_+ \times S^b$} (C);
			\draw [-] (E) -- node [above] {$ D^{a-1}_- \times S^0 \times S^b $} (C);
			\draw [-] (S) -- node [above, rotate=90] {$ S^{a-2} \times D^1_- \times S^b $} (C);
		
			\draw [dashed] (W) to [out=135, in = 135] (N);
			\node at (-2.5,2.5) {$D^{a-1}_+ \times D^1_+ \times S^b$};
		
			\draw [dashed] (N) to [out=45, in = 45] (E);
			\node at (2.5,2.5) {$D^{a-1}_- \times D^1_+ \times S^b$};
		
		
			\draw [dashed] (S) to [out=225, in = 225] (W);
			\node at (-2.5,-2.5) {$D^{a-1}_+ \times D^1_- \times S^b$};

			\draw (C) -- node [below, rotate=-45] {$ S^{a-2} \times S^0 \times D^{b+1}$} (SE);
			\node at (3.5,-1.5) {$D^{a-1}_- \times S^0 \times D^{b+1}$};
			\draw [dashed] (E) to [out = 0, in = 0] (SE);
			\node at (2,-4) {$S^{a-2} \times D^1_- \times D^{b+1}$};
			\draw [dashed] (SE) to [out = -90, in = -90] (S);
	\end{tikzpicture}
	\end{center}
	\caption{A decomposition of the manifold $L$.}
	\label{fig:Torusfirstsurgery}
\end{figure}
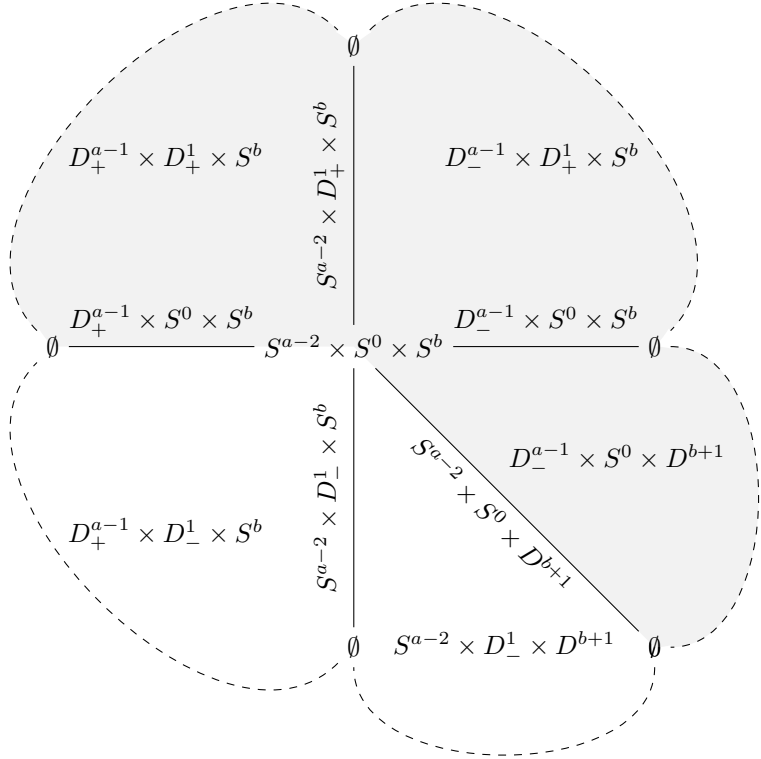

Let $X$ be the manifold corresponding to the shaded portion of the decomposition in Figure~\ref{fig:Torusfirstsurgery}. The manifold corresponding to the unshaded portion is diffeomorphic to the cylinder $S^{d-1} \times D^1_-$. Since the surgered manifold is orientable, the two disks along which the cylinder is attached to $X$ are isotopic.  Thus to prove the lemma it is sufficient to show that $X$ (corresponding to the shaded region) is diffeomorphic to $ (S^{a-1} \times S^{b+1}) \backslash (D^d \times S^0) $. This in turn follows if $X \cup_{S^{d-1} \times S^0} (D^d \times S^0)$ is diffeomorphic to $S^{a-1} \times S^{b+1}$, which is what we will show.

The manifold $X \cup_{S^{d-1} \times S^0} (D^d\times S^0)$ is obtained by taking the manifold from the shaded region of the decomposition in Figure~\ref{fig:Torusfirstsurgery} and attaching $D^d \times S^0 \cong D^{a-1}_+ \times S^0 \times D^{b+1}$ in place of the unshaded region. This gives the decomposition of $X \cup_{S^{d-1} \times S^0 } (D^d \times S^0)$ as a gluing of manifolds with corners depicted in Figure~\ref{fig:ManifoldXwithdisks}.

\begin{figure}[ht]
	\begin{center}
	\begin{tikzpicture}
		 \clip (-4.75,-5.6) rectangle (5.5 ,4.75);
			\fill [black!5]  (0,4) to [out=45, in = 45] (4,0) to [out = 0, in = 0] (4,-4) -- (0,0) -- (0,4);
			\fill [black!5] (N) to [out=45, in = 45] (E) -- (N);
			\fill [black!5] (E) to [out = 0, in = 0] (SE) -- (E);
		
			\node (C) at (0, 0) {$ S^{a-2} \times S^0  \times S^b$};
			\node (W) at (-4, 0) {$ \emptyset $};
			\node (N) at (0, 4) {$ \emptyset $};
			\node (E) at (4, 0) {$\emptyset $};
		
			\node (SE) at (4,-4) {$\emptyset$};

			\draw [-] (W) -- node [above] {$ D^{a-1}_+ \times S^0 \times S^b $} (C);
			\draw [-] (N) -- node [above, rotate=90] {$ S^{a-2} \times D^1_+ \times S^b$} (C);
			\draw [-] (E) -- node [above] {$ D^{a-1}_- \times S^0 \times S^b $} (C);
			\draw (C) -- node [below, rotate=-45] {$ S^{a-2} \times S^0 \times D^{b+1}$} (SE);

			\draw [dashed] (W) to [out=135, in = 135] (N);
			\node at (-2.5,2.5) {$D^{a-1}_+ \times D^1_+ \times S^b$};
		
			\draw [dashed] (N) to [out=45, in = 45] (E);
			\node at (2.5,2.5) {$D^{a-1}_- \times D^1_+ \times S^b$};
		
		
			\node at (3.5,-1.5) {$D^{a-1}_- \times S^0 \times D^{b+1}$};
			\draw [dashed] (E) to [out = 0, in = 0] (SE);
		
			\draw [dashed] (SE) to [out=270, in = 225] (W);

			\node at (-1.5,-2) {$D^{a-1}_+ \times S^0 \times D^{b+1}$};

	\end{tikzpicture}
	\end{center}
	\caption{The manifold $X$ with two disks attached.}
	\label{fig:ManifoldXwithdisks}
\end{figure}
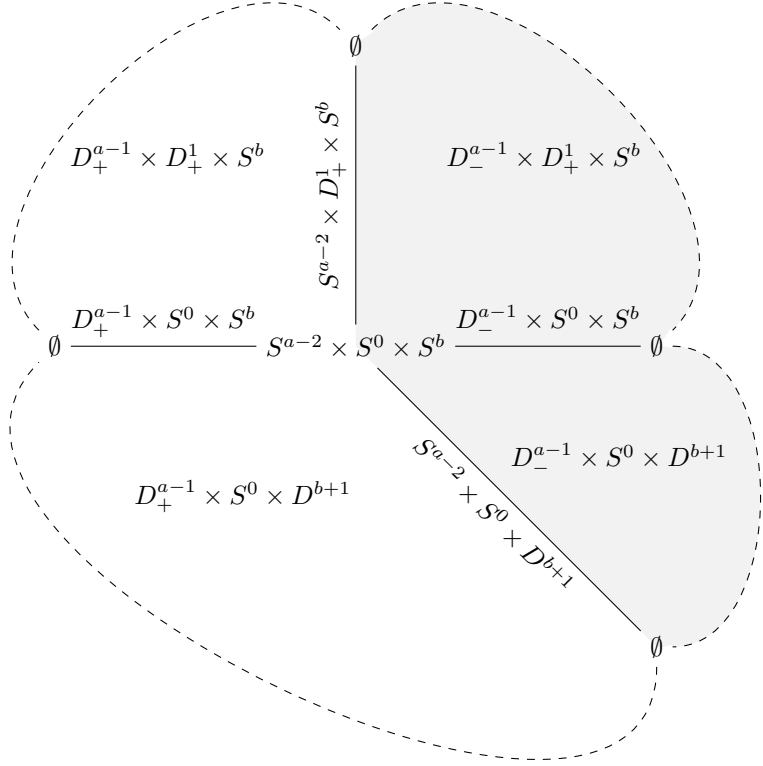

The manifold corresponding to the shaded half of the decomposition in Figure~\ref{fig:ManifoldXwithdisks} is $D^{a-1}_- \times S^{b+1}$. The manifold corresponding to the unshaded half is $D^{a-1}_+ \times S^{b+1}$ and they are glued together along $S^{a-2} \times \times S^{b+1}$. Whence the manifold in question is $S^{a-1} \times S^{b+1}$, as claimed. 
\end{proof}

\begin{proof}[Proof of Prop.~\ref{pro:importantdiffeo}]
	It follows immediately from Lemma~\ref{lem:keydiffeoclosed} and Corollary~\ref{cor:closingmanifolds} that there exists a diffeomorphism of bordisms
	\begin{equation*}
				\left( W_{S^{a-1}} \times S^b\right) \circ \left( S^{a-1} \times D^{b+1}\right)  \cong \left(S^{a-1} \times D^{b+1}\right) \#    S^{a-1} \times S^{b+1} \# S^1 \times S^{d-1}.
			\end{equation*} 
	The bordisms on the left-hand side can be equipped with canonical $B$-structures, and these induce $B$-structures on the right-hand side. The fact that these $B$-structures come from the connected sum of ones on $S^{a-1} \times D^{b+1}$, $S^1 \times S^{d-1}$, and $S^{a-1} \times S^{b+1}$ separately follows from the uniqueness of $B$-structures on $D^d$ and $S^{d-1}$ (implicit in the assumptions of Lemma~\ref{lem:Bsphere}). The proof of Lemma~\ref{lem:movetoconnectifnull} shows that the $B$-structure on $S^{a-1} \times D^{b+1}$ is induced by one on $S^{d}$, which must be the canonical $B$-structure, and hence is the canonical $B$-structure on  $S^{a-1} \times D^{b+1}$.  
\end{proof}

\subsection{Semisimple field theories lead to stable diffeomorphism invariants}\label{sec:semisimpletostablediffeo}

We now show that semisimple topological field theories $Z$ lead to stable diffeomorphism invariants. We first consider the case where $Z$ is furthermore simple, i.e. $Z(S^{d-1})\cong k$. For such theories, we may use the 'eigenvalue equation' of Proposition~\ref{pro:importantdiffeo}, to identify $Z(S^q\times S^q)$ as a (multiplicative factor of an) eigenvalue of an invertible endomorphism and hence conclude that it is non-zero.

\begin{lemma}\label{lem:nonzerotori}
	Let $d=2q$, let $\beta:B \to \rB O(d)$ be a connected tangential $(q-1)$-type, and let 
	 $Z:\Bord_d^B \to \sVec$ be a semisimple topological field theory with $Z(S^{d-1}) \cong k$. Then,  $Z(S^q \times S^q, \overline{\theta})$ is non-zero.\end{lemma}

\begin{proof}
	By definition, the Frobenius algebra $Z(S^q \times S^{q-1}, \overline{\theta})$ is semisimple. By \cite[Thm. 3.4]{MR1760598} (also see~\cite[Prop.~3.3]{Reutter:2020aa}) the window endmorphism $m \circ \Delta: A \to A$ is invertible for any semisimple commutative Frobenius algebra $A$ in $\mathrm{sVec}$. Whence,	
	\begin{equation*}
		Z(W_{S^q} \times S^{q-1}, \overline{\theta}): Z(S^q \times S^{q-1}, \overline{\theta}) \to Z(S^q \times S^{q-1}, \overline{\theta}) 
	\end{equation*}
	is invertible. Applying $Z$ to the diffeomorphism of Proposition~\ref{pro:importantdiffeo} and using multiplicativity under connected sums (Proposition~\ref{prop:multconnectsum}), we find:
	\begin{equation*}
		Z(W_{S^{q}} \times S^{q-1}, \overline{\theta}) \circ Z(S^q \times D^{q}, \overline{\theta}) = Z(S^d, \overline{\theta})^{-2} Z(S^{2q-1} \times S^1, \theta_1) Z(S^q \times S^q, \theta_2) \cdot Z(S^q \times D^{q}, \overline{\theta})
	\end{equation*} 
	for the respective canonical $B$-structures $\overline{\theta}$ and some $B$-structures $\theta_1, \theta_2$. By Lemma~\ref{lem:uniquecanonical}, the canonical $B$-structure on $S^q\times S^q$ is in fact the unique one up to orientation reversing diffeomorphism, and hence $(S^q\times S^q, \theta_2) \cong (S^q \times S^q, \overline{\theta})$.
	
	The composite morphism
	\begin{equation*}
		Z(\emptyset) \to[Z(S^q \times D^q, \overline{\theta})] Z(S^q \times S^{q-1}, \theta) \to [Z(D^{q+1} \times S^{q-1}, \overline{\theta})] Z(\emptyset)
	\end{equation*}
	is $Z(S^d, \overline{\theta})$, which is non-zero (Proposition~\ref{prop:multconnectsum}). Thus 
	\begin{equation*}
			k=Z(\emptyset) \to[Z(S^q \times D^q, \overline{\theta})] Z(S^q \times S^{q-1}, \theta)
	\end{equation*}
	 is a non-zero vector in $Z(S^q \times S^{q-1}, \theta)$. It follows that $Z(S^d, \overline{\theta})^{-2} Z(S^{2q-1} \times S^1, \theta_1) Z(S^q \times S^q, \overline{\theta})$ is an eigenvalue of the invertible endomorphism $Z(W_{S^{q}} \times S^{q-1}, \overline{\theta})$, and is therefore invertible itself. 
\end{proof}

Decomposing any semisimple field theory into simple summands, the main theorem of this section follows.
\begin{theorem}\label{thm:semisimpletostablediffoinv}
	Let $d=2q$, let $\beta:B \to \rB O(d)$ be a connected tangential $(q-1)$-type, and let 
	 $Z:\Bord_d^B \to \sVec$ be any semisimple topological field theory (Definition~\ref{def:semisimpleTFT}). If $B$-structured manifolds $(M_1, \psi_1)$ and $(M_2, \psi_2)$ are  $B$-stably diffeomorphic, then $Z(M_1, \psi_1) = Z(M_2, \psi_2)$.
\end{theorem}

\begin{proof}	
	Any semisimple topological field theory decomposes as a direct sum of simple topological field theories (Cor.~\ref{cor:decomposesemisimple}), and hence it is sufficient to assume $Z$ is simple. 
	By assumption there exists a natural number $N$ and a diffeomorphism of $B$-manifolds
	\begin{equation*}
		(M_1, \psi_1) \# \left( \#^N (S^q \times S^q, \overline{\theta}) \right) \cong (M_2, \psi_2) \# \left( \#^N (S^q \times S^q, \overline{\theta}) \right).
	\end{equation*} 
	where $\overline{\theta}$ is the canoncial $B$-structure on $S^q \times S^q$. By the multiplicativity of simple topological field theories (Proposition~\ref{prop:multconnectsum}) we have
	\begin{equation*}
		Z(S^d)^{-N} Z(S^q \times S^q, \overline{\theta})^N Z(M_1, \psi_1) = Z(S^d)^{-N} Z(S^q \times S^q, \overline{\theta})^N Z(M_2, \psi_2).
	\end{equation*}
	However $Z(S^d)$ is invertible (Proposition~\ref{prop:multconnectsum}) and $Z(S^q \times S^q, \overline{\theta})$ is invertible (Lemma~\ref{lem:nonzerotori}), and thus $Z(M_1, \psi_1) =Z(M_2, \psi_2)$.
\end{proof}

\subsection{Finite path integral theories are semisimple}

Our goal in this section is to show that finite path integral theories induced from invertible field theories are semisimple. 

Our main proof will rely on the fact that any finite path integral theory as in Definition~\ref{def:GenDWTheory} can be \emph{once extended} to a symmetric monoidal $2$-functor $\Bord^B_{d,d-1,d-2} \to \mathrm{2sVec}_k$. Here, $\Bord^B_{d,d-1,d-2}$ denotes the symmetric monoidal bicategory of closed $(d-2)$-$B$-manifolds, compact compatibly $B$-structured $(d-1)$-bordisms and compact compatibly $B$-structured $d$-dimensional bordisms with corners between them (see e.g.~\cite{SPthesis} for details on the construction as a bicategory), and $\mathrm{2sVec}_k$ denotes the symmetric monoidal bicategory of $k$-linear finite semisimple $\mathrm{sVec}_k$-module categories, module functors and module natural transformations. 

\begin{remark}This bicategory $\mathrm{2sVec}_k$ is equivalent to the symmetric monoidal bicategory $\mathrm{sAlg}_k$ of semisimple super algebras (i.e. separable algebra objects in $\mathrm{sVec}_k$), super bimodules and grading-preserving bimodule maps (this follows e.g. from~\cite[Thm 1]{MR1976459}).
\end{remark}

In Remark~\ref{rem:alternativessproof} below, we will sketch another more elementary proof which does not use any bicategorical machinery and instead proceeds by explicitly studying the algebras associated to products of spheres. 

We first show that invertible $\sVec_k$-valued field theories can be uniquely extended to $\mathrm{2sVec}_k$. (The analogous statement about extending $\mathrm{Vec}_k$-valued theories to $\mathrm{2Vec}_k$ fails spectacularly, see~\cite{Schommer-Pries:2017aa}.)
\begin{lemma}\label{lem:extend} Let $W:\Bord_d^X \to \mathrm{sVec}_k$ be any invertible topological field theory. Then, there exists a unique $2$-categorical invertible field theory $\widetilde{W}: \Bord_{d,d-1,d-2}^B \to \mathrm{2sVec}_k$ which extends $W$.
\end{lemma}

Before proving the lemma, we make a brief detour on Picard 2-groupoids. A Picard $2$-groupoid is a symmetric monoidal bicategory, in which all objects, $1$-morphisms and $2$-morphisms are invertible. Equivalently~\cite{MR3958095}, a Picard $2$-groupoid is a spectrum whose only non-trivial homotopy groups are $\pi_0, \pi_1, \pi_2$. 
Out of any Picard 2-groupoid $\cA$  we may extract two Picard 1-groupoids $\cA_{[1 2]}:= \tau_{\geq 1} \cA$ and $\cA_{[0 1]} := \tau_{\leq 1} \cA$. The former may be identified with the groupoid of automorphisms of the unit object of $\cA$, and the later is the groupoid whose objects are those of $\cA$ and whose morphisms are isomorphism classes of 1-morphisms in $\cA$. As in Section~\ref{sec:invertTFT}, these Picard 1-groupoids are determined by their \emph{Postnikov invariants} \cite[App.~B]{MR2192936} (see also \cite{MR2981952}), which take the form $(\pi_1\cA, \pi_2 \cA, k_{\cA_{[1 2]}}: \pi_1 \cA \otimes \ZZ/2\ZZ \to \pi_2 \cA)$ and  $(\pi_0 \cA, \pi_1 \cA, k_{\cA_{[0 1]}}: \pi_0 \cA \otimes \ZZ/2\ZZ \to \pi_1 \cA)$. We will call these the \emph{elementary Postnikov invariants} of $\cA$. 

Let $\mathrm{2sLine}_k$ denote the underlying Picard $2$-groupoid of the symmetric monoidal bicategory $\mathrm{2sVect}_k$. By definition, $\mathrm{2sVect}_k$ is a delooping of $\mathrm{sVect}_k$ and hence we have $(\mathrm{2sLine}_k)_{[12]} \simeq \sLine_k$,  considered in Example~\ref{ex:slines}. On the other hand, $ \pi_0\mathrm{2sLine}_k \cong \pi_0 (\mathrm{2sLine}_k)_{[01]} \cong \ZZ/2\ZZ$ corresponding to the super Morita equivalence classes of Clifford algebras, also known as the Brauer-Wall group or the super Brauer group \cite{MR167498}. A simple calculation shows that the $k$-invariant
\begin{equation*}
	k_{(\mathrm{2sLine}_k)_{[01]}}: \pi_0\mathrm{2sLine}_k \cong \ZZ/2\ZZ \to \ZZ/2\ZZ \cong \pi_1\mathrm{2sLine}_k = \pi_0 \mathrm{sLine}_k
\end{equation*} 
is an isomorphism:  The $Cl_1 \otimes Cl_1$-$Cl_1 \otimes Cl_1$ super bimodule corresponding to the graded twist automorphism turns into the odd line after employing the super Morita equivalence $Cl_1 \otimes Cl_1 \cong Cl_2 \sim k$.

In general, the elementary Postnikov invariants do not determine a Picard 2-groupoid $\cA$. However we have the following special case.

\begin{lemma}\label{lem:charBrownComenetz}
	Let $k$ be an algebraically closed field, not of characteristic 2. Let $\cA$ be a Picard 2-groupoid with $\pi_0 \cA \cong \ZZ/2\ZZ$, $\pi_1 \cA \cong \ZZ/2\ZZ$, $\pi_2 \cA \cong k^\times$, and with non-trivial elementary $k$-invariants 
	\begin{align*}
		k_{\cA_{[01]}}: &\pi_0\cA \cong \ZZ/2\ZZ \stackrel{\cong}{\to} \ZZ/2\ZZ \cong \pi_1 \cA \\
		k_{\cA_{[12]}}: &\pi_1\cA \cong \ZZ/2\ZZ \hookrightarrow k^\times\cong \pi_2 \cA 
	\end{align*} 
Then the underlying $E_{\infty}$-2-type of $\cA$ is equivalent to $\tau_{\geq 0} \Sigma^2 I_{k^{\times}}$, the connective cover of the two-fold suspension of the $k^\times$-Brown-Comenetz dual of the sphere.
\end{lemma}

\begin{proof}
	We have $\pi_0( \tau_{\geq 0} \Sigma^2 I_{k^{\times}}) \cong \ZZ/2\ZZ$, $\pi_1 (\tau_{\geq 0} \Sigma^2 I_{k^{\times}}) \cong \ZZ/2\ZZ$, and $\pi_2( \tau_{\geq 0} \Sigma^2 I_{k^{\times}}) \cong k^\times$. The universal property of the $k^\times$-Brown-Comenetz dual of the sphere implies that 
	\begin{equation*}
		\pi_0 \mathrm{2Fun}(\cA, \tau_{\geq 0} \Sigma^2 I_{k^{\times}}) \cong \Hom(\pi_2 \cA, k^\times),
	\end{equation*}
and thus there is a map $f:\cA \to \tau_{\geq 0} \Sigma^2 I_{k^{\times}}$ which induces an isomorphism on $\pi_2$. Because the elementary $k$-invariant $k_{\cA_{[12]}}$ is injective, it follows that the elementary $k$-invariant $k_{(\tau_{\geq 0} \Sigma^2 I_{k^{\times}})_{[12]}}$ is injective and that $f$ induces an isomorphism on $\pi_1$. Likewise, because the elementary $k$-invariant $k_{\cA_{[01]}}$ is an isomorphism, it follows that the elementary $k$-invariant $k_{(\tau_{\geq 0} \Sigma^2 I_{k^{\times}})_{[01]}}$ is an isomorphism and that $f$ induces an isomorphism on $\pi_0$. Since $f$ is an isomorphism on all homotopy groups, it is an equivalence of $E_{\infty}$-2-types.
\end{proof}

It follows from Lemma~\ref{lem:charBrownComenetz} that the Picard $2$-groupoid $\mathrm{2sLine}_k$ of $\mathrm{2sVec}_k$ has the underlying $E_{\infty}$-2-type of $\tau_{\geq 0} \Sigma^2 I_{k^{\times}}$. See \cite{BLM-preprint} for another proof of this and related results, as well as \cite{Freed-Lecturenotes, Freed:2014aa, MR4313235, MR4268163}.  In other words, the Picard $2$-groupoid $\mathrm{2sLine}_k$ can be characterized by the following universal property, analogous to the characterization of $\mathrm{sLine}_k$ in Corollary~\ref{cor:universalpropertysline}.
\begin{corollary}\label{cor:universalproperty2sline} Let $k$ be an algebraically closed field of characteristic $\neq 2$. Then, $\mathrm{2sLine}_k$ is the unique Picard $2$-groupoid with the property that for any Picard $2$-groupoid $\cA$, the restriction map \begin{equation}\label{eq:universalprop2sline}
	\pi_0 \mathrm{2Fun}(\cA, \mathrm{2sLine}_k) \to \Hom(\pi_2 \cA, k^\times) 
\end{equation}
 is an isomorphism. 
\end{corollary}

 This immediately leads to a proof of Lemma~\ref{lem:extend}.
\begin{proof}[Proof of Lemma~\ref{lem:extend}]
Let $\cB$ be the Picard completion of the 2-category $\Bord^X_{d, d-1, d-2}$. It follows from \cite{Schommer-Pries:2017aa} that $ \cB_{[12]}$ is the Picard completion of $\Bord^X_{d, d-1}$. Thus 
the restriction map 
\begin{equation*}
	\pi_0 \mathrm{2Fun}(\Bord^X_{d, d-1, d-2}, \mathrm{2sLine}_k) \to \pi_0 \Fun(\Bord^X_{d, d-1}, \sLine_k)
\end{equation*}
is an isomorphism, since they are both isomorphic to $\Hom(\Omega_d^{TX}, k^\times)$ by Corollaries~\ref{cor:universalpropertysline} and~\ref{cor:universalproperty2sline}.
\end{proof}

\begin{theorem}\label{thm:DWareSemisimple}
If $d=2q$ and $\beta:B \to \rB O(d)$ is a connected tangential $(q-1)$-type (i.e. the assumptions for Definitions~\ref{def:semisimpleTFT} and~\ref{def:stablediffeo}), then $\FP_{\xi, \omega}: \Bord^B_d \to \sVec_k$ is semisimple for any choice of $\xi:X \to B$ and $\omega$, and hence leads to stable $B$-diffeomorphism invariants. \end{theorem}

\begin{proof}[Proof sketch]
Recall that the character $\omega: \Omega^{T(\beta \circ \xi)} \to k^{\times}$ classifies an invertible field theory $\cW: \Bord^X_d \to \mathrm{sLine}_k$. By Lemma~\ref{lem:extend}, there is therefore a unique invertible extension $\widetilde{W}: \Bord_{d, d-1, d-2}^B \to \mathrm{2sLine}_k$.

Essentially, Theorem~\ref{thm:DWareSemisimple} then follows since any finite path integral theory $\FP_{X, \cW}: \Bord^B_d \to \sVec_k$ built from an invertible field theory $\cW: \Bord_d^X \to \sVec_k$ can be once extended to a 2-categorical finite path integral theory $\widetilde{\FP}_{X, \widetilde{\cW}}:\Bord_{d, d-1, d-2}^X \to \mathrm{2sVec}_k$ constructed from this unique extension $\widetilde{\cW}: \Bord_{d, d-1,d-2}^X \to \mathrm{2sVec}_k$.
The construction of $\widetilde{\FP}_{X, \widetilde{\cW}}$ may be outlined as follows. A version of this construction can also be found in~\cite[\S 8]{MR2648901}.

Let $\mathbb{S}pan(\cS^{\pi}, \mathrm{2sVec})$ denote the symmetric monoidal bicategory whose objects are $\pi$-finite spaces equipped with $2$-functors $ \cL: \tau_{\leq 2} X \to \mathrm{2sVec}$, whose $1$-morphisms are spans $X\ot[f] Y \to[f'] X'$ of $\pi$-finite spaces equipped with natural transformations $f^*\cL \To (f')^*\cL'$ of $2$-functors $\tau_{\leq 2} Y \to \mathrm{2sVec}_k$ and whose $2$-morphisms are appropriate equivalence classes of spans of spans
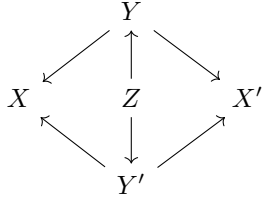
\[\begin{tikzcd}
& Y \arrow[rd]\arrow[ld] & \\
X&Z \arrow[u] \arrow[d] & X'\\
& Y' \arrow[lu]  \arrow[ru]&
\end{tikzcd}
\]
 equipped with modifications between the pulled back natural transformations. Composition of $1$-morphisms and 2-morphisms is given by (homotopy) pullback of spans with appropriate composition of the associated natural transformations and modifications. In the case where all spaces are $1$-groupoids (equivalently, $1$-types), a  detailled construction of this symmetric monoidal bicategory can be found in~\cite{MR3825014}, the general case is entirely analogous. (This symmetric monoidal bicategory $\mathbb{S}pan(\cS^{\pi}, \mathrm{2sVec}_k)$ is the homotopy bicategory of the symmetric monoidal $(\infty,2)$-category of iterated spans with local systems constructed in~\cite{MR3830256}.)

Analogously to Definition~\ref{def:GenDWTheory}, our extended finite path integral theory $\widetilde{\FP}_{X, \widetilde{\cW}}$ is defined as a composite of symmetric monoidal 2-functors $$\Bord_{d,d-1, d-2}^B  \to \mathbb{S}pan(\cS^{\pi}, \mathrm{2sVec}_k) \to \mathrm{2sVec}_k~.$$ Here, the first functor maps a closed $(d-2)$-$B$-manifold $W$ to the space of $X$-structures on $W$ refining the given $B$ structure, equipped with the $\mathrm{2sVec}_k$-valued local system determined by $\widetilde{W}: \Bord^B_{d,d-1,d-2} \to \mathrm{2sVec}_k$ and analogous assignments to higher-dimensional bordisms. 

The second functor $\mathbb{S}pan(\cS^{\pi}, \mathrm{2sVec}_k) \to \mathrm{2sVec}_k$ is a $2$-categorical linearization functor extending the 1-categorical linearization functor from Corollary~\ref{cor:spanfunctor}. In the case where all $\pi$-finite spaces considered are $1$-groupoids, and the target is $\mathrm{2Vec}$ instead of $\mathrm{2sVec}$, this linearization 2-functor is constructed in great detail by Schweigert and Woike in~\cite[Theorem 4.8]{MR3825014}. Replacing the target $\mathrm{2Vec}$ by $\mathrm{2sVec}$, finite $1$-groupoids by $\pi$-finite spaces, and all occurences of their `integral with respect to groupoid cardinality' (see~\cite[Equation~(2.3)]{MR3825014}) by our finite path integral from Definition~\ref{def:finitepathintegral} (which essentially amounts to inserting various homotopy cardinality factors to account for higher homotopies), the proof of~\cite[Theorem 4.8]{MR3825014} goes through almost verbatim to construct a symmetric monoidal $2$-functor $\mathbb{S}pan(\cS^{\pi}, \mathrm{2sVec}_k) \to \mathrm{2sVec}_k$.  (Note that Schweigert and Woike's construction~\cite{MR3825014} more evidently extends the `limit variant' of our linearization functor from Remark~\ref{rem:limitlinearization}. However, as explained in Remark~\ref{rem:limitlinearization}, since all spaces considered here are $\pi$-finite, the limit and colimit variants of the linearization functors are equivalent.)

By construction, this new bicategorical finite path integral theory extends our given field theory $\FP_{\xi, \omega}$. Analogously to~\cite[Thm.~2.10]{Reutter:2020aa}, any once-extendable field theory is automatically semisimple, proving the theorem. \end{proof}

\begin{remark}\label{rem:alternativessproof} We now outline an alternative proof of Theorem~\ref{thm:DWareSemisimple} which avoids bicategories and proceeds by explicitly computing the algebras associated by $\FP_{\xi,\omega}$ to products of spheres.

First, we claim that for any $d$, any $(d-2)$-truncated map $\beta:B\to \rB O(d)$ with connected $B$, any $\xi:X \to B$ with $\pi$-finite fibers, and any character $\omega:\Omega^{T\xi} \to k^\times$ (equivalently, invertible field theory $W: \Bord_d^X \to \mathrm{sLine}_k$), the sphere algebra  $\FP_{\xi, \omega}(S^{d-1})$ (with its unique $B$-structure, cf. Lemma~\ref{lem:sphereFrobAlg}) of the associated finite path integral theory $\FP_{\xi, \omega}: \Bord_d^B \to \sVec_k$ is semisimple. 
This essentially follows from an explicit computation: Let $F$ denote the fiber of $\xi:X \to B$ (at some choice of basepoint) and assume without loss of generality that $F$ is connected (else, $\FP_{\xi, \omega}(S^{d-1})$ decomposes into a direct sum of algebras, one for each component of $F$). Then, it can be shown that $\FP_{\xi, \omega}(S^{d-1})$ is either zero or may be identified, as an algebra, with the fixed point subalgebra $\left(k^{\alpha}[\pi_{d-1}F]\right)^{\pi_1(F)}$ of the $\pi_1(F)$ action on the twisted group algebra $k^{\alpha}[\pi_{d-1}(F)]$, twisted by a certain $2$-cocycle $\alpha \in \rH^2(\pi_{d-1}F, k^\times)$ which is determined by the invertible field theory $W$. In characteristic zero,  twisted group algebras and their fixed point subalgebras are semisimple~\cite[Thm.~9]{MR1503262}.

It remains to show that for finite path integral theories fulfilling the assumptions of Theorem~\ref{thm:DWareSemisimple}, the algebra associated to the product of middle dimensional spheres is semisimple. 
Indeed, for such a theory (i.e. one for which $d=2q$ and for which $\beta: B \to \rB O(2q)$ is $(q-1)$-truncated and $B$ is connected), the algebra $\FP_{\xi, \omega}(S^q\times S^{q-1})$ is the algebra associated to $S^q$ by the dimensionally reduced theory $\FP_{X, \omega} (S^{q-1} \times -): \Bord_{q+1}^{B_{q+1}} \to \mathrm{sVec}_k$ where $B_{q+1} \to \rB O(q+1)$ is the pull back of our original ambient structure $B\to \rB O(2q)$. By Theorem~\ref{thm:DWdimreductisDW}, this reduced theory is again a finite path integral theory. Since we assumed $B \to \rB O(2q)$ to be $(q-1)$-truncated, so is $B_{q+1} \to \rB O(q+1)$ and the preceeding paragraph implies that the algebra associated by $\FP_{X, \omega}(S^{q-1} \times-)$ to $S^q$ is semisimple.
\end{remark}

\section{Stable diffeomorphism and the Main Theorem}\label{sec:stablediffeomandmain}

\subsection{The relation to stable diffeomorphism}\label{sec:stablediffeoandKrecksthm}

In this section we summarize Kreck's classification of manifolds up to stable diffeomorphism \cite{MR1709301}, and explain the connection to finite path integral topological field theories. 

We consider manifolds of dimension $d=2q$, and we fix an ambient tangential structure $\beta:B \to \rB O(2q)$, which is required to be a connected $(q-1)$-tangential type (e.g. $\rB SO(2q)$, $\rB Spin(2q)$, etc.). 
By Proposition~\ref{prop:pullbackstructures}, if $q \geq 1$, there is a corresponding stable tangential structure $\beta_{\infty}:B_\infty \to \rB O$ and a pullback diagram
\begin{center}
\begin{tikzpicture}
	\node (A) at (0,1.5) {$ B $};
	\node (B) at (2.5, 1.5) {$ B_\infty $};
	\node (C) at (0, 0) {$ \rB O(2q) $};
	\node (D) at (2.5, 0) {$ \rB O $};
	\draw [->] (A) to (B); 
	\draw [->] (A) to (C); 
	\draw [->] (B) to (D); 
	\draw [->] (C) to (D); 	
	\node at (0.5, 1) {$\lrcorner$};	
\end{tikzpicture}
\end{center}
Similarly, there is an associated stable normal structure $(\overline{B}_{\infty}, \overline{\beta})$ defined as the composite $B_{\infty} \to[\beta] \rB O \to[-1] \rB O$.
Hence, for a $2q$-manifold $M$ there are natural equivalences between the space of $B$-structures on the tangent bundle $T_M$, the space of $B_\infty$-structures on the stable tangent bundle $\tau_M$ and the space of $\overline{B}_\infty$-structures on the stable normal bundle of $M$. We will intentionally conflate these three types of structure and simply refer manifolds with any of these structures as \emph{$B$-manifolds}.

\begin{definition}
	A $2q$-dimensional $B$-manifold $M$ has:
	\begin{itemize}
		\item a \emph{tangential $(q-1)$-type} $M \to \tau_{\leq n}^B M \to B$, defined by factoring the map to $B$ into a $(q-1)$-connected map followed by a $(q-1)$-truncated map; 
		\item a \emph{stable tangential $(q-1)$-type} $M \to \tau_{\leq n}^{B_\infty} M \to B_\infty$, defined by factoring the map to $B_\infty$ into a $(q-1)$-connected map followed by a $(q-1)$-truncated map; 
		\item a \emph{stable normal $(q-1)$-type} $M \to \tau_{\leq n}^{\overline{B}_\infty} M \to \overline{B}_\infty$, defined by factoring the map to $\overline{B}_\infty$ into a $(q-1)$-connected map followed by a $(q-1)$-truncated map; 
	\end{itemize}
\end{definition}
 
\begin{definition}
	Let $M$ be a $B$-manifold and and $X \to B$ a tangential $(q-1)$-type. An $X$-structure on $M$ is a \emph{tangential $(q-1)$-smoothing} if it induces an equivalence $X \simeq \tau_{\leq q-1}^B M$ of spaces over $B$. 
\end{definition} 
 
There is a corresponding  stable tangential $(q-1)$-type $X_\infty \to B_\infty$ and stable normal $(q-1)$-type $\overline{X}_\infty \to \overline{B}_\infty$. If an $X$-structure on $M$ is a tangential $(q-1)$-smoothing it also induces equivalences $X_\infty \simeq \tau_{\leq n}^{B_\infty} M$ and $\overline{X}_\infty \simeq \tau_{\leq n}^{\overline{B}_\infty} M$. Thus we could equally have called the $X$-structure a \emph{normal $(q-1)$-smoothing}, as is done in \cite{MR1709301}.

A $B$-manifold $M$ with a choice of tangential $(q-1)$-smoothing $M \to X \to B$, gives rise to an element in the semistable bordism group $\Omega_{2q}^{TX}$ (see Definition~\ref{def:tangentialbordism}), and thus also in the stable bordism group $\Omega_{2q}^{\tau X_\infty} \cong \Omega_{2q}^{\nu \overline{X}_\infty}$ (Section~\ref{sec:compntypes}). By Lemma~\ref{lem:TangetBordismIsstablebordismplusEuler} (see Remark~\ref{rmk:injection}) the map 
	\begin{equation*}
		(\chi, p): \Omega^{TX}_{2q} \to \ZZ \oplus \Omega^{\tau X_\infty}_{2q} \cong \ZZ \oplus \Omega^{\nu \overline{X}_\infty}_{2q}
	\end{equation*}
	is injective, where $p: \Omega^{TX}_{2q} \to \Omega^{\tau X_\infty}_{2q} \cong \Omega^{\nu \overline{X}_\infty}_{2q}$ is the natural projection (and in particular surjective) and $\chi$ is the Euler characteristic. Thus, a class in the semistable tangential $X$-bordism group can be thought of as the normal $X$-bordism class together with information about the Euler characteristic.

There are natural equivalences $\Aut_B(X) \simeq \Aut_{B_\infty}(X_\infty) \simeq \Aut_{\overline{B}_\infty}(\overline{X}_\infty)$, and 
the group $\pi_0 \Aut_B(X)$ of (homotopy classes) of automorphisms of $X$ over $B$ acts effectively and transitively on the equivalence classes of tangential $(q-1)$-smoothings on $M$. This action extends to the respective actions on the bordism groups $\Omega_{2q}^{TX}$, and $\Omega_{2q}^{\tau X_\infty} \cong \Omega_{2q}^{\nu \overline{X}_\infty}$. 

Our assumption that $B$ is a connected tangential $(q-1)$-type, implies (assuming $q \geq 1$) that it satisfies the assumptions from Section~\ref{sec:connectsumBFrob}, and in particular that there is a well defined notion of connected sum of $B$-manifolds. Moreover there exists a canonical (unique up to orientation reversal) $B$-structure on $S^q \times S^q$. By construction this agrees with the canonical (normal) $\overline{B}_\infty$-structure constructed by \cite{MR1709301}.

Recall (Definition~\ref{def:stablediffeo}) that two compact connected closed $2q$-dimensional $B$-manifolds,  $M$ and $N$, are $B$-stably diffeomorphic if there exists a natural number $r$ and a $B$-structured diffeomorphism between $M \#^r S^q \times S^q$ and $N\#^r S^q \times S^q$. Note well that the notion of stable diffeomorphism used here takes connected sum with the same number of copies of  $S^q \times S^q$ on both $M$ and $N$.

\begin{theorem}[{\cite{MR1709301}}]\label{thm:KrecksThm}
	Let $M$ and $N$ be compact connected closed $2q$-dimensional $B$-manifolds with the same stable normal $(q-1)$-type $\overline{X}_\infty \to \overline{B}_\infty$. Then the following are equivalent
	\begin{enumerate}
		\item $M$ and $N$ are $B$-stably diffeomorphic
		\item $M$ and $N$ have the same Euler characteristic and represent elements in the stable normal bordism group $\Omega_{2q}^{\nu \overline{X}_\infty}$ which lie in the same orbit under the action of $\pi_0 \Aut_B(X) \cong \pi_0 \Aut_{\overline{B}_\infty}(\overline{X}_\infty)$. \qed
	\end{enumerate}
\end{theorem}

In what follows, we will use a tangential unstable version of Kreck's theorem.
\begin{corollary}Let $M$ and $N$ be compact connected closed $2q$-dimensional $B$-manifolds. Then, the following are equivalent: 
	\begin{enumerate}
		\item $M$ and $N$ are $B$-stably diffeomorphic
		\item $M$ and $N$ have the same tangential $(q-1)$-type $X \to B$ and represent elements in the semistable tangential bordism group $\Omega_{2q}^{TX}$ which lie in the same orbit under the action of $\pi_0 \Aut_B(X)$. 
	\end{enumerate}
\end{corollary}
\begin{proof}As we have observed having equivalent tangential $(q-1)$-type is equivalent to having equivalent stable normal $(q-1)$-type. By Lemma~\ref{lem:TangetBordismIsstablebordismplusEuler} (see Remark~\ref{rmk:injection}) the map 
	\begin{equation*}
		(\chi, p): \Omega^{TX}_{2q} \to \ZZ \oplus \Omega^{\tau X_\infty}_{2q} \cong \ZZ \oplus \Omega^{\nu \overline{X}_\infty}_{2q}
	\end{equation*}
	is injective, where $p: \Omega^{TX}_{2q} \to \Omega^{\tau X_\infty}_{2q} \cong \Omega^{\nu \overline{X}_\infty}_{2q}$ is the natural projection and $\chi$ is the Euler characteristic. The action of $\pi_0\Aut_B(X)$ on $\Omega^{TX}_{2q}$ leaves  the Euler characteristic invariant. Hence, the Corollary is a direct consequence of Kreck's stable diffeomorphism classification Theorem~\ref{thm:KrecksThm}.
\end{proof}

Using this corollary, our main Theorem~\ref{thm:MainThm} may be re-interpreted as follows. 
\begin{corollary}\label{thm:StableDiffeo}
	Let $M$ and $N$ be two $2q$-dimensional $B$-manifolds. Suppose that $M$ and $N$ have $\pi$-finite $B$-tangential $(q-1)$-types. Then $M$ and $N$ are indistinguishable by type-$(q-1)$ finite path integral theories if and only if they are $B$-stably diffeomorphic. \qed 
\end{corollary}

\subsection{Examples and Applications}\label{sec:exmappl}

\subsubsection{Hitchin Exotic Spheres}

If $d > 1$, the tangential $1$-type of a homotopy sphere is given by $\rB Spin(d)$, and each homotopy sphere admits a unique spin structure, yielding a homomorphism $\Theta_d \to \Omega^{TSpin(d)}_d$ from the group $\Theta_d$ of h-cobordism classes of smooth homotopy $d$-spheres (equivalently, for $d>4$, the group of diffeomorphism classes of homotopy spheres). Since all homotopy spheres of the same dimension have the same Euler characteristic, the image injects into $\Omega_d^{Spin}$, and we may as well consider the induced homomorphism $\eta: \Theta_d\to \Omega_d^{Spin}$. Each class $[\Sigma] = \eta(\Sigma)$ represented by a homotopy sphere is invariant under the action of $\Aut_B(\rB Spin(d))$ where $B = \rB SO(d)$ or $B = \rB O(d)$. It follows directly from Theorem~\ref{thm:MainThm} that: 

\begin{corollary}\label{cor:firstsphere}
	Two homotopy spheres $\Sigma$ and $\Sigma'$ are indistinguishable by type-$1$ finite path integral theories if and only if $\eta(\Sigma) = \eta(\Sigma')$, i.e. if and only if they represent the same element in Spin bordism. \qed
\end{corollary}	

The $\alpha$-invariant is a map $\alpha: \Omega^{Spin}_d \to KO_d(pt)$ from Spin bordism to the real K-theory of a point. 
Hitchin \cite{MR358873} (based on results in \cite{MR156292}) proved that the $\alpha$ invariant of a closed spin manifold is an obstruction to the existence of a metric of positive scalar curvature on it. On the other hand, Stolz \cite{MR1056561} proved that a simply connected spin manifold of dimension $d \geq 5$ admits a metric of positive scalar curvature if its $\alpha$-invariant vanishes. The proof uses surgery results obtained (independently) in \cite{MR577131} and \cite{MR535700}, as well as involved calculations within stable homotopy. Hence, a $d$-dimensional homotopy sphere $\Sigma$ with $d\geq 5$ admits a metric of positive scalar curvature if and only if $\alpha(M)$ is trivial. Homotopy spheres for which $\alpha(\Sigma)$ is non-trivial, or which equivalently do not admit a metric of positive scalar curvature, are called \emph{Hitchin spheres}. In~\cite{MR189043}, Anderson-Brown-Peterson showed that Hitchin spheres exist precisely in dimension $d=8k+1$ and $8k+2$ and in fact that for any homotopy sphere $\Sigma$, the eta invariant $\eta(\Sigma) \neq 0$ if and only if $\alpha \circ \eta(\Sigma) \neq 0$.
Hence, the following is an immediate corollary of Corollary~\ref{cor:firstsphere}.

\begin{corollary}\label{cor:hitchinspheredetection}
	There exist oriented and unoriented semisimple topological field theories which distinguish Hitchin spheres from non-Hitchin homotopy spheres. 
	\qed
\end{corollary}

\begin{remark}\label{rmk:explicitTFTexoticsphere}
		In dimensions $d = 8k+1$ or $8k+2$ the alpha invariant $\alpha: \Omega_d^{Spin} \to \ZZ/2\ZZ$ takes values in $\ZZ/2\ZZ$ and can be viewed as a character by identifying $\ZZ/2\ZZ \cong \{ \pm 1 \} \subseteq \CC^*$. Thus this gives rise to an invertible spin topological field theory. The finite path integral construction applied to this invertible theory produces an oriented theory which takes value $-\frac{1}{2}$ on Hitchin spheres and value $\frac{1}{2}$ on non-Hitchin spheres.
\end{remark}

\subsubsection{Examples in dimension 4}

In dimension $2q=4$, Theorem~\ref{thm:StableDiffeo} says that two $4$-manifolds with $\pi$-finite tangential $1$-types are stably diffeomorphic if and only if they are indistinguishable by type-1 finite path integral theories. A connected manifold has $\pi$-finite tangential $1$-type precisely if it has a finite fundamental group. In \cite{TeichnerPHD} Teichner constructed the following examples of homotopy equivalent smooth manifolds 

\begin{theorem}[{\cite{TeichnerPHD}}]\label{thm:TeichnerPHD}
	Let $\pi$ be a finite group with generalized quaternion 2-Sylow subgroup of order larger than eight. Then there exist connected oriented smooth 4-manifolds $M_0$ and $M_1$ whose fundamental groups $\pi_1 M_0 \cong \pi_1 M_1 \cong \pi$, such that $M_0$ is (orientation preserving) homotopy equivalent to $M_1$, but $M_0$ and $M_1$ are not stably diffeomorphic. These manifolds have the following properties
	\begin{enumerate}
		\item $M_0$ and $M_1$ have the same tangential 1-type $\xi: X \to BO(4)$. 
		\item $M_i$ is oriented, and the universal cover of $M_i$ is spin. 
		\item The manifolds $M_i$ are not spin; the second Stiefel Whitney class is a particular specified class induced from a particular class in the cohomology of the generalized quaternion 2-Sylow subgroup.
		\item $M_0$ and $M_1$ are distinguished by a certain `tertiary' bordism invariant $ter: \Omega_4^{TX} \to \ZZ/2\ZZ$. 
	\end{enumerate} 	
\end{theorem}
This leads to the following corollary of Theorem~\ref{thm:StableDiffeo}.
\begin{corollary}
	In each instance of Theorem~\ref{thm:TeichnerPHD}, there are oriented and unoriented type-1 finite path integral theories which distinguish the manifolds $M_0$ and $M_1$. \qed
\end{corollary}

In fact an explicit finite path integral theory can be constructed using the bordism invariant $ter$, much as in Remark~\ref{rmk:explicitTFTexoticsphere}. We also recall that Gompf \cite{MR769285} has shown that smooth manifolds which are orientation preserving homeomorphic
are stably diffeomorphic. So while the manifolds $M_0$ and $M_1$ are homotopy equivalent, they are not homeomorphic. 

In the case of non-orientable 4-manifolds, there are examples of smooth manifolds which are homeomorphic but are not stably diffeomorphic. For example, the tangential $1$-type of $\RR\PP^4$ is $\rB Pin^+(4) \to \rB O(4)$ (the stable normal $1$-type is $\rB Pin^-(4) \to \rB O$). The action of $\pi_0\Aut_{\rB O}(\rB Pin^+) \cong \ZZ/2\ZZ$ on $\Omega^{\tau \rB Pin^+}_4$ is given by inversion. The smooth $\RR\PP^4$, with its two $Pin^+$-structures, represents the classes $\{\pm 1\} \subseteq \ZZ/16\ZZ$, while the Cappell-Shaneson exotic $\RR\PP^4_{C.S}$ \cite{MR418125}, with its two $Pin^+$-structures, represents the classes $\{\pm 9 \}\subseteq \ZZ/16\ZZ$, see \cite{MR958593}. This gives an example of homeomorphic manifolds which have the same tangential $1$-type but lie in distinct orbits in $\Omega^{\tau \rB Pin^+}_4$, and which are thus not stably diffeomorphic. More generally, we can take a (possibly mixed) $S^1$-connected sum (an operation which preserves tangential $1$-type, see~\cite{MR1234481})of $\RR\PP^4$ and Cappell-Shaneson exotic $\RR\PP^4_{C.S}$ which remain homeomorphic. Hence, we obtain the following corollary of Theorem~\ref{thm:StableDiffeo} generalizing Debray's result~\cite{362517} mentioned in the introduction.

\begin{corollary}
When a (possibly mixed) $S^1$-connected sum~\cite{MR1234481} of $\RR\PP^4$ and Cappell-Shaneson exotic $\RR\PP^4_{C.S}$ yield homeomorphic, but not stably diffeomorphic 4- manifolds, they are distinguished by type-1 unoriented finite path integral theories.

This happens for example with the $S^1$-connected sums $\#_{S^1}^r \RR\PP^4$ and $\#_{S^1}^r \RR\PP^4_{C.S}$ for $r\neq 0 ~\mathrm{mod }~4$.\qed
\end{corollary}
When $r=4$, these $\#_{S^1}^4 \RR\PP^4$ and $\#_{S^1}^4 \RR\PP^4_{\textrm{C. S.}}$ lie in the same orbit in $\Omega^{\tau \rB Pin^+}_4 $, and thus are stably diffeomorphic. This raises the following question:

\begin{question}
	Is the $S^1$-connected sum of four copies of real projective space $\RR\PP^4$ diffeomorphic to the $S^1$-connected sum of four copies of Cappell-Shaneson's exotic $\RR\PP^4$?
\end{question}

Similar examples can be obtained for 4-manifolds whose fundamental group is a finite of order $2$ mod $4$, see \cite{Debray:2021tb}.

\subsubsection{Manifolds whose fundamental group is the Thompson group}

The next example shows that we cannot completely ignore the $\pi$-finite condition in Theorem~\ref{thm:StableDiffeo}.

Let $\pi$ be a finitely presented group. Let $K_\pi$ be a connected finite two complex with $\pi_1K \cong \pi$. We can embed $K$ into $\RR^5$ and take a regular neighborhood $W \subset \RR^5$, which deformation retracts onto $K$. The boundary of $W$ is a smooth 4-manifold $M = \partial W$ with $\pi_1 M \cong \pi$. By construction $M$ bounds a framed 5-manifold (the regular neighborhood of $K_\pi$). Whence the signature of $M$ is zero, the tangent bundle of $M_\pi$ is stably trivial, $M$ is oriented, and   spin. The corresponding tangential 1-type is $\rB Spin(4) \times \rB \pi \to \rB O(4)$. 

Suppose now that $\pi$ is also infinite and simple. The for any $\pi$-finite tangential 1-type $(X, \xi)$, the space of $(X, \xi)$-structures on $M$ may be computed as
\begin{equation*}
	\Map_{\rB O(4)}(M, X) \simeq \Map_{\rB O(4)}( \rB Spin(4) \times \rB \pi, X) \simeq \Map_{\rB O(4)}( \rB Spin(4), X)
\end{equation*}
where the last equivalence is induced from the projection map $\rB Spin(4) \times \rB \pi \to \rB Spin(4)$. Thus the type-1 finite path integral invariants of $M$ are completely determined by the corresponding tangential spin bordism class $[M] \in \Omega^{T \rB Spin(4)}_4 \cong \ZZ \oplus \ZZ$. By Remark~\ref{rmk:injection}, this bordism class is determined by the Euler characteristic of $M$ (and the signature which we have already seen is zero). 

If further $\pi$ is acyclic, for example if $\pi = V$ is the large Thompson group, then Lefschetz and Alexander duality imply that $\chi(M) = 2 = \chi(S^4)$. It follows that $[M] = [S^4] \in \Omega^{T \rB Spin(4)}_4$. Thus we obtain:

\begin{proposition}\label{prop:Thompson}
	When $\pi = V$ is the large Thompson group, or any other infinite finitely presented acyclic group, then the 4-manifold manifold $M$ constructed above and the 4-sphere $S^4$ are not stably diffeomorphic, yet are indistinguishable by type-$1$ finite path integral theories.\qed
\end{proposition}

We note the obvious fact that the 4-sphere $S^4$ and $M$ are not even homotopy equivalent. Thus in some cases, with infinite fundamental group, type-$1$ finite path integral theories can fail to detect homotopy type of smooth 4-manifold.

\appendix

\section{$n$-finitely dominated spaces}\label{sec:nfindom}
We recall Definition~\ref{def:nfindominatedmain}:
\begin{definition}\label{def:nfindominated}
	A space 
	$X$ is \emph{$n$-finitely dominated} if there exists an $n$-dimensional finite CW complex $K$ and an
	 $(n-1)$-connected map $K \to X$. 
\end{definition}

Recall that a map being  $(n-1)$-connected means that the induced map $\pi_n(K) \to \pi_n(X)$ is surjective and $\pi_k(K) \cong \pi_k(X)$ for all $k<n$. This definition is relevant to us because of the following lemma whose proof is a standard exercise in obstruction theory. 

\begin{lemma}\label{lem:obstruction}
 If $\psi:Y \to B$ is $n$-truncated, then $\Map_B(X,Y)$ is an $n$-type for all $\xi:X \to B$. If $\psi:Y \to B$ moreover has $\pi$-finite fibers and $X$ is $n$-finitely dominated, then $\Map_B(X, Y)$ is $\pi$-finite.  
\end{lemma}

The following lemma collects sufficient conditions for a space to be $n$-finitely dominated.

  \begin{lemma}\label{lem:findom}
  A space $X$ satisfying any of the following conditions is $n$-finitely dominated: 
  \begin{enumerate}
  \item $X$ is homotopy equivalent to a CW complex with finite $n$-skeleton.
  \item $X$ is a homotopy pushout of $n$-finitely dominated spaces.
  \item There is a $n$-finitely dominated space $B$ and a map $\xi: X \to B$ all of whose homotopy fibers are $n$-finitely dominated.
  \item $X$ is simply connected and $H_i(X)$ is finitely generated for each $i \leq n$.
  \item $X$ is simply connected and the based loop space $\Omega X$ is $n$-finitely dominated.
    \item The $n$-type of $X$ is $\pi$-finite (i.e. $X$ has finitely many components and the first $n$ homotopy groups of every component are finite). 
  \end{enumerate}
  \end{lemma}
  \begin{proof}For (1), the inclusion $\mathrm{sk}_n X \to X$ of the $n$-skeleton is $(n-1)$-connected and hence witnesses $X$ being $n$-finitely dominated. 
To prove (2), given a diagram $U_1 \leftarrow U_3 \to U_2$ of $n$-finitely dominated spaces, let $K_1, K_2, K_3$ be finite $n$-dimensional CW complexes with $(n-1)$-connected maps $K_i \to U_i$. Since $K_3$ is $n$-dimensional and the maps $K_i \to U_i$ for $i=1,2$ are $(n-1)$-connected, it follows from obstruction theory that there are maps, as indicated by the dashed arrows, making the squares commute up to homotopy: 
\begin{center}
\begin{tikzpicture}
	\node (A1) at (0,1.5) {$ K_1 $};
	\node (B1) at (0, 0) {$ U_1 $};
	
	\node (A2) at (2.5, 1.5) {$ K_3 $};
	\node (B2) at (2.5, 0) {$ U_1 \cap U_2 $};
	
	\node (A3) at (5, 1.5) {$ K_2 $};
	\node (B3) at (5, 0) {$ U_2 $};
	
	\draw [->] (A1) to (B1);
	 \draw [->] (A2) to (B2);
	 \draw [->] (A3) to (B3);
	\draw [->, dashed] (A2) to (A1); 
	\draw [->] (B2) to (B1); 
	\draw [->, dashed] (A2) to (A3); 
	\draw [->] (B2) to (B3); 

\end{tikzpicture}
\end{center}
Taking homotopy pushouts, this defines a map from the homotopy pushout $K= K_1 \cup_{K_3} K_2$ to the homotopy pushout $X = U_1 \cup_{U_3} U_2$. 
  Since the pushout of finite $n$-dimensional CW complexes is again finite and $n$-dimensional, and since the pushout of $(n-1)$-connected maps is again $(n-1)$-connected~\cite[Theorem 6.7.9]{MR2456045}, it follows that $K \to X$ witnesses $X$ being $n$-finitely dominated. 
  
  For (3), let $K_0 \to B$ be an $(n-1)$-connected map from a finite $n$-dimensional CW complex. Then the map $X\times_B K_0 \to X$ pulled back from $K_0 \to B$ is also $(n-1)$-connected. Thus, it suffices to show that $X\times_B K_0$ is $n$-finitely dominated --- in other words, without loss of generality,  we can assume that $B = K$ is a finite CW complex. Now we can induct on the cell structure of $B$. When $B$ is a finite union of $0$-cells, then $X$ is a finite disjoint union of $n$-finitely dominated spaces and hence is itself $n$-finitely dominated. Thus we may assume that $B = B_0 \cup_{\partial D^k} D^k$ and that $X|_{B_0}$ is $n$-finitely dominated. Consider the commuting cube
\[\begin{tikzcd}
	& Z&& {F} \\
	{X|_{B_0}} && X \\
	& {\partial D^k} && {D^k} \\
	{B_0} && B
	\arrow[from=2-1, to=2-3]
	\arrow[from=3-2, to=3-4]
	\arrow[from=4-1, to=4-3]
	\arrow[from=2-3, to=4-3]
	\arrow[from=2-1, to=4-1]
	\arrow[from=3-2, to=4-1]
	\arrow[from=3-4, to=4-3]
	\arrow[from=1-4, to=3-4]
	\arrow[from=1-4, to=2-3]
	\arrow[from=1-2, to=2-1]
	\arrow[from=1-2, to=1-4]
	\arrow[from=1-2, to=3-2]
\end{tikzcd}\]
where $F$ is the (homotopy) pullback of $X \to B \leftarrow D^k$ (and hence homotopy equivalent to the homotopy fiber of $X\to B$ at any point in the image of $D^k \to B$) and $Z$ is the (homotopy) pullback of $\partial D^k \to D^k \leftarrow F$ (and hence homotopy equivalent to the product $F \times \partial D^k$). By definition, all four side squares are (homotopy) pullback squares and the lower square is a homotopy pushout. Hence, it follows that the upper square is a homotopy pushout square. 
Notice that since $\partial D^k = S^{k-1}$ is a compact manifold, it is $n$-finitely dominated, and since by assumption $F$ and $X|_{B_0}$ are $n$-finitely dominated, so is also the product $Z\simeq F \times \partial D^k$ and hence it follows from (2) that the pushout $X$ of $X|_{B_0} \leftarrow Z \to F$ is also $n$-finitely dominated.
 
Assertion (4) is a direct corollary of Wall~\cite[Thm.~B]{MR171284}. Assertion (5) follows from (4) and a Serre spectral sequence argument we learned from \cite{Sadofsky-notes}. Consider the based path-loop fibration $\Omega X \to P_*X \to X$. Since $\Omega X$ is $n$-finitely the homology groups $H_k(\Omega X)$ are finitely generated for $k\leq n$. If $i$ be the first dimension in which $\Omega X$ has non-zero reduced homology, then
\begin{equation*}
	H_{i+1}(X) \cong \pi_{i+1}(X) \cong \pi_i(\Omega X) \cong H_i(\Omega X).
\end{equation*}
Thus $H_{i+1}(X)$ is finitely generated and all lower reduced homology groups are trivial. Now suppose that $H_k(X)$ is finitely generated for $k\leq ell$, and consider the $E^2$-term of the Serre Spectral Sequence for the path-loop fibration. We want to show that $H_{\ell+1}(X)$ is finitely generated. The relevant $E^\infty$ groups are zero, and since $E^\infty_{n+1,0} = E^{n+2}_{n+1,0}$, we have $d^{n+1}_{n+1,0}$ is an isomorphism. The target is a quotient of $H_n(\Omega X)$, hence finitely generated. So $E^{n+2}_{n+1,0}$ is finitely generated. Now if we assume by induction that $E^k_{n+1,0}$ is finitely generated for $2 < k \leq n+2$, we examine the differential 
\begin{equation*}
	E^{k-1}_{n+1,0} \to E^{k-1}_{n+2-k,k-2}
\end{equation*}
$E^k_{n+1,0}$ is the kernel of that map and the image is finitely generated as it is a subgroup of a finitely generated group. So we have a short exact sequence
\begin{equation*}
	0 \to E^k_{n+1,0} \to E^{k-1}_{n+1,0} \to im(d^{k-1}_{n+1,0}) \to 0
\end{equation*}
Since the two groups on the end are finitely generated, so is the group in the middle. We conclude that $E^2_{n+1,0} = H_{n+1}(H)$ is finitely generated. 
  
To prove (6), first observe that for finite $\pi$ the space $K(\pi, 1)$ is homotopy equivalent to a CW complex with finite $n$-skeleton (for any $n$) and hence is $n$-finitely dominated. This can for example be seen by presenting $K(\pi, 1)$ as the realization of a simplicial bar complex built from the 
finite group $\pi$. Then, it follows by iteratively applying (5) and observing that $\Omega K(\pi,k) \simeq K(\pi,k-1)$ that $K(\pi,k)$ is $n$-finitely dominated. Hence, using (3) and inducting on the homotopy groups, it follows that every $\pi$-finite space is $n$-finitely dominated. Lastly, applying (3) to the map $X\to \tau_{\leq n}X$ to the $n$-type of $X$, it follows that it suffices that the $n$-type of $X$ is $\pi$-finite.
   \end{proof}

   \begin{remark} By work of Wall~\cite[Thm.~A]{MR171284}, the first condition of Lemma~\ref{lem:findom} is equivalent to the seemingly weaker condition that $X$ is a homotopy retract of a CW complex with a finite $n$-skeleton.
   \end{remark}
      
   \begin{example} Every compact manifold (of arbitrary dimension) is presented by a finite CW complex and hence is $n$-finitely dominated for every $n \geq 0$.
   \end{example}

\begin{example}\label{ex:EMspacesfindom}
	The circle $S^1 \simeq K(\ZZ,1)$ is presented by a finite CW complex and hence is $n$-finitely dominated for every $n \geq 0$. It then follows that $K(\ZZ, k)$ is   $n$-finitely dominated by iteratively applying Lemma~\ref{lem:findom}(5) and observing that $\Omega K(\ZZ,k) \simeq K(\ZZ,k-1)$.
\end{example}
   
\begin{example}\label{ex:BOfindom} The space $\rB O$ and for any $d \geq 0$ the spaces $\rB O(d)$ are $n$-finitely dominated for every $n\geq 0$. Moreover, any connective cover $\rB O(d) \langle k \rangle \to \rB O(d)$ is $n$-finitely dominated. This can be seen as follows. The space  $\rB O(d)$ is the Grassmannian of $d$-planes in infinite dimensional space. The map from the finite Grassmannian $Gr_d(n+d) \to \rB O(d)$ is $n-1$-connected, and since $Gr_d(n+d)$ is a compact manifold, it follows that $\rB O(d)$ is $n$-finitely dominated for all $n$. Likewise the map $\rB O(n) \to \rB O$ is $n-1$-connected, and thus $\rB O$ is $n$-finitely dominated for all $n$. The connective cover $\rB O(d) \langle k \rangle \to \rB O(d)$ sits in a fibration sequence
	\begin{equation*}
		K(\pi_k(\rB O(d)), k-1) \to \rB O(d)\langle k \rangle \to \rB O(d)\langle k-1 \rangle.
	\end{equation*}
Thus iteratively applying Lemma~\ref{lem:findom}(3) and Example~\ref{ex:EMspacesfindom} shows that the connective covers too are $n$-finitely dominated. 
\end{example}
  
For our purposes, a central corollary of Lemma~\ref{lem:findom} is the following:
  \begin{corollary}\label{cor:fiberfinite} Let $n\geq 0$ and assume that $B$ is $n$-finitely dominated. 
   If $\xi:X \to B$ is a map of spaces with $\pi$-finite fibers, then $X$ is $n$-finitely dominated. 
  \end{corollary}
  \begin{proof}
  By Lemma~\ref{lem:findom} (4), the fibers of $\xi$ are $n$-finitely dominated, and hence $X$ is $n$-finitely dominated by Lemma~\ref{lem:findom}(3).   \end{proof}

Let $\tau_{\leq n} \cS_{/B}^\pi$ denote the full subcategory of $\tau_{\leq n} \cS_{/B}$ on those $n$-truncated $\xi: X \to B$ whose fibers are $\pi$-finite, and let $\tau_{\leq n} \cS_{/B}^{\pi~\textrm{f.d.}}$ denote the further full subcategory on those $n$-truncated $\xi:X \to B$ with $\pi$-finite fibers for which furthermore $X$ is $n$-finitely dominated. 

\begin{corollary}\label{cor:pifinislocpifin} The $\infty$-category $\tau_{\leq n} \cS_{/B}^{\pi~\textrm{f.d.}}$ is locally $\pi$-finite. Moreover, if $B$ is $n$-finitely dominated, then $\tau_{\leq n}\cS_{/B}^{\pi~\textrm{f.d.}} = \tau_{\leq n} \cS_{/B}^\pi$. \end{corollary}
\begin{proof} This is a direct consequence of Lemma~\ref{lem:obstruction} and Corollary~\ref{cor:fiberfinite}.
\end{proof}

\bibliographystyle{initalpha}
\bibliography{DWStable}

\end{document}

%% file: DWStable-preamble.tex
\usepackage{amsmath}       
\usepackage{amsthm}        
\usepackage{amssymb}       
\usepackage{amsfonts}
\usepackage[all]{xy}
\usepackage{xspace}
\usepackage{calc}
\usepackage{pgfplots}
\usepgflibrary{fpu}
\usepackage{mathtools}
\usepackage{tabularx}
\usepackage{ifthen}		
\usepackage{graphicx}
\usepackage{docmute}
\usepackage{array}
\usepackage{subfloat} 
\usepackage{bbm}            
\usepackage{etoolbox}  
\usepackage{verbatim}
\usepackage{stmaryrd}
\usepackage{thm-restate}
\usepackage{multicol}
\usepackage{nccmath}

\usepackage{bbm} 

\usepackage[margin=1.5in]{geometry}
\usepackage{tikz}
\usetikzlibrary{matrix}
\usetikzlibrary{decorations.pathreplacing}
\usetikzlibrary{decorations.markings}
\usetikzlibrary{arrows}
\usetikzlibrary{calc}
\usetikzlibrary{cd}
\usetikzlibrary{knots}
\usetikzlibrary{shapes.misc}
\usetikzlibrary{fit}
\usepgflibrary{decorations.pathmorphing}
\usepgflibrary{shapes.geometric}
\newenvironment{tz}[1][]{%
			\begin{tikzpicture}[baseline={([yshift=-.8ex]current bounding 					box.center)},#1] %
				}{%
			\end{tikzpicture} %
			}
			
\tikzset{dot/.style={circle, scale=0.3, fill=black, thick, draw}}
\tikzset{knot/.style={draw=white, double distance=1, line width=0.5, double=black}}

\renewcommand{\to}[1][]{\ensuremath{\xrightarrow{#1}}}
\newcommand{\ot}[1][]{\ensuremath{\xleftarrow{#1}}}
\newcommand{\To}[1][]{\ensuremath{\xRightarrow{#1}}}

\allowdisplaybreaks[1]

\theoremstyle{plain} 
\newtheorem{maintheorem}{Theorem}

\newtheorem{theorem}{Theorem}[section]
\newtheorem{lemma}[theorem]{Lemma}
\newtheorem{corollary}[theorem]{Corollary}          
\newtheorem{proposition}[theorem]{Proposition}              
\newtheorem{cor}[theorem]{Corollary}          
\newtheorem{prop}[theorem]{Proposition}

\theoremstyle{definition} 
\newtheorem{definition}[theorem]{Definition}
\newtheorem{assumption}[theorem]{Assumption}

\theoremstyle{remark}  

\newtheorem{remark}[theorem]{Remark}
\newtheorem{example}[theorem]{Example}

\newtheorem{question}[theorem]{{Question}}
\newtheorem*{guide*}{Guide}
\newtheorem*{outline*}{Outline}
\newtheorem*{remarkohc*}{Remark on higher categories and the cobordism hypothesis}
\newtheorem*{assumpfield*}{Assumptions on the base field}

\DeclareMathOperator{\Aut}{Aut}
\DeclareMathOperator{\Hom}{Hom}
\DeclareMathOperator{\End}{End}

\DeclareMathOperator{\Ext}{Ext}

\DeclareMathOperator*{\colim}{colim}

\DeclareMathOperator{\Fun}{Fun}
\newcommand{\id}{\mathrm{id}}
\newcommand{\iso}{\cong}

\newcommand{\Cob}{\mathrm{Cob}}

\newcommand{\Bord}{\mathrm{Bord}}

\newcommand{\Map}{\mathrm{Map}}

\renewcommand{\Vec}{\mathrm{Vec}}
\newcommand{\Vect}{\mathrm{Vect}}
\newcommand{\Line}{\mathrm{Line}}

\newcommand{\op}{\mathrm{op}}

\newcommand{\sVect}{\mathrm{sVect}}
\newcommand{\sVec}{\mathrm{sVec}}
\newcommand{\sLine}{\mathrm{sLine}}

\newcommand{\FP}{\mathrm{FP}}

\newcommand{\Set}{\mathrm{Set}}

\newcommand{\Fin}{\mathrm{Fin}}
\newcommand{\Surj}{\mathrm{Surj}}
\newcommand{\Inj}{\mathrm{Inj}}

\newcommand{\red}{\mathrm{red}}

\DeclareMathOperator{\Pic}{Pic}

\def\cA{\mathcal A}\def\cB{\mathcal B}\def\cC{\mathcal C}
\def\cF{\mathcal F}
\def\cL{\mathcal L}
\def\cN{\mathcal N}\def\cP{\mathcal P}
\def\cR{\mathcal R}\def\cS{\mathcal S}\def\cT{\mathcal T}
\def\cV{\mathcal V}\def\cW{\mathcal W}
\def\cZ{\mathcal Z}

\def\CC{\mathbb C}
\def\GG{\mathbb G}

\def\PP{\mathbb P}
\def\QQ{\mathbb Q}\def\RR{\mathbb R}

\def\ZZ{\mathbb Z}

\def\sA{\mathsf A}\def\sB{\mathsf B}

\newcommand\rB{\mathrm B}

\newcommand\rH{\mathrm H}

\usepackage[pdftex, colorlinks=true, linkcolor=black, citecolor=black, urlcolor=black, hypertexnames=false,
pdfauthor={\authorpaper},
pdftitle={\titlepaper},
bookmarksdepth=3
	]{hyperref}

\definecolor{DRcolor}{rgb}{0.0,0.5,0.75}	
\definecolor{CSPcolor}{rgb}{0.8,0.0,0.2}

\newcommand{\DRs}[1]{}
\newcommand{\CSPs}[1]{}